\documentclass{amsart}[12 pt]

\usepackage{latexsym}

\usepackage{amsmath,amssymb, amsfonts}
\usepackage{verbatim}
\usepackage{graphicx}
\usepackage{graphics}
\usepackage{geometry}
\usepackage{booktabs}
\usepackage{enumitem}

\newtheorem{proposition}{Proposition}[section]
\newtheorem{theorem}{Theorem}[section]
\newtheorem{definition}{Definition}[section]
\newtheorem{corollary}{Corollary}[section]
\newtheorem{lemma}{Lemma}[section]
\newtheorem{remark}{Remark}[section]

\newtheorem{assumption}{Assumption}[section]

\newcommand{\<}{\langle}
\renewcommand{\>}{\rangle}

\numberwithin{equation}{section}

\allowdisplaybreaks

\begin{document}
\markboth{}{}

\title[Principal Component Analysis for Semimartingales and Stochastic PDE] {Principal Component Analysis for Semimartingales and Stochastic PDE}

\author{Alberto Ohashi}

\address{Departamento de Matem\'atica, Universidade Federal da Para\'iba, 13560-970, Jo\~ao Pessoa - Para\'iba, Brazil}\email{alberto.ohashi@pq.cnpq.br; ohashi@mat.ufpb.br}

\author{Alexandre B. Simas}

\address{Departamento de Matem\'atica, Universidade Federal da Para\'iba, 13560-970, Jo\~ao Pessoa - Para\'iba, Brazil}\email{alexandre@mat.ufpb.br}

\thanks{The first author would like to thank the Mathematics
department of ETH Zurich and Forschungsinstitut f$\ddot{u}$r Mathematik (FIM) for the very kind hospitality during the first year of this research project. In particular, he would like to thank Josef Teichmann  for inspirational discussions about finite dimensional realizations of stochastic PDEs and for his encouragement to study their statistical aspects. We also would like to thank M. Laurini for useful discussions on PCA
}
\date{\today}

 \subjclass{Primary:   ; Secondary: }

\begin{center}
\end{center}

\begin{abstract}
In this work, we develop a novel principal component analysis (PCA) for semimartingales by introducing a suitable spectral analysis for the quadratic variation operator. Motivated by high-dimensional complex systems typically found in interest rate markets, we investigate correlation in high-dimensional high-frequency data generated by continuous semimartingales. In contrast to the traditional PCA methodology, the directions of large variations are not deterministic, but rather they are bounded variation adapted processes which maximize quadratic variation almost surely. This allows us to reduce dimensionality from high-dimensional semimartingale systems in terms of quadratic covariation rather than the usual covariance concept.

The proposed methodology allows us to investigate space-time data driven by multi-dimensional latent semimartingale state processes. The theory is applied to discretely-observed stochastic PDEs which admit finite-dimensional realizations. In particular, we provide consistent estimators for finite-dimensional invariant manifolds for Heath-Jarrow-Morton models. More importantly, components of the invariant manifold associated to volatility and drift dynamics are consistently estimated and identified. The proposed methodology is illustrated with both simulated and real data sets.
\end{abstract}
\maketitle
\tableofcontents

\section{Introduction}
Dimension reduction techniques have been intensively studied over the last years due to the advent of high-dimensional data in a variety of applied fields. Towards an effective reduction dimension, it is crucial to interpret correctly what kind of lower dimensional manifold one has to find in order to represent the data properly. For instance, if the second moment structure reasonable describes the dynamics in the data, then the classical Principal Component Analysis (henceforth abbreviated by PCA) and its various extensions are the natural candidates to reduce dimensionality.

There are many cases where correlation in high-dimensional systems may not be accurately described by covariance structures. An important example is the correlation typically found in high-frequency data which is better described by the so-called quadratic variation matrix

$$
[M]_t: = [M^i,M^j]_t; 1\le i,j\le d;0\le t\le T,
$$
where $M=(M^1, \ldots, M^d)$ is a $d$-dimensional semimartingale sampled over the time horizon $[0,T]$ and $[M^i,M^j]$ is the quadratic covariation process between $M^i$ and $M^j$.


In a financial context, the process $[M]_\cdot$ is called the volatility matrix (sometimes called integrated volatility). The total amount of volatility in a $d$-dimensional semimartingale system over $[0,t]$ is fully described by the following quantity

$$
\|[M]_t\|^2_{(2)} = \sum_{j=1}^{d}(\lambda^j_t)^2,
$$
where $\{\lambda^j_t\}_{j=1}^d$ are the random eigenvalues of $[M]_t$ and $\|\cdot\|_{(2)}$ is the usual Hilbert-Schmidt norm. Volatility is by far the most important quantity which needs to be estimated for asset pricing, asset allocation and risk management, specially in high-dimensional portfolios. The estimation of high-dimensional quadratic variation matrices has been a topic of great interest in the last years. We refer the reader to the works~\cite{bandorff, wang, zeng, mancino, tao, bibinger, fan, podolskij} and other references therein.


Despite all the recent progress on volatility matrix estimation, there has been remarkably little fundamental theoretical study on dimension reduction techniques based on high-dimensional quadratic variation matrices. One notorious difficulty is the dynamic interpretation of directions and principal components over the time horizon which in typical cases is formulated in a high-frequency domain. Indeed, $\{[M]_t; 0\le t\le T\}$ is fully random which makes the analysis more evolved than the standard PCA. More precisely, all the potentially optimal projections will be stochastic processes rather than deterministic vectors.


In view of the fact many correlation structures in high-dimensional data are fully represented by the quadratic variation concept, it is natural and necessary to construct a dimension reduction methodology strictly associated to $[M]$ rather than on classical covariance or conditional distributions. This is the program we start to carry out in this paper.

\subsection{Contributions}
Let $M = (M^1, \ldots, M^d)$ be a $d$-dimensional semimartingale. The starting point of the analysis is to solve an identification problem related to a possible singularity of the random matrix $[M]_T$ which can be typically found e.g in large portfolios of financial assets (see e.g Burashi, Porchia and Trojani~\cite{burashi}, Ait-Sahalia and Xiu~\cite{sahalia} and Fan, Li and Yu~\cite{fan} and other references therein) and affine term structure models~(see e.g Bjork and Land\'en~\cite{bjork2} and Filipovic \cite{filipo3} and Filipovic and Sharef~\cite{filipo6}). More precisely, in the presence of non-trivial correlation among semimartingales, one has $rank~[M]_T < d~a.s$ and then, under mild assumptions, one can split the set $\mathcal{M} = span~\{M^1, \ldots, M^d\}$ into two complementary linear spaces $(\mathcal{W}, \mathcal{D})$ such that

$$
\mathcal{M} = \mathcal{W}\oplus \mathcal{D}
$$
where $\mathcal{W}$ and $\mathcal{D}$ contains only elements of $\mathcal{M}$ with non-zero and zero quadratic variation, respectively. The space $\mathcal{W}$ fully describes the volatility structure of $\mathcal{M}$ while $\mathcal{D}$ is responsible for its hidden pure drift (null quadratic variation) dynamics. At this point, we stress that the potential singularity of the quadratic variation matrix $[M]_T$ introduces non-observable drift components into $M$ which \textit{cannot} be discarded in a high-frequency situation. Both spaces are equally important to explain the dynamics of $M$ in a given physical probability measure. In strong contrast, directions with null variance can be fully discarded in the classical PCA. This is the first major difference between the classical PCA and the theory developed in this article.

We follow the natural and simple idea to seek random variables $v_t=(v^1_t, \ldots, v^d_t)$
such that

$$
\sum_{j=1}^d v^j_t M^j_t
$$
has the largest possible \textit{instantaneous} quadratic variation over $[0,t]$, where $v_t$ is interpreted as a \textit{random coefficient} at time $t\in[0,T] $ rather than a process. By iterating this procedure in an orthogonal way, we shall get a linear transformation of $M$ which under some mild primitive conditions will be a finite-dimensional semimartingale ranked in terms of quadratic variation. Starting with consistent estimators $\widehat{[M]}_T$ for the quadratic variation matrix $[M]_T$ (see e.g  \cite{mikland, wang, zeng, mancino, tao, bibinger, fan, podolskij} and other references therein), we are able to propose consistent estimators for $(\mathcal{W},\mathcal{D})$ by means of a simple eigenvalue analysis of $\widehat{[M]}_T$ based on high-frequency observations of $M$. This allows us to reduce dimensionality in terms of quadratic variation in a very clear and consistent way. Equally important, the methodology also estimates bounded variation components in $\mathcal{D}$ which can not be neglected in multi-dimensional semimartingale systems.

The PCA for semimartingales introduced in the first part of the article is applied to the estimation of principal components of discretely-observed space-time semimartingales which describe stochastic partial differential equations (henceforth abbreviated by stochastic PDEs) admitting finite-dimensional realizations. In particular, in the second part of this article, we illustrate the theory by studying the problem of the estimation of the so-called finite-dimensional invariant manifolds w.r.t to a stochastic PDE

\begin{equation}\label{eqpde}
dr_t = \Big(A(r_t) + F(r_t)\Big)dt + \sum_{j=1}^m\sigma^j(r_t)dB^j_t; r_0=h\in E; 0\le t\le T,
\end{equation}
where $E$ is a potentially infinite-dimensional Sobolev-type space of continuous functions and $(A,F,\sigma^i; 1\le i\le m)$ satisfy standard assumptions for the existence of solution.

Many space-time phenomena in natural and social sciences can be described by solutions of stochastic PDEs like (\ref{eqpde}). However, the intrinsic infinite-dimensionality of space-time data generated by models like (\ref{eqpde}) creates a big challenge in the statistical analysis of these models. In particular cases, it is well-known that one can reduce dimensionality and still get a very rich class of space-time data generated by models of type (\ref{eqpde}). For instance, under Lie algebra conditions (see Filipovic and Teichmann \cite{filipo1} and Bjork and Svensson~\cite{bjork1}) on the coefficients of (\ref{eqpde}), it is well known that there exists a family of affine manifolds $\{\mathcal{G}_t; 0\le t\le T\}$ of curves and a $d$-dimensional semimartingale factor process $M$ such that

\begin{equation}\label{cona}
r_t(\cdot) = \mathcal{G}_t(\cdot, M_t); r_0=h; 0\le t\le T,
\end{equation}
where $\mathcal{G} = \{ \mathcal{G}_t(\cdot; x); x\in \mathcal{X}\subset \mathbb{R}^d; 0 \le t\le T\}\subset E$ is a finite-dimensional parameterized family of smooth curves. We shall write it as $\mathcal{G}_t = \phi_t + V$ where $V$ is a $d$-dimensional vector space generated by smooth curves and $\phi$ is an $E$-valued smooth parametrization which we assume to be a zero quadratic variation function.

Two central unsolved problems in the stochastic PDE modelling are: (i) the construction of statistical tests to check existence of $\mathcal{G}$ and (ii) the development of related estimation methods. The importance of this research agenda can be mainly understood in applications to interest rate modelling and other term-structure problems in Mathematical Finance. The literature is vast so we refer the reader to e.g~\cite{bjork, bjork1, bjork2, filipo, filipo1, filipo2, filipo3, filipo5, richter, teichmann, harms, ohashi, angeline, slinko, mancino1} and other references therein. In short, under the assumption of existence of $\mathcal{G}$, the estimation of $V$ is essential for a consistent calibration of potentially infinite-dimensional term-structure models.

Under the assumption that the stochastic PDE (\ref{eqpde}) admits an affine finite-dimensional representation~(\ref{cona}), we apply the semimartingale PCA to estimate and identify components of invariant manifolds $\mathcal{G}$ which depicts volatility and drift dynamics in space. More precisely, let us consider the finite rank random linear operator $Q_T:E\rightarrow V$ defined by

$$Q_T f:=\langle Q_T(\cdot, ), f \rangle_E; ~f\in E, $$
where $Q_T(u,v):=[r(u),r(v)]_T; u,v \ge \mathbb{R}_+$, $\langle \cdot, \cdot\rangle_E$ is the inner product of $E$ and we set $\mathcal{Q}:=~range~Q_T$. We notice the quadratic variation of the stochastic PDE (\ref{eqpde}) is fully generated by $Q_T$. In particular, the associated Hilbert-Schmidt norm

$$
\|Q_T\|^2_{(2)}=\sum_{j=1}^{dim~V}\theta^2_j
$$
fully describes the total amount of energy related to the quadratic variation of (\ref{eqpde}) over $[0,T]$. Here, $\{\theta_i\}_{i=1}^{dim~V}$ are the eigenvalues of $Q_T$ arranged in decreasing order.


In general, $dim~\mathcal{Q} \le ~dim~V~a.s$, but in typical situations we do have $\text{dim}~\mathcal{Q}< \text{dim}~V$~\footnote{The empirical literature on interest rate modelling reports strong evidence of correlation among risk factors (see e.g~\cite{alexander, piazzesi,buraschi1}) which suggest that one can typically find $dim ~\mathcal{Q} < \text{dim}~V$ in case of affine models. From theoretical side, this phenomena is also related to no-arbitrage restrictions imposed on affine models. See e.g~\cite{bjork1,filipo5,filipo6,almeida1,almeida2} and other references therein.}. Let $\mathcal{N}$ be the complementary subspace of $\mathcal{Q}$ in $V$. Under mild assumptions, we have the following splitting

$$
V = \mathcal{Q}\oplus \mathcal{N}~a.s.
$$
In one hand, the pair of subspaces $(\mathcal{Q}, \mathcal{N})$ should be considered as the analogous spaces to $(\mathcal{W},\mathcal{D})$ but in the spatial variable. On the other hand, we stress that $M$ is not observed and (\ref{cona}) is treated as a factor model

\begin{equation}\label{spacetimedata}
r_t = \phi_t + \sum_{j=1}^{\text{dim}~V}M^j_t\lambda_j; V = \text{span}~\{\lambda_1, \ldots, \lambda_d\}
\end{equation}
with dimension $d=dim~\mathcal{Q} + dim~\mathcal{N}$. The present methodology allows us to estimate and identify directions of the invariant manifold which come from the volatility (represented by $\mathcal{Q}$) and the drift (represented by $\mathcal{N}$). More importantly, we are able to identify them separately which allows us to estimate null and non-null quadratic variation factors by projecting space-time data of the form (\ref{spacetimedata}) onto a pair of estimated vector spaces $(\widehat{\mathcal{Q}}\oplus \widehat{\mathcal{N}})$. We consider this separation feature as the most important aspect of the second part of this article. As a by-product, our methodology brings two contributions to the field: It provides a consistent volatility dimension reduction and a method to estimate hidden pure drift components in space-time semimartingale data generating processes.

Our methodology is a combination of classical factor models jointly with suitable random transformations over the space of latent semimartingales.  More precisely, our approach consists essentially in two steps: Firstly, we apply an empirical covariance operator onto the space-time data to obtain a factor decomposition of the form
$$\widehat{r}_t(x) = \sum_{j=1}^k \widehat{Y}_t^k \widehat{\lambda}^k(x),~t\in \Pi, x\in \Pi'$$
in the spirit of discrete-type factor models (see e.g Stock and Watson \cite{stock}, Bai~\cite{bai1} and Bai and Ng~\cite{bai}), but in a high-frequency setup as opposed to the usual panel data. In other words, $\Pi\times\Pi'$ is a refining partition of a two-dimensional set $[0,T]\times[a,b]$. In linear structures, the covariance operator only neglects components of null empirical variance so that, under suitable conditions, our first step does not loose information from the invariant manifold $V = \mathcal{Q}\oplus\mathcal{N}$. The second step consists in using the semimartingale PCA jointly with suitable random rotations of latent factor estimators $\widehat{Y}_t$ to infer the underlying semimartingale structure of the data. It is not easy to foresee that this two-step procedure would work. Indeed, to the best of our knowledge it is not known that the covariance operator decomposition is strong enough to provide a resulting process which is amenable to a consistent quadratic variation analysis. In fact, the sequence $\widehat{Y}_t$ is not even associated to a semimartingale, so that the quadratic variation analysis based on this two-step procedure must be considered in a broader sense. The proof that this strategy works is the content of the second part of the paper.

It is important to stress the both steps in our methodology are equally important. For instance, the naive application of classical factor models to infer quadratic variation is non-sense when applied to semimartingale systems. Moreover, a more straightforward strategy based directly on an empirical quadratic variation does not work in full generality due to a possible singularity of the matrix. This last procedure forces the assumptions that $\text{dim}~V=\text{dim}~\mathcal{Q}$ which may not be optimal (in the mean-square sense) in typical situations when $\text{dim}~\mathcal{N}>0$. This is the reason why the two-step procedure in this work is implemented. For instance, Pelger~\cite{pelger} studies principal components directly from the empirical quadratic covariation for factors with jumps and discrete loading factors. One crucial assumption in his setup is the non-singularity of the quadratic variation matrix which restricts the applicability in multivariate systems with non-trivial correlation typically found in large portfolios and interest rate models.

We should mention that another possible framework is introduced by Ait Sahalia and Xie \cite{sahalia2} who interpret principal component analysis by means of the underlying volatility process. The main drawback of this strategy is the fact that rank of the volatility matrix may be strictly smaller than $\text{dim}~[M]_T$ (as shown in Proposition \ref{diffusionmatrix}), thus resulting in a substantial underestimation of $\text{dim}~\mathcal{W}$ and their associated factors. Therefore, similar to Pelger~\cite{pelger}, the strategy introduced by \cite{sahalia2} does not recover in full generality the whole semimartingale stricture ($\mathcal{M}$) involved in the optimal decomposition due to a possible non-negligible dimension ($\text{dim}~\mathcal{D})$ associated to the drift. In addition, the strong assumption of simple eigenvalues imposed in \cite{sahalia2} rules out many finite-dimensional semimartingale systems typically found in applications.

\subsection{Organization of the paper}
The remainder of this article is structured as follows. Section \ref{intro} presents some notation and preliminary results. Section \ref{randomdirections} presents the spectral analysis on a generic quadratic variation matrix. Section \ref{mainexamples} illustrates the existence of bounded variation components in portfolio management and interest rate models. Section \ref{estimationofwd} presents the consistency results for the estimators of the dynamic spaces. Section \ref{estimationFDIM} presents the application of semimartingale PCA to the problem of estimating finite-dimensional invariant manifolds for stochastic PDEs. Section \ref{numerics} presents the numerical results and applications to real data. An Appendix is given in Section \ref{Appendix} which presents an estimator for $\text{dim}~\mathcal{Q}$.

\section{Assumptions and Preliminary Results}\label{intro}
At first, let us fix notation.
\subsection{Notation}
Throughout this article, we are going to work with a fixed stochastic basis of the form $(\Omega, \mathcal{F}_T,\mathbb{F},\mathbb{P} )$ where $(\Omega, \mathcal{F}_T, \mathbb{P})$ is a probability space equipped with a sample space $\Omega$, a sigma-algebra $\mathcal{F}_T$, probability measure $\mathbb{P}$ and a fixed terminal time $0< T< \infty$. We equip the interval $[0,T]$ with the Borel sigma algebra $\mathcal{B}_T$ and we assume the filtration $\mathbb{F}:=\{\mathcal{F}_t; 0\le t \le T\}$ satisfies the usual conditions.

All the algebraic setup in this article will be based on the real linear space $\mathcal{X}^d$ constituted by the set of all $\mathbb{R}^d$-valued $\mathcal{B}_T\times \mathcal{F}_T$-measurable processes. In this article, the most important subclass of $\mathcal{X}^d$ will be the subspace $\mathcal{S}^d
$ constituted by the set of all $\mathbb{R}^d$-valued continuous $\mathbb{F}$-semimartingales on $(\Omega, \mathcal{F}_T,\mathbb{F},\mathbb{P} )$. When $d=1$, we set $\mathcal{S}:=\mathcal{S}^1, \mathcal{X} := \mathcal{X}^1$. We denote $\mathbb{L}^{0,k}_t$ as the set of all $\mathbb{R}^k$-valued and $\mathcal{F}_t$-measurable random variables for $k\ge 1$ and $t\in [0,T]$. Throughout this article, we adopt the following convention: If $Y\in \mathcal{X}^d$, then $Y_t$ is interpreted as a column random vector in $\mathbb{R}^d$. Convergence in probability will be denoted by $\stackrel{p}{\to}$.

In the remainder of this article, $\Pi$ denotes a deterministic partition $0=t_0 < t_1 <\ldots < t_n=T$ and $\|\Pi\|: = max_{1\le i\le n}|t_i-t_{i-1}|$. The set $\mathbb{M}_{p\times q}$ denotes the space of all $p\times q$-real matrices and $\mathbb{M}^+_{p\times p}$ is the subspace of $p\times p$ non-negative symmetric real matrices. The norm of linear operators between Hilbert spaces will be the standard Hilbert-Schmidt norm $\|\cdot\|_{(2)}$ and $P^\top$ denotes the transpose of a matrix $P\in \mathbb{M}_{p\times q};p,q\ge 1$. If $A,B$ are two linear subspaces of $\mathcal{X}$ with $A\subset B$, then we denote $\pi_A$ the usual projection of $B$ onto the quotient space $B/A$. Throughout this article, we omit the variable $\omega\in \Omega$ when no confusion arises.

\subsection{Analysis of quadratic variation matrices}

In this work, the following bracket will play a key role in our analysis

\begin{equation}\label{quadraticcov}
[X,Y]_t:=\lim_{\|\Pi\|\rightarrow 0}\sum_{t_i\in \Pi} \big(X_{t_i} - X_{t_{i-1}}\big)\big( Y_{t_i} - Y_{t_{i-1}} \big); 0\le t\le T,
\end{equation}
in probability.
\begin{definition}\label{qdef}
The \textbf{quadratic covariation} $\{[X,Y]_t; 0\le t\le T\}$ exists for a given pair $(X,Y)\in \mathcal{X}^2$ if the limit (\ref{quadraticcov}) exists for every sequence of partitions $\Pi$ such that $\|\Pi\|\rightarrow 0$. We say that $X\in \mathcal{X}$ has \textbf{null quadratic variation} if $[X,X]_\cdot =0~a.s$
\end{definition}
Of course, $\{[X^1,X^2]_t; 0\le t\le T\}$ is a well-defined bounded variation adapted process for every $(X^1,X^2)\in \mathcal{S}^2$. To shorten notation, we sometimes set $[Y]:= [Y,Y]$ for $Y\in \mathcal{X}$. For a given $X = (X^1,\ldots, X^d)\in \mathcal{X}^d$ and $(t,\omega)\in [0,T]\times \Omega$, with a slight abuse of notation, we write $[X]_t(\omega)\in \mathbb{M}^+_{d\times d}$ to denote the following random matrix

\begin{equation}\label{entmatrix}
[X]_t(\omega):=[X^i,X^j]_t(\omega);~i,j=1,\ldots, d;~0\le t\le T,\omega \in \Omega,
\end{equation}
whenever the right-hand side of (\ref{entmatrix}) exists.

In the remainder of this section, $M=(M^1,\ldots,M^d)$ is a given $d$-dimensional measurable process. We say that $M\in \mathcal{X}^d$ is truly $d$-dimensional if its components $M^1,\ldots,M^d$ are linearly independent over the vector space $\mathcal{X}$. Throughout this paper, we are going to assume the following standing assumptions:

\begin{assumption}\label{A1}
$M$ is a truly $d$-dimensional measurable process.
\end{assumption}
\begin{assumption}\label{A2}
The quadratic variation matrix $\{[M]_t; 0\le t\le T\}$ exists and if there exists $i=1,\ldots,d$ such that $\mathbb{P}\{[M^i,M^i]_t>0\}>0$ then we have $\mathbb{P}\{[M^i, M^i]_t>0\}=1$.
\end{assumption}
\begin{remark}
We clearly do not loose generality by imposing Assumption \ref{A1}. Assumption \ref{A2} is very natural since our theory relies on the study of a realization of the quadratic variation matrix, and thus it is necessary that we do not get a realization of null quadratic variation from a non-null quadratic variation process.
\end{remark}

\

\noindent \textbf{Example:} One typical example of semimartingale which satisfies Assumption \ref{A2} is given by the $2d$-Heston model $(M^i,V^i); i=1,\ldots, d$ with correlation in $[-1,1]$ where $V^i$ denotes the $i$th square-root-type stochastic volatility component for $i=1,\ldots, d$. Then, one can easily check that for every $t\in (0,T]$, we have

$$[M^i , M^i]_t = \int_0^t|M^i_s|^2V^i_sds > 0 ~a.s ~\text{for every}~i=1,\ldots, d.$$
Hence, the classical Heston model satisfies Assumption \ref{A2}.

\

Let $\mathcal{M}_t := span\{M^1,\ldots,M^d\}$ be the linear space spanned by the $1$-dimensional measurable processes $M^1,\ldots,M^d$ over $[0,t]$ for $0 \le t\le T$. Assumption \ref{A1} yields $\dim\mathcal{M}_t = d$ for every $t\in (0,T]$. Let us now split $\mathcal{M}_t$ into two orthogonal subspaces. At first, we set

\begin{equation}\label{Dt}
\mathcal{D}_t := \{X\in\mathcal{M}_t; [X]_t = 0 ~a.s\}.
\end{equation}
Observe that $\mathcal{D}_t$ is a well-defined linear subspace of $\mathcal{M}_t$ for every $t\in [0,T]$. More importantly, the following remark holds.

\begin{remark}
We recall that any continuous bounded variation local martingale must be constant a.s. Moreover, for every $t\in (0,T]$

$$\{\omega; [Y,Y]_t(\omega)=0\} = \{\omega\in \Omega; N_\cdot(\omega)=0~\text{over the interval}~[0,t]\}$$
where $N$ is the local martingale component of the special semimartingale decomposition of some $Y\in\mathcal{S}$. Therefore, Assumption \ref{A2} allows us to state that if $M\in \mathcal{S}^d$ is a truly $d$-dimensional process, then $\mathcal{D}_t$ is a subspace of $\mathcal{M}_t$ only constituted by continuous bounded variation adapted processes over $[0,t]$.
\end{remark}

\begin{definition}
Let $\mathcal{M}_t$ be the span generated by a truly $d$-dimensional measurable process $M\in \mathcal{X}^d$ over $[0,t]$. If $\dim\mathcal{D}_t >0$, then we say that $\mathcal{M}_t$ has a \textbf{null quadratic variation component} over the interval $[0,t]$. In particular, if $M\in \mathcal{S}^d$ and $\dim\mathcal{D}_t >0$, then we say that $\mathcal{M}_t$ has a \textbf{bounded variation component} over $[0,t]$.
\end{definition}
Let us give a toy example showing how a non-trivial dimension induced by bounded variation processes may appear in a very simple context.

\

\noindent \textbf{Example:} Let $B$ be a one-dimensional Brownian motion and let $M_t = (B_t, B_t + t); 0\le t\le T.$ Of course, $M$ is a truly $2$-dimensional semimartingale where $dim~\mathcal{M}_t=2$ for every $t\in (0,T]$. In particular, we clearly have $dim~\mathcal{D}_t=1$ for every $t\in (0,T]$.

\

For a deeper discussion of bounded variation components on semimartingale systems, we refer the reader to Section \ref{mainexamples}. Let us now provide a natural notion of ``quadratic variation dimension'' in $\mathcal{M}_t$. To do so, let us consider the following quotient space

\begin{equation}\label{quot}
\widetilde{\mathcal{M}}_t := \mathcal{M}_t/\mathcal{D}_t;~0 \le t\le T.
\end{equation}
By definition, $\widetilde{\mathcal{M}}_t$ can be identified by $(\mathcal{M}_t, \sim)$ where the equivalence relation is given by

\begin{equation}\label{eqrel}
X\sim Y \Leftrightarrow X - Y~\text{is a null quadratic variation process in}~\mathcal{M}_t~\text{over the interval}~[0,t].
\end{equation}

The following simple result connects the rank of $[M]_t$ with the dimension of $\widetilde{\mathcal{M}}_t$.

\begin{lemma}\label{WZ}
Let $M\in \mathcal{X}^d$ be a truly $d$-dimensional measurable process satisfying Assumption \ref{A2}. Then,
$$rank [M]_t = \dim~\widetilde{\mathcal{M}}_t~a.s$$
for $t\in [0,T]$.
\end{lemma}
\begin{proof}
The result for $t=0$ is obvious so we fix $0 < t \le T$. Let $p_t = \dim~\widetilde{\mathcal{M}}_t$ and let $\pi_{\mathcal{D}_t}:\mathcal{M}_t\to\widetilde{\mathcal{M}}_t$ be the standard projection of $\mathcal{M}_t$ onto $\widetilde{\mathcal{M}}_t$. Now, we observe that for each $P\in\mathcal{M}_t$, $\pi_{\mathcal{D}_t}(P)$ is a set of continuous measurable processes, each of which differs from each other by a continuous null quadratic variation measurable process over the interval $[0,t]$. Nevertheless, for each process in $\pi_{\mathcal{D}_t}(P)$ its quadratic variation is equal to $[P]_t$. Therefore, we may define its quadratic variation as $[P]_t$. By the polarization identity, we may define

\begin{equation}\label{polar}
[\pi_{\mathcal{D}_t}(P),\pi_{\mathcal{D}_t}(Q)]_t:=[P,Q]_t
\end{equation}
for any $P,Q\in \mathcal{M}_t$. In particular, this shows that $[N,Z]_t$ is a well-defined random variable for any $N,Z\in \widetilde{\mathcal{M}}_t$.

Since $\text{span}~\{\pi_{\mathcal{D}_t}(M^1),\ldots,\pi_{\mathcal{D}_t}(M^d)\} = \widetilde{\mathcal{M}}_t$, then $\pi_{\mathcal{D}_t}(M^1),\ldots,\pi_{\mathcal{D}_t}(M^d)$ contains $p_t$ linearly independent components in the vector space $\widetilde{\mathcal{M}}_t$. Therefore, $ \dim~\widetilde{\mathcal{M}}_t$ equals to the number of linearly independent components in the subset $\{\pi_{\mathcal{D}_t}(M^1),\ldots,\pi_{\mathcal{D}_t}(M^d)\}$.

Let us now consider a subset of $k$ equivalence classes $\pi_{\mathcal{D}_t}(M^{\sigma(1)}),\ldots,\pi_{\mathcal{D}_t}(M^{\sigma(k)})$, where $\sigma:\{1,\ldots,k\}\to\{1,\ldots,d\}$
is a function. Let $c_1,\ldots,c_k\in\mathbb{R}$. In the sequel, we denote by $\overrightarrow{0}$ the null element of $\widetilde{\mathcal{M}}_t$. With this notation at hand, Cauchy-Schwartz inequality yields

 $$\sum_{i=1}^k c_i \pi_{\mathcal{D}_t}(M^{\sigma(i)})=\overrightarrow{0}\Leftrightarrow \forall N\in\widetilde{\mathcal{M}}_t, \left[\sum_{i=1}^k c_i \pi_{\mathcal{D}_t}(M^{\sigma(i)}),N\right]_t=0~a.s.$$
 In particular,
 $$\sum_{i=1}^k c_i \pi_{\mathcal{D}_t}(M^{\sigma(i)})=\overrightarrow{0}\Leftrightarrow \forall j=1,\ldots,k, \left[\sum_{i=1}^k c_i \pi_{\mathcal{D}_t}(M^{\sigma(i)}),\pi_{\mathcal{D}_t}(M^{\sigma(j)})\right]_t=0~a.s.$$

By recalling that $\{\pi_{\mathcal{D}_t}(M^{\sigma(1)}),\ldots,\pi_{\mathcal{D}_t}(M^{\sigma(k)})\}$ is linearly independent if, and only if,
 $$\sum_{i=1}^k c_i \pi_{\mathcal{D}_t}(M^{\sigma(i)})=\overrightarrow{0} \Rightarrow c_1=\cdots =c_k=0,$$
then the statement $\{\pi_{\mathcal{D}_t}(M^{\sigma(1)}),\ldots,\pi_{\mathcal{D}_t}(M^{\sigma(k)})\}$ is a linearly independent set is equivalent to the system of equations
 $$\sum_{i=1}^k c_i \left[\pi_{\mathcal{D}_t}(M^{\sigma(i)}),\pi_{\mathcal{D}_t}(M^{\sigma(j)})\right]_t=0~a.s,\qquad j=1,\ldots,k$$
has only the trivial solution $c_1=\cdots=c_k=0$ almost surely. In other words,

\begin{equation}\label{polar1}
\det\Bigg(\left[\pi_{\mathcal{D}_t}(M^{\sigma(i)}),\pi_{\mathcal{D}_t}(M^{\sigma(j)})\right]_t; i,j=1,\ldots, k\Bigg)\neq 0~a.s.
\end{equation}
From (\ref{polar}), (\ref{polar1}) and Assumption \ref{A2}, we shall conclude the proof.
\end{proof}

Summing up the results of this section, we arrive at the following direct sum

\begin{equation}\label{sp1}
\mathcal{M}_t = \mathcal{W}_t \oplus \mathcal{D}_t ; 0\le t \le T,
\end{equation}
where $\{\mathcal{W}_t; 0\le t \le T\}$ is the unique (up to isomorphisms) family of complementary linear subspaces of $\mathcal{M}_t$ which realizes (\ref{sp1}). One should notice that $\mathcal{W}_t$ is formed by the null process in $\mathcal{X}$ on $[0,t]$ and of elements $V$ in $\mathcal{M}_t$ such that $[V,V]_t >0~a.s$. Of course, $\mathcal{W}_t$ is isomorphic to $\widetilde{\mathcal{M}}_t$ for every $t\in [0,T]$. To shorten notation, in the remainder of this article, we write $\mathcal{M} := \mathcal{M}_T, \mathcal{W} := \mathcal{W}_T$, $\widetilde{\mathcal{M}}:=\widetilde{\mathcal{M}}_T$ and $\mathcal{D} := \mathcal{D}_T$.

\section{Random directions and principal components}\label{randomdirections}
Let us start with some heuristics related to reduction dimension for a high-dimensional vector of semimartingales $M = (M^1,\dots, M^d)\in \mathcal{S}^d$ which we suspect there may be some redundancy in the sense of quadratic variation. Perhaps there may be some way to combine $M^1,\ldots,M^d$ that captures much of the quadratic variation in a few aggregate semimartingales. In particular, we shall seek random variables $v_t=(v^1_t, \ldots, v^d_t)\in \mathbb{L}^{0,d}_t$
such that

\begin{equation}\label{first}
S_t:=\sum_{j=1}^d v^j_t M^j_t
\end{equation}
has the largest possible \textit{instantaneous} quadratic variation over $[0,t]$, where $v_t=(v^1_t, \ldots, v^d_t)$ in (\ref{first}) is interpreted as a \textit{random coefficient} at time $t\in[0,T] $ rather than a process. In other words, we seek a random linear combination of the form (\ref{first}) such that

$$\sum_{i,j=1}^dv^i_tv^j_t[M^i,M^j]_t$$
has almost surely the largest possible value over the subset of $\mathbb{L}^{0,d}_t$ with Euclidean norm 1 for a given $t\in [0,T]$. Indeed, we do compute the quadratic variation of the linear combination $S$ at time $t$ by considering $v_t$ as a random constant over $[0,t]$ which yields

$$\sum_{i,j=1}^dv^i_tv^j_t[M^i,M^j]_t = [S,S]_t.
$$

The random coefficient
$$\bar{v}_t = \underset{ w_t\in \mathbb{L}^{0,d}_t, \|w_t\|_{\mathbb{R}^d}=1}{\operatorname{argmax}}\Bigg[ \sum_{i=1}^d w^i_t M^i, \sum_{i=1}^d w^i_t M^i \Bigg]_t$$
encodes the way to combine $M^1, \ldots, M^d$ to maximize instantaneous quadratic variation at time $t\in [0,T]$. The new variable - the leading principal component - is $\sum_{i=1} ^d\bar{v}^i_tM^i_t$. We shall continue this strategy by seeking a possible lower dimensional pairwise orthogonal sequence of aggregate variables which might explain most of the quadratic variation at each time $t\in [0,T]$.

For simplicity of exposition, we assume that one observes all trajectories of a given truly $d$-dimensional continuous time semimartingale $M$ satisfying Assumptions \ref{A1} and \ref{A2}. Let us now interpret the eigenvalues and eigenvectors of the quadratic variation matrix in a similar
manner of what we interpret in covariance matrices as in classical PCA. In the sequel, we introduce the brackets which encode quadratic variation of random linear combinations as described at the beginning of this section

\begin{equation}\label{rqf}
\langle X^\top_t Y\rangle_t:=\sum_{i,j=1}^dX^i_tX^j_t[Y^i,Y^j]_t;~X_t\in \mathbb{L}^{0,d}_t, Y\in \mathcal{S}^d
\end{equation}
The bracket $\langle X^\top_t Y, P^\top_t R \rangle_t $ is naturally defined by polarization. The bracket given in (\ref{rqf}) encodes the quadratic variation of $X^\top_t Y$ at time $t\in [0,T]$ where $X_t$ is considered as a random constant over $[0,t]$ in the computation of (\ref{rqf}). This is perfectly consistent to what happens in practice because at a given time $t\in [0,T]$, one observes a high-frequency data from a semimartingale $M$ over $[0,t]$ and one has to decide if there exist linear combinations of the elements of $M_t$ which summarizes the quadratic variation $[M]_t$.



\begin{lemma}\label{eigenvatheorem}
Let $M\in \mathcal{S}^d$ be a truly $d$-dimensional semimartingale satisfying Assumption 2.2. Let us consider the vector of eigenvalues $\lambda^1_t(\omega),\ldots,\lambda^d_t(\omega)$ (ordered in such way that $\lambda^1_t(\omega)\ge \lambda^2_t(\omega)\ge \ldots \ge \lambda^d_t(\omega)$) of the matrix $[M]_t(\omega)$ for each $(\omega,t)\in \Omega\times [0,T]$. Then, for each $i$, $\{\lambda^i_t; 0\le t\le T\}$ is an adapted bounded variation process.
\end{lemma}

\begin{proof}
By the very definition, any eigenvalue $\lambda_t(\omega)$ is a root of the characteristic polynomial $p(\lambda)=~det (\lambda I - [M]_t(\omega))$ of the random matrix $[M]_t(\omega)$. The degree of this polynomial is $d$ and its coefficients depend on the entries of $[M]_t(\omega)$, except that its term of degree $d$ is always $(-1)^d\lambda^d$. This allows us to conclude that the ordered eigenvalues are $\mathbb{F}$-adapted. In particular, by the classical Weyl's perturbation theorem, we know there exists a deterministic constant $C$ such that

$$max_{j}|\lambda^j_t(\omega) - \lambda^j_s(\omega)|\le C \|[M]_t(\omega) - [M]_s(\omega)\|_{\infty}; (\omega,t)\in \Omega\times [0,T]$$
where $\|\cdot\|_{\infty}$ denotes the entrywise $\infty$-norm of a symmetric matrix. By writing $\|[M]_t - [M]_s\|_\infty = max_{1\le j\le d}\sum_{i=1}^d|[M^i,M^j]_t(\omega) - [M^i,M^j]_s(\omega)|$, we clearly see $t\mapsto \lambda^j_t(\omega)$ has bounded variation for almost all $\omega\in \Omega$.
\end{proof}

We are now able to summarize our discussion with the following result.
\begin{proposition}\label{ranking}
Let $M$ be a semimartingale satisfying Assumptions \ref{A1} and \ref{A2}. For a given $t\in [0,T]$, let $(\lambda^1_t, \ldots, \lambda^d_t)$ be the list of eigenvalues of $[M]_t$ (arranged in decreasing order) and let $(v^1_t,\ldots, v^d_t)$ be an associated set of eigenvectors. Then, for every $t\in [0,T]$

$$\langle(v^1_t)^\top M\rangle_t = \max_{\substack{X_t\in \mathbb{L}^{0,d}_t},\\ \|X_t\|_{\mathbb{R}^d}=1}\langle X^\top_t M\rangle_t = \lambda^1_t~a.s$$

$$\langle (v^k_t)^\top M\rangle_t = \max_{\substack{X_t\in \mathcal{V}^k_t},\\ \|X_t\|_{\mathbb{R}^d}=1}\langle X^\top_t M\rangle_t = \lambda^k_t~a.s;~k=2,\ldots, d$$
where $\mathcal{V}^k_t:=~\text{orthogonal complement of}~span ~\{v^1_t, \ldots, v^{k-1}_t\}$ in $\mathbb{R}^d$ for $k=2,\ldots, d$.
In addition, if $t\mapsto [M]_t$ is a generic\footnote{For a continuous real-valued function $f$ defined in a neighborhood of $t_0$, the order of flatness $m_{t_0}(f)$ at $t_0$ is defined by the supremum of all integers $p$ such that $f(t) = (t-t_0)^p g(t)$ near $t_0$ for a continuous function $g$. We say that two functions $f$ and $h$ meet of order $\ge p$ at $t_0$ when $m_{t_0}(f - h) \ge  p$. Let $A(t); 0\le t\le T$ be a parameterized family of self-adjoint matrices. We say that the curve $t \mapsto A(t)$ is generic, if no two of continuously parameterized eigenvalues meet of infinite order at any $t\in [0,T]$ if they are not equal for all $t$. We refer the reader to e.g Rutter \cite{rutter} and Alekseevsky, Kriegl, Losik, and Michor \cite{alek} for further details.} smooth curve a.s and $dim~\widetilde{\mathcal{M}}_t=p$ is a.s constant over the time interval $(0,T]$, then there exists a choice of adapted eigenvector processes $(v^1,\ldots, v^d)$ over $[0,T]$ such that


$$S^i_t:=(v^i_t)^\top M_t; 0\le t\le T,$$
is a semimartingale for each $i\in \{1,\ldots, p\}$.

\end{proposition}
\begin{proof}
Fix a realization $\omega\in\Omega$ and $t\in [0,T]$. Let $A = (a_{ij})$ be a $d\times d$ matrix with entries given by
$$a_{ij} = [M^i,M^j]_t(\omega);~i,j=1\ldots,d.$$
It follows from Assumptions \ref{A1} and \ref{A2} that $A$ is a non-negative definite matrix. Now, let us take $v_t\in \mathbb{L}_t^{0,d}$, and let $z = (z_1,\ldots,z_d) \in \mathbb{R}^d$
be given by $z_i = v_t^i(\omega)$. Then,

$$
\langle z,Az\rangle_{\mathbb{R}^d}= \langle v_t^\top M\rangle_t.
$$
Now, the variational characterization of eigenvalues follows from standard arguments on quadratic forms over $\mathbb{R}^n$ for each $(\omega,t)\in \Omega\times [0,T]$. For the second part, if $t\mapsto [M]_t(\omega)$ is $C^\infty$ then from Theorem 7.6 in Alekseevsky et al~\cite{alek}, one can choose smooth versions for related eigenvectors $v^1(\omega), \ldots, v^d(\omega)$ with bounded variation paths. By Gaussian elimination and Lemma \ref{eigenvatheorem}, one can readily see that one can choose it in such way that $(v^1,\ldots, v^d)$ is a $d$-dimensional adapted process. The usual integration by parts for stochastic integrals allows us to state that $S = (S^1,\dots, S^p)$ is a semimartingale.
\end{proof}

Similar to the classical PCA methodology based on covariance matrices, Proposition \ref{ranking} yields a dimension reduction based on quadratic variation rather than covariance as follows. Let $M$ be a truly $d$-dimensional semimartingale satisfying Assumption \ref{A2} and let us assume that one observes $[M]_t(\omega)$ for a given $(\omega,t)\in \Omega\times (0,T]$. Summing up the above results, we shall reduce dimensionality as follows

\begin{equation}\label{PCArank}
S^i_t= \sum_{j=1}^d v^{ij}_t M^j_t;i=1, \ldots, dim~\widetilde {\mathcal{M}}_t,~0 < t \le T.
\end{equation}
At this point it is pertinent to make some remarks about (\ref{PCArank}). At first, the assumption in Proposition \ref{ranking} that $dim~\widetilde{\mathcal{M}}_t=p$ is constant a.s over $(0,T]$ holds in typical cases found in practice.


\begin{remark}
In order to get semimartingale principal components, the assumption that $t\mapsto [M]_t$ is generic cannot be avoided. See e.g example 7.7 in Alekseevsky et al \cite{alek}. However, one should notice that if two eigenvalues meet at an infinite order at a time $t_0$, then all derivatives at this point must coincide. \end{remark}

By the very definition, $\lambda^1_t\ge \lambda^2_t\ge \ldots \ge \lambda^d_t\ge 0~a.s~\text{for every}~t\in[0,T]$ which means that $S^i$ presents the $i$th largest quadratic variation among $\{S^1, \ldots,S^p\}$. One should notice that the principal components are orthogonal in the sense

$$\langle v^i_t, [M]_tv^j_t\rangle_{\mathbb{R}^d} = \langle (v^i_t)^\top M, (v^j_t)^\top M\rangle_t = 0~a.s; 0 \le t\le T, i\neq j$$
where $S^i_\cdot = (v^i_\cdot)^\top M_\cdot, S^j_\cdot = (v^j_\cdot)^\top M_\cdot$ for $i\neq j$. Moreover, the $i$-th eigenvector $v^i_t$ must be interpreted as the random direction in $\mathbb{R}^d$ at time $t$ which maximizes $\Big[\sum_{j=1}^d a_j M^j,\sum_{j=1}^d a_j M^j\Big]_t$ over $a\in \mathcal{V}^i_t; \|a\|_{\mathbb{R}^d} = 1$.

\begin{remark}
We stress that
$$[S^i,S^i]_t \neq \langle (v^i)^\top_t M\rangle_t = \sum_{\ell,m=1}^d v^{i\ell}_tv^{im}_t[M^\ell,M^m]_t; 0\le t\le T, i=1\ldots, d$$
where $S^i_r = \sum_{j=1}^d v^{ij}_r M^j_r; 0\le r\le t, 1\le i\le d$. Therefore, our methodology is rather different from Ait-Sahalia and Xiu \cite{sahalia2}. In qualitative terms, our framework does not loose information in terms of the underlying quadratic variation space $\mathcal{W}$ (See Proposition \ref{diffusionmatrix}) and hence in terms of $\mathcal{M}$ as well. In addition, we do not require a simple eigenvalue structure as required in \cite{sahalia2}.
\end{remark}

Let us now briefly discuss the importance of the subspaces $(\mathcal{D},\mathcal{W})$ in concrete multi-dimensional semimartingale systems.

\section{Bounded variation component and quadratic variation in $\mathcal{M}$}\label{mainexamples}
In this section, we discuss two concrete examples of models which exemplify the importance of analyzing the principal components of high-dimensional semimartingale systems in terms of $(\mathcal{W}, \mathcal{D})$ rather than covariance matrices.

\subsection{Correlation in $d$-dimensional asset prices}
Correlation among asset prices is a well-known phenomena and it has been studied by many authors in the context of covariance and, more recently, quadratic variation matrices. Let us suppose the asset log-prices form a $d$-dimensional It\^o process

\begin{equation}\label{obsSM}
M^i_t = M^i_0 + \int_0^tb^i_sds + \sum_{j=1}^d\int_0^t \sigma^{ij}_sdB^j_s; ~i=1,\ldots, d; 0\le t\le T,
\end{equation}
where $b:[0,T]\times\Omega\rightarrow \mathbb{R}^d$ and $\sigma:[0,T]\times \Omega \rightarrow \mathbb{R}^{d\times d}$ satisfy usual conditions to get a well-defined $d$-dimensional semimartingale. For simplicity of exposition, let us assume that $d$ is known.

One typical example of the existence of bounded variation component in $\mathcal{M}$ is the occurrence of correlation among $M^1, \ldots, M^d$ which can be measured by volatility, i.e., quadratic variation. This type of phenomena has been recently studied by Ait-Sahalia and Xiu \cite{sahalia} who identify nontrivial correlation among $\{M^i; i=1,\ldots, d\}$ by means of suitable estimators $\widehat{[M^i, M^j]}_T; i,j=1, \ldots, d$. In the presence of correlation among assets as in \cite{sahalia}, the subspace $\mathcal{D}$ naturally emerges as a non-trivial subspace of $\mathcal{M}$ due to the fact that $rank~[M]_T< d$. See also Buraschi, Porchia and Trojani \cite{burashi} for a discussion of correlation in the context of optimal portfolio choice.

\subsection{Stochastic PDEs with finite-dimensional realizations}\label{expSPDE}
Let us describe how $(\mathcal{W},\mathcal{D})$ arises in the context of stochastic PDEs. Let us concentrate the discussion in one major research theme related to interest rate modelling: The calibration problem of Heath-Jarrow-Morton models~\cite{heath} (henceforth abbreviated by HJM) based on forward rate curves. We refer the reader to e.g~\cite{bjork,bjork1,bjork2,filipo} and other references therein for a detailed discussion on this issue. The classical HJM model can be described by a stochastic PDE of the form

\begin{equation}\label{bspde1}
dr_t = \Big(A(r_t) + \alpha_{HJM}(r_t)\Big)dt + \sum_{i=1}^m \sigma^i(r_t)dB_t^i; \quad r_0 \in E,
\end{equation}

\noindent where $A = \frac{d}{dx}$ is the first-order derivative operator acting as an infinitesimal generator of a $C_0$-semigroup on a separable Hilbert space $E$ which we assume to be a space of functions $g:\mathbb{R}_+\rightarrow \mathbb{R}$.
The drift vector field $\alpha_{HJM}$ has great importance for pricing and hedging derivative products and it is fully determined by $\sigma = \{\sigma^1, \ldots, \sigma^m\} $ under a martingale measure. See e.g~\cite{filipo} for more details.

One central issue in the literature is the use of the stochastic PDE (\ref{bspde1}) in practice. In this case, it is very important to know when (\ref{bspde1}) admits a finite-dimensional subset $\mathcal{G}$ where the stochastic PDE never leaves as long as the initial forward rate curve $r_0\in \mathcal{G}$, namely

$$
\mathbb{P}\{ r_t \in \mathcal{G};~\forall t \in [0,T] \}=1\quad \text{if}~r_0\in \mathcal{G}.
$$
The subset $\mathcal{G}$ can be interpreted as a finite-dimensional parameterized family of smooth curves $\mathcal{G} = \{G(\cdot; x); x\in \mathcal{Z}\subset \mathbb{R}^d\}\subset E$ which can be used to estimate the volatility component of the model (\ref{bspde1}) starting with an initial curve $r_0\in \mathcal{G}$. See e.g \cite{angeline,bjork1}. Therefore, one central issue in interest rate modelling is the existence, characterization and estimation of $\mathcal{G}$. See \cite{bjork,bjork1,bjork2,filipo,filipo1,filipo2,filipo3,ohashi,richter,angeline,mancino1} and other references therein.

As far as the existence is concerned, Bjork and Svensson \cite{bjork2} and Filipovic and Teichmann~\cite{filipo1} have shown that the existence of $\mathcal{G}$ is equivalent to

$$dim~\{\mu,\sigma^i;i=1,\ldots,m \}_{LA}< \infty,$$
in a neighborhood of $r_0$, where $\mu$ is the Stratonovich drift induced by $\sigma$ and $x\mapsto\{\mu,\sigma^1,\ldots,\sigma^m \}_{LA}(x)$ is the Lie algebra generated by the vector fields $\mu,\sigma^1,\ldots,\sigma^m$. In fact, $\mathcal{G}\subset E$ must be an affine submanifold of $E$.
In particular, there exists a parametrization $\phi:[0,T]\rightarrow E$, a truly $d$-dimensional Brownian semimartimgale $M = (M^1, \ldots, M^d)$ and a linear subspace $V=~\text{span}~\{\lambda_1, \ldots, \lambda_d\}$ spanned by a basis $\{\lambda_i\}_{i=1}^d$ such that

\begin{equation}\label{fdre}
r_t(x) = \phi_t(x) + \sum_{j=1}^d M^i_t \lambda_i(x)~a.s; 0\le t\le T; x\ge 0.
\end{equation}
Under some assumptions (see e.g Duffie and Khan \cite{duffie}), the semimartingale state process $M$ can be generically written as an affine process. In contrast to the previous example of sample data from the $d$-dimensional semimartingale (\ref{obsSM}), $M$ in (\ref{fdre}) is \textit{not} observed.

For a given pair $(M,V)$ as above, one can actually show there exists a unique splitting $V = V_1\oplus V_2$ which realizes

 $$r_t(x) = \phi_t(x) + \sum_{i=1}^p Y^i_t \varphi_i(x) +\sum_{j=p+1}^d \tilde{Y}^j_t \varphi_j(x) ~a.s$$
for $ 0\le t\le T; x\ge 0$. Here, $\{Y^i;i=1,\ldots, p\}$ is a basis for $\mathcal{W}$ and $\{\tilde{Y}^j; j=p+1, \ldots, d\}$ it is basis for $\mathcal{D}$ such that

$$\mathcal{M} = \mathcal{W}\oplus\mathcal{D}.$$
Moreover, $V_1 =~\text{span}~\{\varphi_1, \ldots, \varphi_p\}$ and $V_2=~\text{span}~\{\varphi_{p+1}, \ldots,\varphi_{d}\}$. The loading factors associated to $V_2$ are related to the risk factors in $\mathcal{D}$ which in turn are associated to no-arbitrage restrictions.

Under the assumption that a stochastic PDE (one typical example is (\ref{bspde1})) admits a finite-dimensional realization (\ref{fdre}), we are going to present consistent estimators for the minimal invariant subspace $V$. More precisely, based on high-frequency data and techniques from factor analysis, we take advantage of the structure induced by $(\mathcal{W},\mathcal{D})$ in order to provide consistent estimators $(\hat{V}_1, \hat{V}_2)$ for $(V_1,V_2)$ related to the minimal invariant subspace $V$.

\subsection{Noise dimension vs quadratic variation dimension}\label{diffusioncomp}
It is convenient to point out that the rank of a quadratic variation matrix is \textit{not} the maximal rank of the underlying volatility process studied by Jacod and Podolskij~\cite{jacod} and Fissler and Podolskij~\cite{fissler}. See also Sahalia and Xiu \cite{sahalia2} for a similar framework. In fact, let $M$ be a $d$-dimensional It\^o process of the form

$$M_t = M_0+ \int_0^t b_sds + \int_0^t\sigma_sdB_s; 0\le t\le T.$$
Let $R_t:=\sup_{0\le s < t}~rank~(c_s); 0 < t \le T$ where $c_s:=\sigma_s \sigma^\top_s; 0\le s\le T.$
\begin{proposition}\label{diffusionmatrix}
If $\sigma$ has continuous paths, then $R_t \le rank~[M]_t~a.s$ for every $t\in [0,T]$. Moreover, the inequality may be strict.
\end{proposition}
\begin{proof}
Let us fix a realization $\omega\in \Omega$ in a set of full measure and some $t$ in $[0,T]$. Let, also, $R_t>0$.
Then, since $c_t$ is a continuous matrix-valued function, and the rank is an integer-valued lower-semicontinuous function, there exists $t^\ast\in [0,t]$
such that $rank~c_{t^\ast} = R_t$.

Since $c_{t^\ast}$ is a non-negative definite matrix, we can find a set of $R_t$ linearly independent eigenvectors for $c_{t^\ast}$,
say, $v_1,\ldots,v_{R_t}$, with respective eigenvalues $\lambda_1,
\ldots,\lambda_{R_t}$, such that $\lambda_i>0,$ for $i=1,\ldots, R_t$.

Now, observe that if $c_1,\ldots,c_{R_t}$ are real numbers such that $c_1^2+\cdots+c_{R_t}^2>0$, then by putting
$w = c_1v_1+\cdots+c_{R_t}v_{R_t}$ and using the orthogonality of the eigenvectors, we have

\begin{equation}\label{directionpos}
\langle w,c_{t^\ast}w\rangle_{\mathbb{R}^d} = c_1^2\lambda_1+\cdots+c_{R_t}^2\lambda_{R_t}>0.
\end{equation}

Note also that, for any such vector $w$, the function $t\mapsto \langle w,c_t w\rangle_{\mathbb{R}^d}$ is continuous, so we can find
an open interval $I$ containing $t^\ast$, with length $|I| = 2\delta~(\text{for some}~\delta>0)$, satisfying
\begin{equation}\label{direction1}
\forall s\in I,\quad \langle w,c_sw\rangle_{\mathbb{R}^d} > 1/2 \<w,c_{t^\ast}w\>_{\mathbb{R}^d}.
\end{equation}

Furthermore, using the non-negative definiteness of $c_s$, we have that
\begin{equation}\label{direction2}
\forall u\in [0,T],\quad \<w,c_{u}w\>_{\mathbb{R}^d} \geq 0.
\end{equation}

Now, suppose, by contradiction, that $rank~[M]_t<R_t$. Then, we can find real numbers
$c_1,\ldots,c_{R_t}$, with $c_1^2+\cdots+c_{R_t}^2>0$, such that, for $w = c_1v_1+\cdots+c_{R_t}v_{R_t}$,
where $v_1,\ldots,v_{R_t}$ are the eigenvectors of $c_{t^\ast}$ given above, we have $[M]_t w =0$, and, in particular,
$$\<w,[M]_tw\>_{\mathbb{R}^d} = 0.$$

Then, using \eqref{directionpos}, \eqref{direction1} and \eqref{direction2}, we obtain
\begin{eqnarray*}
0 = \<w,[M]_tw\>_{\mathbb{R}^d} &=& \int_0^t \<w,c_s w\>_{\mathbb{R}^d}ds\\
&\geq& \int_I \<w,c_sw\>_{\mathbb{R}^d}ds\\
&>& \int_I 1/2\<w,c_{t^\ast}w\>_{\mathbb{R}^d} ds\\
&=& \delta \<w,c_{t^\ast}w\>_{\mathbb{R}^d}  >0.
\end{eqnarray*}
This contradiction shows that $ R_t \le rank [M]_t$.

To show that the inequality may be strict, consider the following example: Let us assume that $T\ge 1$ and we take
$$\sigma_s = \begin{pmatrix}
f(s)& 0\\
0& f(s-1)
\end{pmatrix},$$
where $f(t) = t(1-t)1\!\!1_{[0,1]}$, with $1\!\!1_A$ is the indicator function of the set $A$. Then, clearly,
$$c_s = \begin{pmatrix}
f(s)^2& 0\\
0& f(s-1)^2
\end{pmatrix},$$
and $R_t=1$, for all $t>0$, whereas $rank [M]_t=2$ for $t>1$.
\end{proof}

\begin{remark}
The main message of the above proposition is that if a direction has a non-null quadratic variation for some time $t_0>0$, then this direction has non-null quadratic variation for all times $t\geq t_0$. This phenomenon does not occur with the volatility matrix
$c_s$, as shown above.
\end{remark}
We also stress that Assumptions \ref{A1} and \ref{A2} yield the study of a statistical test to check the existence of a null quadratic variation component in $\mathcal{M}$. The full derivation of the statistical test will be further explored in a future paper.

\begin{corollary}\label{test}
Let $M\in \mathcal{X}^d$ be a truly $d$-dimensional process satisfying Assumption \ref{A2}. Let $\lambda^1_T,\ldots,\lambda^d_T$ be the ordered eigenvalues of the associated quadratic variation matrix $[M]_T$ such that $\lambda^1_T\ge \ldots\ge \lambda^d_T$. The test $H_0:\lambda^d_T=0$ versus
 $H_1:\lambda^d_T>0$, is a well-defined statistical test and it is equivalent to $H_0: rank [M]_T < d$ versus $H_1: rank [M]_T =d$.
\end{corollary}

\begin{remark}\label{correlationA}
It is pertinent to interpret $\mathcal{M} = \mathcal{W}\oplus\mathcal{D}$ from the perspective of semimartingale-based factor models. When $rank~[M]_T < d$ then

$$\mathcal{M} = \mathcal{W}\oplus\mathcal{D}$$
where $dim~\mathcal{D} > 0$. In applications, one may think $\mathcal{M}$ as the space of high-dimensional portfolios composed by $M$ which can be depicted into two dynamic spaces. When $[M]_T$ is singular, then the dynamic space has to be filled with zero quadratic variation dynamics which can be neglected only if one is solely interested in volatility. We stress that this phenomena is intrinsic to the principal component analysis of high-dimensional semimartingale systems.
\end{remark}

\section{Estimation of $(\mathcal{W},\mathcal{D})$}\label{estimationofwd}
In this section, we show how to estimate the pair $(\mathcal{W},\mathcal{D})$ which realizes

$$\mathcal{M} = \mathcal{W}\oplus\mathcal{D}$$
for a given observed process $M\in\mathcal{X}^d$ satisfying Assumptions \ref{A1} and \ref{A2}. The reader may think $(\mathcal{W}, \mathcal{D})$ as a pair of \textit{factor spaces} which are not observed. We stress even if one observes all trajectories of $M$, the components of $\mathcal{D}$ are not visible when $dim~\mathcal{D}>0$.

\subsection{Identification of the Spaces $(\mathcal{W},\mathcal{D}$)}
Throughout this section, we are going to fix a truly $d$-dimensional process $M = (M^1, \ldots, M^d)\in\mathcal{X}^d$ satisfying Assumption \ref{A2}. Let $\mathcal{M} = \mathcal{W}\oplus \mathcal{D}$ be the splitting introduced in (\ref{sp1}). We assume that $dim~\mathcal{W}=p$ and $dim~\mathcal{D}=d-p$, where $1\le p\le d$. In order to clarify the exposition, we first assume that one is able to observe all trajectories of a given $M\in \mathcal{X}^d$ in continuous time.

\begin{proposition}\label{randomV}
Let $M=(M^1,\ldots, M^d)$ be a $d$-dimensional process satisfying
Assumptions \ref{A1} and \ref{A2}, $span~\{M^1,\ldots, M^d \}=~\mathcal{M}$ and let
$[M]_T$ be the quadratic variation matrix of $M$. Let $\{v_1, \ldots, v_d\}$ be an orthonormal basis formed by eigenvectors associated to the ordered
(decreasing order) eigenvalues of $[M]_T$. Let $\mathcal{V}:\Omega\rightarrow \mathbb{M}_{d\times d}$ be the random matrix given by

$$
\mathcal{V}(\omega): = \{v_{ij}(\omega); 1\le i,j\le d\}.
$$
where $v_i = (v_{i1}, \ldots,v_{id}); 1\le i\le d$. Then there exists a set $\Omega^*$ of full measure such that for each realization $\omega\in \Omega^*$, $\{(\mathcal{V}(\omega)M_{\cdot})^i;p+1\le i\le d\}$ is a basis for $\mathcal{D}$ and $\{(\mathcal{V}(\omega)M_{\cdot})^i;1\le i\le p\}$ is a basis for $\mathcal{W}$. Moreover,

\begin{equation}\label{factorranking}
[(\mathcal{V}M)^1]_T\ge [(\mathcal{V}M)^1]_T\ldots \ge [(\mathcal{V}M)^d]_T~a.s.
\end{equation}
\end{proposition}
\begin{proof}
By applying the standard spectral theorem on $[M]_T(\omega)$, we can find a set of eigenvectors $\{v_i(\omega); 1\le i\le d\}$ associated to $[M]_T(\omega)$ which constitutes an orthonormal basis for $\mathbb{R}^d$, so that $\mathcal{V}(\omega)$ is invertible for every $\omega\in \Omega$. Let $p=dim~\mathcal{W}$. If $d-p>0$ then Lemma \ref{WZ} yields $v_i\in Ker~[M]_T~a.s; p+1\le i\le d$. Therefore, $[M]_Tv_i$ is null a.s for every $i\in \{p+1, \ldots, d\}$ which implies that the last $d-p$ rows of $\mathcal{V}(\omega)\circ [M]_T(\omega)$ are null for every $\omega\in \Omega^*$ where $\Omega^*$ has full probability. Let us fix $\omega^*\in \Omega^*$ and we write $\mathcal{V}= \mathcal{V}(\omega^*)$, $v_i=v_i(\omega^*)$, $(J^1, \ldots, J^d)\in \mathcal{X}^d$, where $J^i= (\mathcal{V}M)^i; 1\le i\le d$. Now, since $\mathcal{V}$ is invertible then $\{J^1, \ldots, J^d\}$ is a linearly independent subset of $\mathcal{M}$. Moreover,

$$\sum_{j=1}^d v_{ij} \big[M^\ell,M^j\big]_T = 0~a.s, \quad \ell=1,\ldots,d; i=p+1,\ldots, d$$
which by linearity implies that

\begin{equation}\label{f}
\Big[M^\ell,\sum_{j=1}^d v_{ij}M^j\Big]_T = 0~a.s; \ell=1,\ldots,d; i=p+1,\ldots, d.
\end{equation}
More importantly, (\ref{f}) yields $\Big[ \sum_{j=1}^d v_{ij}M^j \Big]_T=0~a.s; p+1\le i\le d$. Since $span~\{J^{p+1}, \ldots, J^d\}\subset \mathcal{M}$, we actually have $\text{span}~\{J^{p+1}, \ldots, J^d\}\subset \mathcal{D}$ and the linear independence yields $span~\{J^{p+1}, \ldots, J^d\}= \mathcal{D}$. Therefore,

\begin{equation}\label{f1}
\text{span}~\{J^1, \ldots, J^d\} = \text{span} \{J^1 ,\ldots, J^p\}\oplus\mathcal{D}\subset \mathcal{M} = \mathcal{W}\oplus\mathcal{D}.
\end{equation}
Since $\{J^1\ldots, J^p\}$ is a linearly independent subset of $\mathcal{M}$, than (\ref{f1}) yields

$$span \{J^1,\ldots, J^p\} = \mathcal{W}.$$
Lastly, the ordering (\ref{factorranking}) is an immediate consequence of Proposition \ref{ranking}.
\end{proof}
With the obvious modifications, we stress the result of Proposition \ref{randomV} also holds over $[0,t]$ for every $0< t< T$.

\subsection{Estimation of the spaces $(\mathcal{W},\mathcal{D})$}
Let us suppose that we are in the same setup of the previous section, but now we have a high-frequency of observations at hand from a truly $d$-dimensional process $M=(M^1,\ldots,M^d)$ satisfying Assumption \ref{A2}. In this section, the high-frequency data is assumed to be observed at common regular times for each $M^i; i=1\ldots, d$. We leave the case of non-synchronous data to a future research. Throughout this section, we assume the existence of a consistent estimator $\widehat{[M]}_T$ for $[M]_T$ which satisfies the following assumption:

\begin{assumption}\label{C1}
$\widehat{[M]}_T$ is a sequence of non-negative definite and self-adjoint matrices such that
$\widehat{[M]}_T \stackrel{p}{\to} [M]_T$ as $\|\Pi\|\rightarrow 0$.
\end{assumption}
In the sequel, we fix $\widehat{[M]}_T$ satisfying Assumption \ref{C1} and we choose~\footnote{For instance, if $\mathbb{E}\|\widehat{[M]}_T - [M]_T\|^2_\mathbf{F}\le O(r_n)$ than choosing $\epsilon\rightarrow 0$ in such way that $\epsilon^2 (r_n)^{-1}\rightarrow \infty$ as $n\rightarrow\infty$ allows us to take $\hat{p}=\text{the number of non-zero eigenvalues of}~\widehat{[M]}_T$ bigger than $\epsilon$ as a consistent estimator.} any consistent estimator $\hat{p}$ for $rank~[M]_T$. The goal of this section is to describe a generic estimation methodology based on the existence of $\widehat{[M]}_T$ satisfying Assumption \ref{C1}. We stress the results of this section do not depend on the estimator of the quadratic variation matrix. We refer the reader to e.g \cite{mikland, wang, zeng, mancino, tao, bibinger, fan, podolskij} and other references therein for a complete view of the estimation methods for $[M]_T$.

We need to define a metric notion on the set of finite-dimensional subspaces embedded on a possibly infinite-dimensional vector space. For this task, we make use of the same metric between subspaces defined by Bathia et al.~\cite{bathia}. Let $\mathcal{N}_1$ and $\mathcal{N}_2$ be two
finite-dimensional Hilbert subspaces of an inner product vector space $H$ with dimensions $m_1$ and $m_2$, respectively. Let $\{\zeta_{i1},\ldots,\zeta_{im_i}\}$ be an orthonormal basis of $\mathcal{N}_i$, $i=1,2$. Then, we define

\begin{equation}\label{metricsub}
D(\mathcal{N}_1,\mathcal{N}_2) := \sqrt{1-\frac{1}{\max\{m_1,m_2 \}}\sum_{k=1}^{m_1}\sum_{j=1}^{m_2} (\langle\zeta_{2j},\zeta_{1k}\rangle_{H})^2}.
\end{equation}
In the sequel, we need to compute distances for finite-dimensional subspaces which are not embedded in a natural common Hilbert space. For this reason, let $A$ be a finite-dimensional linear space. If $A_1$ and $A_2$ are finite-dimensional subspaces of $A$, then we define

\begin{equation}\label{metricsub1}
d(A_1,A_2):=D(\Phi(A_1), \Phi(A_2))
\end{equation}
where $\Phi:A\rightarrow \mathbb{R}^m;~i=1,2$ is the canonical isomorphism and $dim~A= m$. One can easily check that $d$ is indeed a metric over the set of all finite-dimensional subspaces of $A$. The metric $d$ in~(\ref{metricsub1}) is very convenient to study consistency of subspace estimators.

Before presenting the main result of this section, we need two preliminary lemmas.

\begin{lemma}\label{conv-esp-D}
Let $C_n,C:\Omega\rightarrow \mathbb{M}_{d\times d}$ be a sequence of self-adjoint real $d\times d$ matrices such that $C_n\stackrel{p}{\to} C$ as $n\rightarrow \infty$. Assume that $q=\dim Ker(C)~a.s$ and let us denote by ${v_1^n},\ldots,{v}_q^n$ a set of orthonormal eigenvectors associated to the $q$ least eigenvalues of $C_n$. Let $K_n = \text{span}~\{v_1^n,\ldots,v_q^n\}$ and $K = Ker(C)$. Then,

 $$D(K_n,K)\stackrel{p}{\to} 0$$
as $n\rightarrow \infty$.
\end{lemma}
\begin{proof}
Let$\{v_i\}_{i=1}^d$ be an orthonormal basis for $\mathbb{R}^d$ given by eigenvectors of $C$. Let $v_1,\ldots,v_q$ be an orthonormal subset of eigenvectors of $C$ associated to eigenvalues $\alpha_1,\ldots,\alpha_q$ and $Ker(C)=\text{span}~\{v_1,\ldots,v_q\}$. To shorten notation, in the sequel we denote by $\langle \cdot, \cdot\rangle=\|\cdot\|^{1/2}$ the inner product over Euclidean spaces. We may assume that $0 < q < d$. Let $\{v_{q+1},\ldots,v_d\}$ be a basis for the orthogonal complement $K^\perp$. At first, we notice that since $K_n$ and $K$ have the same dimension, it is sufficient to prove that $D(K_n,K^\perp) \stackrel{p}\to 1$. This is equivalent to prove that

$$\sum_{j=1}^q \sum_{i=1}^{d-q} (\<v_j^n,v_{q+i}\>)^2\stackrel{p}\to 0\quad {as}~n\rightarrow \infty.$$
To do so, let $q_{i,j} = \<v_j^n,v_{q+i}\>v_{q+i}$, and note that $\|q_{i,j}\|\leq 1~a.s$ and $Cq_{i,j} = \<v_j^n,v_{q+i}\>\alpha_{q+i}v_{q+i}$.
Therefore,
$$\<Cq_{i,j},v_j^n\> = \alpha_{q+i}(\<v_j^n,v_{q+i}\>)^2 \Rightarrow \sum_{i=1}^d\sum_{j=1}^{d-q} \<Cq_{i,j},v_j^n\> = \sum_{i=1}^d\sum_{j=1}^{d-q}\alpha_{q+i}(\<v_j^n,v_{q+i}\>)^2,$$
and since $\sum_{i,j}\alpha_{q+i}(\<v_j^n,v_{q+i}\>)^2\geq \alpha_{q+1}\sum_{i,j}(\<v_j^n,v_{q+i}\>)^2~a.s$ we may conclude that
\begin{eqnarray*}
\sum_{i=1}^d\sum_{j=1}^{d-q}(\<v_j^n,v_{q+i}\>)^2 &\leq& \frac{1}{\alpha_{q+1}} \sum_{i=1}^d\sum_{j=1}^{d-q}\<v_j^n,Cq_{i,j}\>\\
&=& \frac{1}{\alpha_{q+1}} \sum_{i=1}^d\sum_{j=1}^{d-q} \<q_{i,j},Cv_j^n\>\\
&\leq& \frac{1}{\alpha_{q+1}} \sum_{i=1}^{d}\sum_{j=1}^{d-q} \|q_{i,j}\|\cdot\|Cv_j^n\|\\
&\leq& \frac{1}{\alpha_{q+1}} \sum_{i=1}^d \sum_{j=1}^{d-q}\|Cv_j^n\|~a.s~\forall n\ge 1.
\end{eqnarray*}

We now claim that
\begin{equation}\label{mconv}
\sup_{\substack{v\in K_n\\ \|v\|=1}} \|C v\|\rightarrow  0.
\end{equation}
Let $\alpha^n_1\ge \alpha^n_2\ge \ldots\ge \alpha^n_q$ be the ordered eigenvalues of $C_n$ related to the $q$ least eigenvalues. Let $\gamma_n$ be the number of non-zero eigenvalues of $C_n$. We have $\mathbb{P}\{\gamma_n = d-q\} =1$ for every $n$ sufficiently large so that

$$\sup_{\substack{v\in K_n\\ \|v\|=1}} \|C_n v\| \leq \alpha^n_{1}\stackrel{p}{\to} 0$$
as $n\rightarrow \infty$.

On the other hand, $C_n\stackrel{p}{\to} C$ as $n\rightarrow \infty$ and hence

$$\sup_{\substack{v\in\mathbb{R}^p\\ \|v\| =1}} \| C_n v - Cv\| \stackrel{p}{\to} 0$$
as $n\rightarrow \infty$.
Therefore, triangle inequality yields
\begin{eqnarray*}
\sup_{\substack{v\in K_n \\ \|v\|=1}} \|C v\| &\leq& \sup_{\substack{v\in K_n\\ \|v\| =1}} \| C_n v - Cv\| + \sup_{\substack{v\in K_n\\ \|v\| =1}} \|C_n v\|\\
& &\\
&\leq& \sup_{\substack{v\in\mathbb{R}^p\\ \|v\| =1}} \| C_nv - Cv\| + \sup_{\substack{v\in K_n\\ \|v\|=1}} \|C_nv\|\\
& &\\
&\stackrel{p}{\to}& 0
\end{eqnarray*}
as $n\rightarrow \infty$. This shows (\ref{mconv}) and we may conclude the proof.
\end{proof}

\begin{lemma}\label{deterministickernel}
Let $M\in \mathcal{X}^d$ be a truly $d$-dimensional process satisfying Assumption \ref{A2}. Then, the set $\ker [M]_t$ is deterministic~\footnote{A random set $A$ is deterministic if there exists a subset $A\subset \mathbb{R}^d$ such that $A=B$~a.s.} for every $t\in [0,T]$.
\end{lemma}
\begin{proof}
For $t=0$ the statement is obvious, so let us fix $t\in (0,T]$ and let $\mathcal{D}_t$ be the subspace of $\mathcal{M}_t$ given by (\ref{Dt}). Let $p_t$ be the dimension of $\widetilde{\mathcal{M}}_t$. Let $N^1,\ldots,N^{d-p_t}$ be a basis of $\mathcal{D}_t$ and let $R^1,\ldots,R^{p_t}$ be a complement basis of $\mathcal{M}_t$ in such a way that $\{N^1,\ldots,N^{d-p_t},R^1,\ldots,R^{p_t}\}$ is a basis of $\mathcal{M}_t$. Let $A$ be the change of basis from $\{N^1,\ldots,N^{d-p_t},R^1,\ldots,R^{p_t}\}$ to $M = \{M^1, \ldots, M^d\}$ with matrix representation $A=\{(a_{ij})_{1\leq i,j\leq d}\}$. We set
$\Omega^*:= \Omega - O$ where $O:=\big\{\omega; rank~[M]_t(\omega)\neq p_t~\text{or}~[N^\ell]_t(\omega)>0~\text{for some}~\ell \in \{1, \ldots, d-p_t\}\big\}$. From Lemma \ref{WZ} and the definition of $\mathcal{D}_t$, we know that $\Omega^*$ has full probability. We pick $\omega\in \Omega^*$. Of course,

$$a_1 := (a_{11},\ldots,a_{d1}),\ldots,a_{d-p_t}:=(a_{1(d-p_t)},\ldots,a_{d(d-p_t)})$$
constitutes a set of $d-p_t$ linearly independent deterministic vectors in $\mathbb{R}^d$ and by the every definition

$$[M]_t(\omega)a_\ell = \sum_{k=1}^d a_{k\ell }[M^i, M^k]_t(\omega) = [M^i, N^\ell]_t(\omega)=0 $$
for $1\le \ell \le d-p_t, 1\le i\le d$.
Since $ker [M]_t(\omega)\subset \mathbb{R}^d$ has dimension $d-p_t$ for every $\omega\in \Omega^*$, then $ker [M]_t(\omega) = \text{span}~\{a_1, \ldots,a_{d-p_t}\}$ for every $\omega\in \Omega^*$.
\end{proof}

Let $\hat{\mathcal{V}}$ be the orthogonal matrix formed by orthonormal eigenvectors of $\widehat{[M]}_T$. Of course, we are not able to prove that $\widehat{\mathcal{V}}M$ converges to $\mathcal{V}M$ due to the lack of identification of eigenvectors. What is true is the following notion of convergence. In the sequel, if $\{A_n,B_n; n\ge 1\}$ is a sequence of random variables, then

$$A_n\succeq Bn\quad\text{as}~n\rightarrow \infty$$
means that, $\mathbb{P} (A_n < B_n)\rightarrow 0$ as $n\rightarrow \infty$. We similarly define $\preceq$ and $A_n\simeq B_n$ when both $A_n\succeq B_n$ and $A_n\preceq B_n$ as $n\rightarrow \infty$.

\begin{theorem}\label{WDTh}
Let $M = (M^1,\ldots,M^d)$ be a process satisfying Assumptions \ref{A1} and \ref{A2}. Let $\widehat{[M]}_T$ be a consistent estimator for $[M]_T$ satisfying Assumption \ref{C1} and let $\hat{p}$ be any consistent estimator for $\text{rank}~[M]_T$. Let $\widehat{\mathcal{V}}$ be the orthogonal matrix whose rows are formed by eigenvectors of
$\widehat{[M]}_T$. If $(\hat{J}^1_{\cdot},\ldots,\hat{J}^d_{\cdot}) := \widehat{\mathcal{V}}M_\cdot$, then let us define $\widehat{\mathcal{W}} := span~\{\hat{J}^{1},\ldots,\hat{J}^{\hat{p}}\}$ and $\widehat{\mathcal{D}} := span~\{\hat{J}^{\hat{p}+1},\ldots,\hat{J}^{d}\}$. Under the above conditions, we have

$$d(\widehat{\mathcal{W}},\mathcal{W})\stackrel{p}{\to}0\hbox{~~and~~}d(\widehat{\mathcal{D}},\mathcal{D})\stackrel{p}{\to} 0,$$
as $\|\Pi\|\rightarrow 0$. If $\widehat{\mathcal{M}}:=\widehat{\mathcal{W}}\oplus \widehat{\mathcal{D}}$ then $d(\widehat{\mathcal{M}},\mathcal{M})\stackrel{p}{\to}0$ as $\|\Pi\|\rightarrow 0$. Moreover,

\begin{equation}\label{r1}
[\hat{J}^1]_T\succeq\ldots\succeq[\hat{J}^{\hat{p}}]_T
\end{equation}

\begin{equation}\label{r2}
[\hat{J}^{i}]_T\simeq 0;\quad \hat{p}\le i\le d~\text{as}~\|\Pi\|\rightarrow 0.
\end{equation}
\end{theorem}
\begin{proof}
Recall the definition of the isomorphism $\Phi$ used in (\ref{metricsub1}). From Lemma \ref{deterministickernel}, we have
$\Phi(\mathcal{D}) = Ker([M]_T)$ and by the very definition of $\widehat{\mathcal{D}}$, we also have $\Phi(\widehat{\mathcal{D}}) = Ker(\widehat{[M]}_T)$. Thus, from Lemma \ref{conv-esp-D} above, we have

$$d(\mathcal{D},\widehat{\mathcal{D}})\stackrel{p}{\longrightarrow} 0.$$

Now, notice that
$$\mathbb{R}^d = \Phi(\mathcal{D})\oplus \Phi(\mathcal{W}) = \Phi(\widehat{\mathcal{D}})\oplus\Phi(\widehat{\mathcal{W}}).$$
Therefore, it follows from the definition of the metric $d$ that
$$d(\widehat{\mathcal{W}},\mathcal{W})\stackrel{p}{\longrightarrow} 0.$$
Since $\hat{p}$ is an integer-valued consistent estimator, we shall assume that $\hat{p}=p$. By the very definition, we know that

$$
\langle \hat{v}^{i}_T, \widehat{[M]}_T\hat{v}^{i}_T\rangle_{\mathbb{R}^d} \ge\langle \hat{v}^{i+1}_T, \widehat{[M]}_T\hat{v}^{i+1}_T\rangle_{\mathbb{R}^d}~a.s; 1\le i\le d-1.
$$
and
$$
[\hat{J}^i]_T = \langle \hat{v}^{i}_T, [M]_T \hat{v}^i_T\rangle_{\mathbb{R}^d} ~a.s; 1\le i\le d.
$$
Let us write

\begin{eqnarray*}
[\hat{J}^i]_T - [\hat{J}^{i+1}]_T &=& \big([\hat{J}^i]_T  -
\langle \hat{v}^{i}_T, \widehat{[M]}_T\hat{v}^{i}_T\rangle_{\mathbb{R}^d}\big) + \big(\langle \hat{v}^{i}_T, \widehat{[M]}_T\hat{v}^{i}_T\rangle_{\mathbb{R}^d} - \langle \hat{v}^{i+1}_T, \widehat{[M]}_T\hat{v}^{i+1}_T\rangle_{\mathbb{R}^d}\big)\\
& &\\
&+& \big(\langle \hat{v}^{i+1}_T, \widehat{[M]}_T\hat{v}^{i+1}_T\rangle_{\mathbb{R}^d}  - [\hat{J}^{i+1}]_T\big); 1 \le i\le d-1.
\end{eqnarray*}
By construction, $max_{1\le i\le d}|\hat{v}^i_T|$ is bounded in probability and $\|\widehat{[M]}_T - [M]_T\|_{\mathbf{F}}\rightarrow 0$ in probability as $\|\Pi\|\rightarrow 0$. Moreover, $\big(\langle \hat{v}^{i}_T, \widehat{[M]}_T\hat{v}^{i}_T\rangle_{\mathbb{R}^d} - \langle \hat{v}^{i+1}_T, \widehat{[M]}_T\hat{v}^{i+1}_T\rangle_{\mathbb{R}^d}\big)\ge 0~a.s$ and hence (\ref{r1}) holds true. The proof of (\ref{r2}) is similar.
\end{proof}

A straightforward consequence is the following result.

\begin{corollary}\label{WDCor}
Assume that hypotheses in Theorem \ref{WDTh} hold and let $Y\in \mathcal{M}$ be discretely-observed at $\{Y_{t_r}; 0\le r\le n\}$ over $[0,T]$, where $0=t_0 < t_1 \ldots< t_n=T$. Then, there exists $\alpha=(\alpha^{1}, \ldots, \alpha^{d})\in\mathbb{R}^d$ such that

\begin{equation}\label{olsconv}
max_{0\le r\le n}\Big|Y_{t_r} - \sum_{\ell=1}^{\hat{p}}\alpha^{\ell}\hat{J}^{\ell}_{t_r} - \sum_{k={\hat{p}+1}}^{d}\alpha^{k}\hat{J}^{k}_{t_r}\Big|\stackrel{p}{\to}0,
\end{equation}
as $max_{1\le i\le n}|t_r-t_{r-1}|\rightarrow 0$.
\end{corollary}
\begin{proof}
Let us equip $\mathcal{X}$ with the topology of the uniform convergence in probability. Let $\mathcal{H}$ be the smallest finite-dimensional subspace of $\mathcal{X}$ which contains $\{M^1,\ldots, M^d;\hat{J}^1, \ldots, \hat{J}^d\}$. Let $\Phi:\mathcal{H}\rightarrow \mathbb{R}^{m}$ be the canonical isomorphism for some $m>0$. We notice that $\Phi$ is actually an homeomorphism when $\mathcal{H}$ is endowed with the subspace topology. From Theorem \ref{WDTh} and the definition of the metric $d$, we know that

\begin{equation}\label{progine}
d(\mathcal{M},\widehat{\mathcal{M}}) = D(\Phi(\mathcal{M}),\Phi(\widehat{\mathcal{M}})) = \sqrt{2d}\sup_{\|v\|_{\mathbb{R}^d}=1} \|\mathcal{T}_{\Phi(\mathcal{M})}v - \mathcal{T}_{\Phi(\widehat{\mathcal{M}})}v\|_{\mathbb{R}^d}\stackrel{p}{\to}0
\end{equation}
as $\|\Pi\|\rightarrow 0$, where $\mathcal{T}_A$ denotes the projection onto a closed subspace $A\subset \mathbb{R}^d$. Then from (\ref{progine}) and using the fact that $\Phi$ is an homeomorphism, we get the existence of $\alpha=(\alpha^{1}, \ldots, \alpha^{d})\in\mathbb{R}^d$ such that

$$\Big|\Phi(Y) - \sum_{\ell=1}^{\hat{p}}\alpha^{\ell}\Phi(\hat{J}^{\ell}) - \sum_{k={\hat{p}+1}}^{d}\alpha^{k}\Phi(\hat{J}^{k})\Big|\stackrel{p}{\to}0$$
as $\|\Pi\|\rightarrow 0$ which implies assertion in (\ref{olsconv}).
\end{proof}

Under the assumptions of Theorem \ref{WDTh}, if $Y\in \mathcal{M}$ is a discretely-observed semimartingale at $\{Y_{t_k}; 1\le k\le n\}$ over $[0,T]$, then we shall use Corollary \ref{WDCor} to estimate by OLS

$$
\hat{\alpha} :=  \underset{ \alpha\in \mathbb{R}^d}{\operatorname{argmin}}\sum_{\ell=1}^n \Big|Y_{t_\ell} - \hat{\mathcal{V}}M_{t_\ell}\cdot \alpha^\top\Big|^2,
$$
the regression coefficients which provide us the precise linear contribution of non-null quadratic variation and pure drift components in $\mathcal{W}$ and $\mathcal{D}$, respectively. In this case, the following linear combination

$$
\hat{Y}_{k} := \sum_{\ell=1}^{\hat{p}} \hat{\alpha}^{\ell}\hat{J}^{\ell}_{t_k} + \sum_{r=\hat{p}+1}^d \hat{\alpha}^{r} \hat{J}^{r}_{t_k}; i=1\ldots, d, k=0, \ldots, n.
$$
depicts $\{Y_{t_r}; 0\le r\le n\}$ into elements of $\widehat{\mathcal{W}}\oplus \widehat{\mathcal{D}}$ over the sample $\{Y_{t_k}; 0\le k\le n\}$ in $[0,T]$. The estimation of the factor spaces $(\mathcal{W},\mathcal{D})$ provides a tool to optimal asset allocation/dimension reduction in high-dimensional portfolios composed by semimartingales, a topic which will be further explored in a future paper.

\section{Estimation of Finite-Dimensional Invariant Manifolds}\label{estimationFDIM}
In this section, we apply the theory developed in previous sections to present a methodology for the estimation of finite-dimensional invariant manifolds related to space-time data generated by stochastic PDEs of the form

\begin{equation}\label{spde1}
dr_t = \big(A(r_t) + F(r_t)\big)dt + \sum_{j=1}^m\sigma^j(r_t)dB^j_t;t\ge 0;r_0=h\in E,
\end{equation}
where $A$ is an infinitesimal generator of a $C_0$-semigroup on a separable Hilbert space $E$ which we assume to be a subspace of absolutely continuous functions $g:K\rightarrow \mathbb{R}$ where for simplicity of exposition we work with the one-dimensional space\footnote{Indeed, it is not too difficult to extend the results of this section to the multi-dimensional case where $K$ is a compact subset of $\mathbb{R}^n$. This type of flexibility is important to treat more complex space-time data such as volatility surfaces in Financial Engineering.} set $K=[a,b]$ where $-\infty< a\le x \le b < +\infty$. The vector fields $F,\sigma^i; i=1,\ldots, m$ are assumed to be Lipschitz and the dimension $m$ is fixed.

\subsection{Splitting the invariant manifold}
Let us now introduce the basic geometric objects related to the stochastic PDE~(\ref{spde1}) that we are interested in estimating. We refer the reader to Tappe~\cite{tappe1} for a very clear treatment of these objects.
\begin{definition}
A family $(\mathcal{V}_t)_{t\ge 0}$ of affine manifolds in $E$ is called a foliation generated by a finite-dimensional subspace $V\subset E$ if there exists $\phi\in C^1(\mathbb{R}_+; E)$ such that

$$\mathcal{V}_t= \phi(t) + V;~t\ge 0.$$
The map $\phi$ is a parametrization of $(\mathcal{V}_t)_{t\ge 0}$.
\end{definition}

\begin{remark}
We notice that the parametrizations of $(\mathcal{V}_t)_{t\ge 0}$ are not unique, but for any distinct parametrizations $\phi^1$ and $\phi^2$ we have $\phi^1(t)-\phi^2(t)\in V$ for every $t\in [0,T]$.
\end{remark}

In the remainder of this paper, $(\mathcal{V}_t)_{t\ge 0}$ denotes a foliation generated by a finite-dimensional subspace.
\begin{definition}
The foliation $(\mathcal{V}_t)_{t\ge 0}$ of affine manifolds is invariant w.r.t the stochastic PDE~(\ref{spde1}) if for every $t_0\in \mathbb{R}_+$ and $h\in \mathcal{V}_{t_0}$ we have

$$\mathbb{P}\{r_t\in \mathcal{V}_{t_0+t},~\hbox{for all}~ t\ge 0\}=1$$
for $r_0=h$.
\end{definition}

The above objects lead us to the following definition which is the main object of statistical study in this section.

\begin{definition}
We say that the stochastic PDE~(\ref{spde1}) has an affine realization generated by a finite-dimensional subspace $V\subset E$ if for each $h_0\in dom~(A)$ there exists a foliation $(\mathcal{V}^{h_0}_t)_{t\ge 0}$ generated by $V$ with $h_0\in \mathcal{V}^{h_0}_0$ which is invariant w.r.t~(\ref{spde1}). An affine realization with a generator $V$ is called minimal, if for another affine realization generated by some subspace $W$ we have $V\subset W$.
\end{definition}

\begin{remark}
Suppose that the stochastic PDE~(\ref{spde1}) has an affine realization generated by a subspace $V$. We recall that for each $h_0\in dom~(A)$ the foliation $(\mathcal{V}^{h_0}_t)_{t\ge 0}$ generated by $V$ is uniquely defined. See e.g~[Lemma 2.7~\cite{tappe1}].
\end{remark}
See Section \ref{expSPDE} for a brief discussion on affine realizations in the context of Mathematical Finance. Throughout this paper, we assume that the stochastic PDE data generating process satisfies the following assumption.

\

\noindent \textbf{Assumption (A1):} The stochastic PDE~(\ref{spde1}) has an affine realization generated by a finite-dimensional subspace.

\

Let us now introduce the basic operators which will encode the underlying loading factors of the stochastic PDE that we are interested in estimating. We fix once and for all a terminal time $0< T < \infty$, $r_0\in dom~(A)$, the minimal subspace generator $V$ of (\ref{spde1}) spanned by linearly independent vectors $\{w_1,\ldots,w_d \}$  and a parametrization $\phi\in C([0,T];E)$ with null quadratic variation $[\phi(u)]_T=0;~u\in [a,b]$. Under Assumption~(A1), the stochastic PDE~(\ref{spde1}) has a strong solution. From the reproducing kernel property of $E\subset C([a,b];\mathbb{R})$, the evaluation map $\tau_u: f\mapsto f(u)$ is a bounded linear functional and therefore point-wise evaluation of the stochastic PDE is well-defined for every point-space and the following representation holds

\begin{equation}\label{pointspde}
r_t(u) = r_0(u) + \int_0^t \big(A(r_s)(u) + F(r_s)(u)\big)ds + \sum_{i=1}^m\int_0^t\sigma^i(r_s)(u)dB^i_s,
\end{equation}

\noindent where we set $r_t(u):=\tau_ur_t$ for $0\le t\le T$ and $u\in [a,b]$. Let us consider the following kernels

$$\sigma_t(u,v):=\sum_{j=1}^m\sigma^j(r_t)(u)\sigma^j(r_t)(v);0\le t\le T,$$

$$Q_{T}(u,v):=[r(u),r(v)]_T=\int_0^{T}\sigma_s(u,v)ds,~u,v\in [a,b].$$

The above kernels induce random linear operators $Q_{T}$ and $\sigma_t$ defined almost everywhere by
$$(Q_{T} f)(\cdot) := \langle Q_{T}(\cdot, ), f \rangle_E; f\in E.$$

$$\sigma_t f(\cdot):= \langle \sigma_t(\cdot,),f\rangle_E;~f\in E, 0\le t\le T.$$

\noindent By the very definition, the random linear operator $Q_{T}$ can be written as
$$
(Q_{T}f)(u) = \int_0^{T}(\sigma_sf)(u)ds;~f\in E.
$$
where we denote $\mathcal{Q}:=Range~Q_T$. In the remainder of this article, we denote by $\mathcal{N}$ the supplementary subspace of $\mathcal{Q}$ in the minimal subspace $V$.

From~Assumption (A1), we know (see e.g Th. 2.11 and (2.27) in \cite{tappe1}) that there exists a truly $d$-dimensional semimartingale $Z=(Z^1,\ldots, Z^p)$ which realizes the strong solution~(\ref{pointspde}) as follows

\begin{equation}\label{ss}
r_t(u)=\phi_t(u) + \sum_{i=1}^dZ^i_t w_i(u);~0\le t\le T,u\in [a,b].
\end{equation}

\begin{definition}
We say that the stochastic PDE in~(\ref{spde1}) admits a finite-dimensional realization (FDR) if for each $h\in dom~(A)$ there exists a truly $d$-dimensional semimartingale $Z\in \mathcal{S}^d$, a parametrization $\phi\in C([0,T];E)$ and a linearly independent set $\{w_1\ldots, w_d\}\subset E$ which realize~(\ref{ss}).
\end{definition}

See~e.g~\cite{bjork1,tappe1,filipo,filipo5} for more details on this affine construction of the stochastic PDE. Representation~(\ref{ss}) is not unique but it will be the basis for our splitting scheme as follows. At first, in order to apply the spectral analysis in previous sections, we will assume the following hypothesis on the stochastic PDE (\ref{spde1}):

\

\noindent \textbf{Assumption (A2):} For each initial condition $h\in dom~(A)$, there exists a factor representation $Z$ which realizes (\ref{ss}) and it satisfies Assumption \ref{A2}.

\

In the sequel, if $L\in \mathbb{M}_{d\times d}$ and $\eta=(\eta_1,\ldots,\eta_d)$ is a list of real-valued functions on $[a,b]$, then $\eta(x) = (\eta_1(x), \ldots, \eta_d(x))\in \mathbb{M}_{d\times 1}$ and we set $L\eta$ meaning the $\mathbb{R}^d$-valued function $x\mapsto L\eta(x)$.

\begin{remark}\label{randomrep1}
Let $r_t(u)=\phi_t(u) + \sum_{i=1}^dZ^i_t w_i(u);~0\le t\le T,u\in [a,b]$ be a representation of the FDR of (\ref{spde1}). Let $\mathcal{A}\in \mathbb{M}_{d\times d}$ be a non-singular random matrix. Then

\begin{equation}\label{randomrep2}
r_t(x) = \phi_t(x) + \sum_{j=1}^d Y^j_t \varphi_j(x); 0\le t\le T, x\in [a,b]
\end{equation}
where $\varphi = (\mathcal{A}^{-1})^\top w$ is a random basis for $V$ and $Y_\cdot=\mathcal{A}Z_\cdot\in \mathcal{X}^d$.
\end{remark}

We can actually write $Q_{T}$ in terms of any representation~(\ref{ss}) as follows

\begin{equation}\label{ss1}
(Q_Tf)(u)=\sum_{i,j=1}^d\langle f,w_{i}\rangle_E w_{j}(u)[Z^i, Z^j]_T;~f\in E;u\in [a,b],
\end{equation}
and, moreover, the following remark holds.

\begin{remark}\label{invprop}
From Lemma \ref{WZ}, one can easily see that under Assumption (A2), any truly $d$-dimensional factor process realizing (\ref{ss}) (or (\ref{randomrep2})) will satisfy Assumption \ref{A2}.
\end{remark}

In the sequel, we need to introduce new notation. For a given $Z\in \mathcal{X}^d$ satisfying Assumptions \ref{A1} and \ref{A2}, we denote $\mathcal{M}(Z):=\text{span}~\{Z^1,\ldots, Z^d \}$, $\widetilde{\mathcal{M}}(Z):=\mathcal{M}(Z)/\mathcal{D}(Z)$ where $\mathcal{D}(Z):=\{X\in \mathcal{M}(Z); [X]_\cdot=0~a.s ~\text{on}~[0,T]\}$ and the quotient space is defined by the equivalence relation (\ref{eqrel}) over $[0,T]$. We stress that $\mathcal{M}(Z), \mathcal{D}(Z)$ and $\widetilde{\mathcal{M}}(Z)$ are $\mathcal{M}, \mathcal{D}$ and $\widetilde{\mathcal{M}}$, respectively, which are defined in (\ref{quot}) for the specific choice $M=Z$.

In practice, we are not able to observe any semimartingale factor $Z=(Z^1,\ldots, Z^d)$ of a stochastic PDE admiting a FDR. But it will be very important for our estimation strategy to identify the pair $(\mathcal{Q}, \mathcal{N})$ in terms of the random matrix $[Z]_T$, or more precisely, in terms of the quadratic variation of random rotations of $Z$. Next, we recall the following result.

\begin{lemma}\label{Qoperator}
Let $r$ be the stochastic PDE (\ref{spde1}) satisfying Assumptions (A1-A2) and admitting a FDR generated by the minimal foliation $\mathcal{V}^h_t =\{\phi_t + V\};0\le t\le T$ where $dim~V=d$ and $r_0=h$. Then, we shall represent (\ref{spde1}) as follows


\begin{equation}\label{goodrepr}
r_t = \phi_t + \sum_{i=1}^p Y^i_t\varphi_i + \sum_{j=p+1}^d Y^j_t\varphi_{j}; 0\le t\le T,
\end{equation}
where $Y$ is a truly $d$-dimensional semimartingale $Y$ satisfying $\mathcal{W}(Y) = \text{span} \{Y^1,\ldots, Y^p\}$, $\mathcal{D}(Y) = \text{span} \{Y^{p+1}, \ldots, Y^d\}$ and $V=\mathcal{Q}\oplus \mathcal{N}$, where $\mathcal{Q}=\text{span}~\{\varphi_{1},\ldots, \varphi_{p}\}$ and $\mathcal{N}=\text{span}~\{\varphi_{p+1},\ldots, \varphi_{d}\}$.
\end{lemma}
\begin{proof}
By assumption, there exists a truly $d$-dimensional semimartingale $Z = (Z^ 1, \ldots, Z^d)$ satisfying Assumption \ref{A2} and a basis $w = \{w_i\}_{i=1}^d$ for $V$ such that

$$r_t = \phi_t+ \sum_{i=1}^d Z^i_t w_i; 0\le t\le T.$$

From (\ref{ss1}), we have $\mathcal{Q}\subset V~a.s$ so that we shall consider the random operator $Q_T$ restricted to $V$ as follows $Q_T: \Omega\times V\rightarrow V$. Moreover, from (\ref{ss1}) we readily see that the random matrix of the linear operator $Q_T$ is given by $\{[Z^i,Z^j]_T; 1\le i, j\le d\}$ for any pair $(Z,w)$ of latent semimartingale representation $Z$ and a basis $w$ for $V$. By Lemma \ref{WZ}, we have $dim~\mathcal{Q} = dim~\widetilde{\mathcal{M}}(Z)~a.s$. Let $Y = \{Y^1, \ldots,Y^d\}$ be a truly $d$-dimensional semimartingale such that $\{Y^1, \ldots, Y^p \}$ is a basis for $\mathcal{W}(Z)$ and $\{Y^{p+1}, \ldots, Y^d\}$ is a basis for $\mathcal{D}(Z)$ where $p=dim~\mathcal{Q}$. Then $\text{span} \{Y^1, \ldots, Y^d\}=\mathcal{M}(Z)$ and $Y$ satisfies Assumptions \ref{A1} and \ref{A2}. Let $I:\mathcal{M}(Z)\rightarrow \mathcal{M}(Z)$ be the linear isomorphism given by the change of basis from $Z$ to $Y$. If $[I]^Z_Y = \{a_{ij}; 1\le i,j\le d\}$ is the matrix of $I$, then we shall write

\begin{equation}\label{secrep}
r_t = \phi_t + \sum_{i=1}^dY^i_t \varphi_i ; 0\le t\le T,
\end{equation}
where $\varphi_j := \sum_{i=1}^da_{ij}w_i; 1\le j\le d$. By writing $Q_T$ in terms of the basis $\{\varphi_j\}_{j=1}^d$ and using (\ref{secrep}), we clearly see that $\mathcal{Q}= \text{span}~\{\varphi_1, \ldots, \varphi_p\}$. By taking $\mathcal{N}=\text{span} \{\varphi_{p+1}, \ldots, \varphi_d\}$, we then conclude (\ref{goodrepr}).
\end{proof}

The main message of Lemma \ref{Qoperator} is the following. When the stochastic PDE is projected onto $\mathcal{Q}$ ($\mathcal{N}$), then the associated latent factors are non-null quadratic variation (bounded variation) semimartingales. We remark that the form of the FDRs (\ref{goodrepr}) has already been derived in Bjork and Land\'en \cite{bjork1} and Filipovic and Teichmann \cite{filipo5} in the context of HJM models. Lemma \ref{Qoperator} provides an explicit splitting for $V$ by separating the loading factors which generate $\mathcal{Q}$ from its complementary subspace $\mathcal{N}$ attached to their associated spaces $\mathcal{W}(Y)$ and $\mathcal{D}(Y)$, respectively.

Summing up the above results, we arrive at the following identification result.
\begin{proposition}\label{mainREP}
Let $r$ be the stochastic PDE (\ref{spde1}) satisfying Assumptions (A1-A2). For a given $h\in dom~(A)$, let $\mathcal{V}^h_t = \phi_t + V; 0\le t\le T$ be the minimal foliation generated by some $V$ such that $r_0=h \in \mathcal{V}^h_0$. Let

$$r_t = \phi_t + \sum_{i=1}^d Z^i_t\eta_i; 0\le t \le T,$$
be a factor semimartingale representation, where $V =~\text{span}~\{\eta_1, \ldots, \eta_d\}$ and $Z$ satisfies Assumptions \ref{A1} and \ref{A2}. Let $\mathcal{A}\in\mathbb{M}_{d\times d}$ be a nonsingular random matrix. Let $\varphi(x) = (\mathcal{A}^{-1})^\top\eta(x); x\ge 0$ and $Y_t =\mathcal{A}Z_t; 0\le t\le T$. Let $\mathcal{L}:\Omega\rightarrow \mathbb{M}_{d\times d}$ be the random matrix whose rows are given by
$\mathcal{L}_i = v_i; 1\le i\le d$ where $\{v_1, \ldots, v_d\}$ is an orthonormal eigenvector set of $[Y]_T$ associated to the ordered eigenvalues $q_1\ge q_2\ge \ldots \ge q_d~a.s$. Then

\begin{equation}\label{goodN1}
\mathcal{Q}=~\text{span}~\Big\{(\mathcal{L}\varphi)_{1},\ldots, (\mathcal{L}\varphi)_{p}\Big\}~a.s, \mathcal{N}=~\text{span}~\Big\{(\mathcal{L}\varphi)_{p+1},\ldots, (\mathcal{L}\varphi)_{d} \Big\}~a.s,
\end{equation}
and
\begin{equation}\label{goodN2}
 \mathcal{W}(Y)=~\text{span}~\Big\{(\mathcal{L}Y)^{1},\ldots, (\mathcal{L}Y)^{p}\Big\}, \mathcal{N}(Y)=~\text{span}~\Big\{(\mathcal{L}Y)^{p+1},\ldots, (\mathcal{L}Y)^{d} \Big\}.
\end{equation}
\end{proposition}
\begin{proof}
This is a straightforward consequence of Proposition \ref{randomV}, Lemma \ref{Qoperator} and the identity
$$\langle Z_t,\eta(x) \rangle_{\mathbb{R}^d} =\langle \mathcal{A}Z_t,(\mathcal{A}^{-1})^\top\eta(x) \rangle_{\mathbb{R}^d} = \langle \mathcal{L}\mathcal{A}Z_t,\mathcal{L}(\mathcal{A}^{-1})^\top\eta(x) \rangle_{\mathbb{R}^d},$$
$0\le t\le T, x\ge 0$ due to the orthogonality of the random matrix $\mathcal{L}$.
\end{proof}

\subsection{Preliminaries on Factor models}\label{prpca}
The goal of this section is to describe an estimation methodology for the pair $(\mathcal{Q},\mathcal{N})$ which generates invariant foliations for stochastic PDEs of the form (\ref{spde1}). The methodology will be inspired by the so-called Factor Analysis developed in the Econometrics literature (see e.g \cite{stock},~\cite{bai}, \cite{bai1}, \cite{forni}), but with some fundamental differences: \textbf{(a)} Unlike the classical discrete Factor Analysis, we are working with an underlying continuous time process sampled in high-frequency at discrete points in time and space. \textbf{(b)} The spaces $(\mathcal{Q}, \mathcal{N})$ cannot be identified by applying standard techniques from Factor Analysis due to the rather distinct behavior between quadratic variation and covariance matrices in the high-frequency setup. \textbf{(c)} More importantly, the factor analysis introduced here allows us to reduce and rank the underlying semimartingale factors in terms of quadratic variation rather than covariance, including bounded variation components.

Throughout this section, Assumptions~(A1-A2) are in force. We also assume the underlying state-space $E$ is the Sobolev space of absolutely continuous functions $f:[a,b]\rightarrow \mathbb{R}$ such that

$$\|f\|^2_E:= |f(a)|^2 + \int_a^b |f'(x)|^2\mu(dx)< \infty$$
where $\mu$ is absolutely continuous w.r.t Lebesgue measure (see e.g \cite{filipo}) and we write $\langle\cdot, \cdot\rangle_E$ to denote the associated inner product. For simplicity of exposition, we work with the closed subspace of $E$ formed by functions $f(a)=0$ and we set $\frac{d\mu}{dx}= 1$. With a slight abuse of notation we denote it by $E$.

We are going to fix the minimal invariant foliation $\mathcal{V}_t = \phi_t + V$ generated by a $d$-dimensional subspace $V$ equipped with a basis $\{\lambda_1,\ldots, \lambda_d\}$ and a truly $d$-dimensional semimartingale $(Z^1, \ldots, Z^d)$ satisfying Assumption \ref{A2} such that

\begin{equation}\label{sss}
r_t = \phi_t + \sum_{j=1}^dZ^j_t\lambda_j; 0\le t\le T.
\end{equation}
In this section, we work in a high-frequency setup as follows. To shorten notation, the points of partition in time  $(t^n_i)_{i=1}^{\bar{n}}$ and space $(x^N_j)_{j=1}^{\bar{N}}$ will be denoted by $t_i=t^n_i$ and $x_j=x^N_j$, respectively, and we set $\rho(n):=\sup_{1\le i\le \bar{n}-1}|t_{i+1}-t_{i}|$ and $\delta(N):=\sup_{1\le j\le \bar{N}-1}|x_{j+1}-x_{j}|$. We will assume the samplings in time and space will be equally spaced and equidistant. For the sake of preciseness, it should be noted we are dealing with a sequence of refining partitions and we always assume that $\rho(n)\rightarrow 0$, $\delta(N)\rightarrow 0$, $\bar{n}\rightarrow \infty$, $\bar{N}\rightarrow \infty$ as $n,N\rightarrow \infty$, where both $n$ and $N$ goes to infinity.

We assume that the observations are generated by a space-time process

\begin{equation}\label{observedXcont}
X_t(x) := r_t(x) + \varepsilon_t(x); 0\le t\le T, x\in [a,b]
\end{equation}
where $\varepsilon$ represents a space-time error component satisfying some regularity conditions. In this section, we assume that one is able to sample the curves $x\mapsto X_t(x)$ in high-frequency in time. For instance, term-structure objects like interpolated forward rate curves are examples of this type of data. See e.g ~\cite{laurini} and other references therein.

In particular, under Assumptions (A1-A2), the $(\bar{n}\times \bar{N})$-matrix $X_{t_i}(x_j)$ of observations admits an affine noisy representation

\begin{equation}\label{observedX}
X_{t_i}(x_j)=\phi_{t_i}(x_j) + \sum_{k=1}^dZ^k_{t_i}\lambda_k(x_j) + \varepsilon_{t_i}(x_j)
\end{equation}
for $i=1,\ldots, \bar{n}$ and $j=1,\ldots, \bar{N}$. Throughout this section, we assume that $\phi$ is known by the observer and with a slight abuse of notation  we write $X$ for the difference $X-\phi$. In matrix representation, we shall write

$$\mathbb{X}=\mathbb{Z}\Lambda^{\top} + \mathcal{E},\quad \mathbb{X}_{i}=\Lambda \mathbb{Z}_{i} + \mathcal{E}_{i}; 1\le i\le \bar{n}$$
where $\Lambda:=\{\lambda_j(x_i); 1\le i\le \bar{N}, 1\le j\le d\}$, $\mathbb{X} := \{X_{t_i}(x_j); 1\le i\le \bar{n}, 1\le j\le \bar{N}\}$, $\mathbb{Z} := \{Z^j_{t_i}; 1\le i\le \bar{n}, 1\le j\le d\}$ and $\mathcal{E}:=\{\varepsilon_{t_i}(x_j); 1\le i\le \bar{n}, 1\le j\le \bar{N}\}$.

\subsection{Estimating the underlying dimension}

Obviously, the first step is to estimate the underlying dimension of the finite-dimensional realization. But this is an almost straightforward application of Bai and Ng~\cite{bai}. Indeed, we are interested in solving the following optimization problem (for large $n,N$)

$$
\min_{\Lambda^k, \mathcal{Y}(k)}
\rho(n)\delta(N)\sum_{i=1}^{\bar{n}}\sum_{j=1}^{\bar{N}}\Big(X_{t_i}(x_j)-\langle g^k(x_j),Y_{t_i}(k)\rangle_{\mathbb{R}^k} \Big)^2,
$$
where the minimum is taken over the set of real matrices with columns
$$\Lambda^k=(g^1,\ldots, g^k)\in \mathbb{M}_{\bar{N}\times k}\quad;~\mathcal{Y}(k) = (Y(1), \ldots, Y(k))\in \mathbb{M}_{\bar{n}\times k},$$
subject to either $\delta(N)\Lambda_k^\top \Lambda_k =I_k$ or $\rho(n) \mathcal{Y}^\top(k)\mathcal{Y}(k) = I_k$ (Identity matrix in $\mathbb{M}_{k\times k}$). Here $g^i:=(g^i(x_1), \ldots , g^i(x_{\bar{N}}))^\top$ and $Y(i):= (Y_{t_1}(i), \ldots, Y_{t_{\bar{n}}}(i))^\top$ for $1\le i\le k$. The index $k$ encodes the allowance of $k$ factors in the estimation procedure.

\begin{remark}
In order to avoid curse of dimensionality issues, we do assume $k < \min \{\bar{n},\bar{N}\}$ and $n,N\rightarrow \infty$ jointly.
\end{remark}

The factor estimator is defined as follows. Let $\hat{Y}(k)\in \mathbb{M}_{\bar{n}\times k}$ be the random matrix defined by $\hat{Y}_{t_i,j}(k):=\rho(n)^{-1/2}y^j_{t_i};~1\le j\le k, 1\le i\le \bar{n}$ whose the $j$th column
$$y^{j}:= (y^{j}_{t_1},\ldots, y^{j}_{t_{\bar{n}}})\in \mathbb{M}_{\bar{n}\times 1}$$
is an eigenvector associated to the $j$-th largest eigenvalue of $\mathbb{X}\mathbb{X}^{\top}\in \mathbb{M}_{\bar{n}\times \bar{n}}$ subject to $\rho(n)\hat{Y}^{\top}(k)\hat{Y}(k)=I_k$. The loading factor estimator is given by $\hat{\Lambda}^k:=\rho(n)\mathbb{X}^\top \hat{Y}(k)$

In the sequel, we denote
$$
V(k,\hat{Y}(k)):=\min_{\Lambda^k}
\rho(n)\delta(N)\sum_{i=1}^{\bar{n}}\sum_{j=1}^{\bar{N}}\Big(X_{t_i}(x_j)-\langle g^k(x_j),\hat{Y}_{t_i}(k)\rangle_{\mathbb{R}^k} \Big)^2.
$$
The estimation procedure for the underlying dimension of $V$ is due to
Bai and Ng~\cite{bai}. They propose a class of information criteria of the form
\begin{equation}\label{pc}
PC(k):=V(k,\hat{Y}(k))+kq(n,N)
\end{equation}
for suitable penalty functions $q(n,N)$. One can show the estimation of $dim~V$ can be still carry out on the basis of the ideas contained in ~\cite{bai} even in the high-frequency setup, as long as the following assumptions hold true. The following assumptions are inspired by Bai and Ng \cite{bai} and Bai \cite{bai1} but in the context of a continuous time setup sampled at discrete times. For the sake of completeness, we list them here. In the sequel, $\mathcal{H}^q$ is the space of $q$-integrable continuous Brownian semimartingales.

\

\noindent \textbf{(D1)} $Z^j\in\mathcal{H}^4$ for each $j=1,\ldots, d$ and

$$\rho(n)\sum_{i=1}^{\bar{n}}Z_{t_i}Z^{\top}_{t_i}\rightarrow \Sigma_{Z}:= \Big(\langle Z^i, Z^j\rangle_{L^2([0,T];\mathbb{R})}\Big)_{1\le i,j\le d}$$
in probability as $n\rightarrow \infty$ and $\Sigma_Z$ is a $d\times d$ positive definite matrix a.s

\

\noindent \textbf{(D2)} $\sup_{j\ge 1}\|\lambda(x_j)\|_{\mathbb{R}^d} < \infty$ and $$\Bigg\|\delta(N)\sum_{j=1}^{\bar{N}}\lambda^{\top}(x_j)\lambda(x_j)-\int_a^b\lambda^{\top}(x)\lambda(x)dx\Bigg\|_{(2)}\rightarrow 0$$
as $\delta(N)\rightarrow 0$. Moreover, $\Sigma_\lambda:=\int_a^b\lambda^{\top}(x)\lambda(x)dx$ is a $d\times d$-positive definite matrix.

\

\noindent \textbf{(D3)} The error process $\varepsilon$ satisfies assumptions:

\begin{itemize}
\item $\mathbb{E}\varepsilon_{t_i}(x_j)=0,~\mathbb{E}\sup_{i,j}|\varepsilon_{t_i}(x_j)|^8 < \infty$

\item If $\gamma_N(t_i,t_j):=\mathbb{E}\langle \varepsilon_{t_i},\varepsilon_{t_j} \rangle_{\mathbb{R}^{\bar{N}}}\delta(N)$ then $\sup_{i}\gamma_N(t_i,t_i) < \infty$ and the sum $\rho(n)\sum_{i,j=1}^{\bar{n}}|\gamma_N(t_i,t_j)|$ is bounded in $n,N$.

\item $\sup_{N\ge 1}\delta(N)\sum_{\ell,m=1}^{\bar{N}}\sup_{i}|\mathbb{E}\varepsilon_{t_i}(x_m)\varepsilon_{t_i}(x_\ell)| < \infty$.

    \item $\sup_{n,N\ge 1}\delta(N)\rho(n) \sum_{i,j=1}^{\bar{n}} \sum_{\ell,m}^{\bar{N}}|\mathbb{E}\varepsilon_{t_i}(x_\ell)\varepsilon_{t_j}(x_m)| < \infty.$

        \item $\mathbb{E}\Big|\delta^{1/2}(N)\sum_{\ell=1}^{\bar{N}}[\varepsilon_{t_i}(x_\ell)\varepsilon_{t_j}(x_\ell)-
            \mathbb{E}\varepsilon_{t_i}(x_\ell)\varepsilon_{t_j}(x_\ell)]\Big|^4.$

       \item The error $\varepsilon$ and the factors $Z$ are mutually independent.
\end{itemize}

\

\noindent \textbf{(D4)}

$$\sup_{n,N}\sup_{t_s}\mathbb{E}\Big\| \sqrt{\rho(n)\delta(N)}\sum_{i=1}^{\bar{n}} \sum_{j=1}^{\bar{N}} Z_{t_i} \big[\varepsilon_{t_i}(x_j)\varepsilon_{t_s}(x_j) - \mathbb{E}[\varepsilon_{t_i}(x_j)\varepsilon_{t_s}(x_j)\big]  \Big\|^2_{\mathbb{R}^d} < \infty$$

$$\sup_{n,N}\mathbb{E}\Big\| \sqrt{\rho(n)\delta(N)}\sum_{i=1}^{\bar{n}}\sum_{j=1}^{\bar{N}} Z^\top_{t_i}\lambda (x_j)\varepsilon_{t_i}(x_j)    \Big\|^2_{(2)}<\infty$$

\begin{remark}\label{discussionFR}
The assumption $\text{Rank}~\Sigma_Z=d$ a.s is not strong. Indeed, since $\Sigma_Z$ is a Gramian matrix, then if $Z$ does not satisfy \textbf{(D1)} then we shall reduce the effective dimension without losing information. More importantly, we stress that \textbf{(D1)} implies that factors satisfy Assumption 2.1 but it does not imply that $[Z]_T$ has full rank a.s. The fact that $\Sigma_\lambda$ is a positive definite matrix is equivalent to the fact that $\{\lambda_1,\ldots\lambda_d\}$ is linearly independent on the state space $E$ equipped with the $L^2([a,b];\mathbb{R})$-inner product\footnote{Since we work with the subspace of functions $f\in E$ such that $f(a)=0$, then $\langle\cdot, \cdot \rangle_{L^2}$ is indeed an inner product over $E$.}. In contrast to the usual factor analysis (see e.g~\cite{bai,bai1}), we stress that $\Sigma_Z$ is random.
\end{remark}

In this case, under some mild growth condition on $q(n,N)$,

\begin{equation}\label{dhat}
\hat{d}:=~argmin_{1\le k\le kmax}~PC(k)
\end{equation}
will be a consistent estimator for $dim~V$, where $kmax$ is an arbitrary integer such that $d \le kmax$. The proof of this statement will be inspired by the arguments given by Bai and Ng~\cite{bai} and Bai~\cite{bai1}. In one hand, in contrast to \cite{bai} and \cite{bai1}, our asymptotic matrix $\Sigma_Z$ is random and the sampling should be in high-frequency. On the other hand, Assumption \textbf{D1} allows us to prove similar results without significant extra effort. For the sake of completeness, we give the details here. In the sequel, we denote $C_{nN} := min\{\delta(N)^{-1/2},\rho(n)^{-1/2}\}$ and

$$\gamma_N(t_{\ell},t_i):=\delta(N)\mathbb{E}\langle \varepsilon_{t_{\ell}}, \varepsilon_{t_i}\rangle_{\mathbb{R}^N}\quad  \theta_N(t_{\ell},t_i):=\delta(N)\langle \varepsilon_{t_{\ell}}, \varepsilon_{t_i}\rangle_{\mathbb{R}^N}-\gamma_N(t_{\ell},t_i)
$$

$$
\xi_N(t_\ell,t_i):=Z^{\top}_{t_i}\Lambda^{\top}\varepsilon_{t_\ell}\delta(N);\quad \eta_N(t_\ell,t_i):=Z^{\top}_{t_\ell}\Lambda^{\top}\varepsilon_{t_i}\delta(N)
$$
for $1\le i,\ell\le \bar{n}$.

\begin{lemma}\label{techlemma}
If Assumptions \textbf{(D1-D2-D3-D4)} hold, then

\noindent (a) $\rho(n)\sum_{\ell=1}^{\bar{n}}\hat{Y}_{t_\ell}(d)\gamma_N(t_\ell,t_i) = O_{\mathbb{P}}(\frac{1}{\rho(n)^{-1/2}C_{nN}})$

\

\noindent (b) $\rho(n)\sum_{\ell=1}^{\bar{n}}\hat{Y}_{t_\ell}(d)\theta_N(t_\ell,t_i)=O_{\mathbb{P}}(\frac{1}{\delta(N)^{-1/2}C_{nN}})$

\

\noindent (c) $\rho(n)\sum_{\ell=1}^{\bar{n}}\hat{Y}_{t_\ell}(d)\xi_N(t_\ell,t_i) = O_{\mathbb{P}}(\delta(N)^{1/2})$

\

\noindent (d) $\rho(n)\sum_{\ell=1}^{\bar{n}}\hat{Y}_{t_\ell}(d)\eta_N(t_\ell,t_i)= O_{\mathbb{P}}(\frac{1}{\delta(N)^{-1/2}C_{nN}})$.

\end{lemma}
\begin{proof}
Let $L_{n,N}$ be the diagonal matrix of the eigenvalues of $\rho(n)\delta(N)\mathbb{X}\mathbb{X}^{\top}$ arranged in decreasing order. From \textbf{(D1-D2-D3-D4)}, one can easily check that $\|\rho(n)\delta(N) \mathbb{X}\mathbb{X}^\top \|_{(2)} = O_{\mathbb{P}}(1)$ and hence $\|L_{n,N}\|_{(2)} = O_{\mathbb{P}}(1)$. In this case, the same argument given in the proof of Lemma A1 in \cite{bai1} allows us to state that

\begin{equation}\label{averageY}
C^2_{nN} \Big(\rho(n)\sum_{i=1}^{\bar{n}}\|\hat{Y}_{t_i}(d) - H_d^\top Z_{t_i}\|^2_{\mathbb{R}^{d}} \Big) = O_{\mathbb{P}}(1)
\end{equation}
where $H_d:=\delta(N)\Lambda^\top \Lambda \mathbb{Z}^\top \hat{Y}(d)\rho(n) L^{-1}_{n,N}\in \mathbb{M}_{d\times d}$. Assumptions \textbf{(D1-D2-D3-D4)} together with (\ref{averageY}) allow us to repeat the same argument given in the proof of Lemma A2 in \cite{bai1} to conclude that the statement hold true. We omit the details.
\end{proof}

The next result was enunciated by Bai and Ng~\cite{bai} in Lemma A3 (in the context of a discrete-time model and deterministic $\Sigma_Z$) without a complete proof. For sake of completeness, we give the details here in our context.

\begin{lemma}\label{LnN}
Let $L_{n,N}$ be the diagonal matrix of the eigenvalues of $\rho(n)\delta(N)\mathbb{X}\mathbb{X}^{\top}$ arranged in decreasing order. If Assumptions \textbf{(D1-D2-D3-D4)} hold then
$$L_{n,N}\stackrel{p}\to \mathcal{C}:=diag~(c_1,\ldots, c_d)$$
as $n,N\rightarrow \infty$, where $(c_1, \ldots, c_d)$ are the eigenvalues (in decreasing order) of $\Sigma_{\lambda}\Sigma_Z$.
\end{lemma}
\begin{proof}
We follow closely the idea contained in the proof of Proposition 1 in \cite{bai1}. By the very definition, $\rho(n)\delta(N) \mathbb{X}\mathbb{X}^\top\hat{Y}(d) = \hat{Y}(d)L_{n,N}~a.s$ and hence

$$\Big(\delta(N)\Lambda^\top \Lambda\Big)^{1/2}\rho(n) \mathbb{Z}^\top\Big(\rho(n)\delta(N) \mathbb{X}\mathbb{X}^\top\Big)\hat{Y}(d) = \Big(\delta(N)\Lambda \Lambda^\top\Big)^{1/2}\big(\rho(n)\mathbb{Z}^\top\hat{Y}(d)\big)L_{n,N}$$
From the identity $\mathbb{X} = \mathbb{Z}\Lambda^\top + \mathcal{E}$, we actually have

\begin{eqnarray}
\nonumber\Big(\delta(N)\Lambda^\top \Lambda\Big)^{1/2}\big(\rho(n)\mathbb{Z}^\top\mathbb{Z}\big)\big(\delta(N)\Lambda^\top \Lambda\big) \nonumber\big(\mathbb{Z}^\top\hat{Y}(d)\rho(n)\big) + s_{nN} &=& \Big(\delta(N)\Lambda \Lambda^\top\Big)^{1/2}\\
\label{inter1}& &\\
\nonumber&\cdot& \big(\rho(n)\mathbb{Z}^\top\hat{Y}(d)\big)L_{n,N}
\end{eqnarray}
where

\begin{small}
\begin{eqnarray}
\nonumber s_{n,N}&:=&\big(\delta(N)\Lambda^\top \Lambda\big)^{1/2}\big[\rho(n)\big(\mathbb{Z}^\top\mathbb{Z}\big)\rho(n)\delta(N)\Lambda^\top\mathcal{E}^\top \hat{Y}(d) + \nonumber\rho(n)\delta(N)\mathbb{Z}^\top\mathcal{E}\Lambda\mathbb{Z}^\top\hat{Y}(d)\rho(n)\\
\label{inter2} & &\\
\nonumber& + & \rho(n)\delta(N)\mathbb{Z}^\top\mathcal{E}\mathcal{E}^\top\hat{Y}(d)\rho(n)]=o_{\mathbb{P}}(1)
\end{eqnarray}
\end{small}
due to Lemma \ref{techlemma}. Let $U_{n,N}:=\Big(\delta(N)\Lambda^\top \Lambda\Big)^{1/2}\big(\rho(n)\mathbb{Z}^\top\mathbb{Z}\big)\big(\delta(N)\Lambda^\top \Lambda\big)^{1/2}$ and
$$E_{n,N}:=\Big(\delta(N)\Lambda \Lambda^\top\Big)^{1/2}\big(\rho(n)\mathbb{Z}^\top\hat{Y}(d)\big).$$

We shall write (\ref{inter1}) as follows

$$[U_{n,N} + s_{n,N}E_{n,N}E^+_{n,N}]E_{n,N} = E_{n,N}L_{n,N}$$
where $E^+_{n,N}$ is the pseudoinverse of $E_{n,N}$. Then each column of $E_{n,N}$ is an eigenvector of $N_{n,N} + s_{n,N}E_{n,N}E^+_{n,N}$. Since $E_{n,N}E^+_{n,N} = O_{\mathbb{P}}(1)$ then (\ref{inter2}) and Assumptions \textbf{(D1, D2)} yield

$$\|U_{n,N} +s_{n,N}E_{n,N}E^+_{n,N}- \Sigma_{\lambda}^{1/2}\Sigma_{Z}\Sigma^{1/2}_{\lambda}\|_{(2)} \stackrel{p}\to 0$$
as $n,N\rightarrow \infty$. By the continuity of the eigenvalues, we do have $\|L_{n,N} - \mathcal{C}\|_{(2)}\stackrel{p}\to 0$
as $n,N\rightarrow \infty$. Since $\Sigma_{\lambda}^{1/2}\Sigma_{Z}\Sigma^{1/2}_{\lambda}$ and $\Sigma_{\lambda}\Sigma_Z$ have the same random eigenvalues, we conclude the proof.
\end{proof}

\begin{lemma}\label{Gmatrix}
Assume that hypotheses \textbf{(D1-D2-D3-D4)} hold and the eigenvalues of $\Sigma_{\lambda}\Sigma_{Z}\in \mathbb{M}_{d\times d}$ are distinct almost surely. Then, for every $j=1,\ldots, d$, there exists a random vector $(G_{1j},\ldots, G_{dj})$ such that

$$\Bigg(\sum_{\ell=1}^{\bar{n}}\sqrt{\rho(n)}y^1_{t_\ell}Z^j_{t_\ell},\ldots ,\sum_{\ell=1}^{\bar{n}}\sqrt{\rho(n)}y^{\hat{d}}_{t_\ell}Z^j_{t_\ell}\Bigg) \stackrel{p}{\to} \big(G_{1j},\ldots, G_{dj}\big)$$
as $n,N\rightarrow \infty$. Moreover, the matrix $G := (G_{ij})_{1\le i,j\le d}$ is invertible a.s and it is given by $G=\mathcal{C}^{1/2}\Phi^\top \Sigma^{-1/2}_{\lambda}$ and $\Phi$ is the eigenvector matrix related to $\mathcal{C}$ subject to $\Phi^\top\Phi=I_d~a.s$.
\end{lemma}
\begin{proof}
By using Lemma \ref{LnN}, the proof is identical to Proposition 1 in Bai \cite{bai1} even in the case when $\Sigma_Z$ is random. We refer the reader to the discussion in page 162 in \cite{bai1}.
\end{proof}

We are now able to present the following result.

\begin{lemma}\label{pestimation}
Let us assume that assumptions~\textbf{(D1,D2, D3, D4)} hold and let $\hat{d} = \text{arg~min}_{1\le k \le kmax}PC(k)$. Assume the eigenvalues of $\Sigma_{\lambda}\Sigma_{Z}\in \mathbb{M}_{d\times d}$ are distinct almost surely. Then, $lim_{n,N\rightarrow \infty} \mathbb{P}[\hat{d} = d] = 1$ if (i) $q(N, n) \rightarrow 0$ and (ii)
$ C_{nN}q(N, T) \rightarrow  \infty$ as $n,N\rightarrow \infty$ where $C_{nN} = min\{\delta{(N)}^{-1/2}, \rho(n)^{-1/2}\}$.
\end{lemma}
\begin{proof}
The same arguments given in the proof of Theorem 1 in Bai and Ng~\cite{bai} apply in our context. In particular, Lemmas 2, 3 and 4 in Bai and Ng~\cite{bai} can be similarly proved in our context as well by using Assumptions \textbf{(D1, D2, D3, D4)} and the fact that $\Sigma_\lambda\Sigma_Z$ has distinct eigenvalues a.s. In particular, the fact that $\Sigma_Z$ is not deterministic is not essential for the validity of the analogous results of Lemmas 2, 3 and 4 given by \cite{bai} in our context, as long as $rank~\Sigma_Z=d~a.s$ (Assumption \textbf{D1}). In particular, for $k< d$ let us define $J^\top_k:=\hat{Y}^\top(k)\mathbb{Z}\rho(n)\Lambda^\top\Lambda\delta(N)\in \mathbb{M}_{k\times d}$. In our context, Lemma 3 in \cite{bai} can be written as follows: There exists $\tau_k>0~a.s$ such that

$$\liminf_{n,N\rightarrow \infty}V(k,\mathbb{Z}J_k) - V(d,\mathbb{Z}) = tr (R_k.\Sigma_\lambda) =: \tau_k$$
in probability, where $R_k:=\Sigma_Z - \Sigma_Z \mathbb{H}_k(\mathbb{H}_k^\top\Sigma_Z\mathbb{H}_k)^{-1}\mathbb{H}^\top_k\Sigma_Z$ and $\mathbb{H}_k:=\lim_{n,N\rightarrow \infty}J_k$ exists due to Lemma \ref{Gmatrix}. By construction $rank~\mathbb{H}_k=k < d~a.s$. Assumptions \textbf{(D1-D2)} yield $tr (R_k.\Sigma_\lambda)>0$ a.s. By writing

$$PC(k) - PC(d) = V(k,\hat{Y}(k)) - V(d, \hat{Y}(d)) - (d-k)q(n,N)$$
and splitting
\begin{eqnarray*}
 V(k,\hat{Y}(k)) - V(d, \hat{Y}(d))&=& [V(k,\hat{Y}(k)) - V(k,\mathbb{Z}J_k)] + [V(k,\mathbb{Z}J_k) -  V(d, \mathbb{Z}J_d) ]\\
 & &\\
 &+& [V(d, \mathbb{Z}J_d) - V(d, \hat{Y}(d))],
\end{eqnarray*}
we shall use the same argument in the proof of Th 1 in \cite{bai} to conclude that

$$\lim_{n,N\rightarrow\infty}\mathbb{P}\{PC(k)< PC(d)\}=0,$$
for each $k <  d$. If $kmax\ge k \ge d$, then similar to Lemma 4 in \cite{bai}, Assumptions ~\textbf{(D1,D2, D3, D4)} and the fact that the eigenvalues of $\Sigma_{\lambda}\Sigma_{Z}\in \mathbb{M}_{d\times d}$ are distinct almost surely yield

$$V(k, \hat{Y}(k)) - V(d, \hat{Y}(d))) = O_{\mathbb{P}}(C_{nN}^{-2}).$$
The rest of the proof is identical to the proof of Th 1 in \cite{bai}, so we omit the details.
\end{proof}

\subsection{Main Results}\label{mainsection} Let us now present the main results of this section. The following list of assumptions will also be in force throughout this section.

\

\noindent \textbf{(Q1)} The eigenvalues of $\Sigma_\lambda \Sigma_Z\in \mathbb{M}_{d\times d}$ are distinct almost surely.

\

\noindent \textbf{(Q2)} We assume

$$\rho(n)\sum_{1\le \ell < s\le \bar{n}}|y^k_{t_\ell}y^k_{t_s}|$$
is bounded in probability for every $k\in \{1, \ldots, d\}$.

\

\noindent \textbf{(Q3)}

$$\rho(n)\sum_{1\le \ell < s\le \bar{n}}y^k_{t_\ell}y^k_{t_s}\langle \varepsilon_{t_\ell},\lambda_r\rangle_{\mathbb{R}^N}\delta(N)\langle \varepsilon_{t_s},\lambda_j\rangle_{\mathbb{R}^N}\delta(N)\stackrel{p}{\to} 0$$
as $n,N\rightarrow \infty$ for each $k,r,j\in \{1,\ldots,d \}$.

\

\noindent \textbf{(Q4)} There exists a sequence of natural numbers $\{\gamma(n); n\ge 1\}$ decaying to zero such that

$$\mathbb{E}\sum_{i=1}^{\bar{n}}\|\Delta \varepsilon_{t_i}\|^2_{\mathbb{R}^N}\delta(N)=O(\gamma(n)).$$

\

\noindent \textbf{(Q5)} $\sup_{0\le t\le T}\|\varepsilon_t \|^2_E$ is bounded in probability and for each $i\in \{1, \ldots, d\}$,

$$\rho(n)\sum_{1\le \ell < s\le \bar{n}}|y^i_{t_\ell}y^i_{t_s}|\|\varepsilon_{t_\ell}\|_E\|\varepsilon_{t_s}\|_E\stackrel{p}\to 0$$
as $n\rightarrow \infty$.

\

\begin{remark}
Assumption \textbf{(Q1)} is essential to our estimation procedure because it yields an asymptotic $Y\in \mathcal{X}^d$ and a random basis for $V$ which will allow us to construct a consistent pair of estimators $(\hat{\mathcal{Q}}, \hat{\mathcal{N}})$ for the splitting $V = \mathcal{Q}\oplus \mathcal{N}$ of the invariant manifold $V$. The technical conditions~\textbf{(Q2, Q3, Q5)} are not strong since they impose a very mild growth condition on the eigenvectors of $\mathbb{X}^\top \mathbb{X}$. Assumption~\textbf{(Q4)} is quite natural for error structures arising in space-time data generated by stochastic PDEs. For example, as far as the consistency problem of the HJM model (see section 4.2), assumption \textbf{(Q4)} means that the initial fitting method used to interpolate points which generates $X$ cannot introduce an extrinsic volatility for the market. In other words, \textbf{(Q4)} rules out pure martingale error structures.
\end{remark}

The starting point for the estimation of $(\mathcal{Q},\mathcal{N})$ is to take advantage of the identities (\ref{goodN1}) and (\ref{goodN2}) based on a quadratic variation matrix $[Y]_T$ constructed from an asymptotic $Y \in \mathcal{X}^d$ satisfying Assumptions \ref{A1} and \ref{A2}. We define such process as follows: Let $Z$ be a factor representation of (\ref{spde1}) satisfying Assumption (A2) and the (\textbf{D1-D2-D3-D4-Q1}) and let $G$ be the associated matrix defined in Lemma \ref{Gmatrix}. Since $G$ is nonsingular a.s and $\Sigma_\lambda$ is positive definite, then the random matrix matrix $\mathcal{A} = \mathcal{C}^{-1}G\Sigma_{\lambda} = (A_{ij})_{1\le i,j\le d}\in \mathbb{M}_{d\times d}$ given by

\label{Aoperator}
\begin{equation}
A_{ij}=\sum_{k=1}^d c^{-1}_iG_{ik}\int_a^b\lambda_k(x)\lambda_j(x)dx;~1\le i,j\le d,
\end{equation}
is non-singular a.s. Then we shall apply Remark \ref{randomrep1} to state that $\mathcal{A}Z\in \mathcal{X}^d$ is a truly $d$-dimensional process and it is a factor measurable process realizing (\ref{randomrep2}) for the basis (loading factors) $(\mathcal{A}^{-1})^\top\lambda$. From Remark \ref{invprop}, $\mathcal{A}Z$ satisfies Assumptions \ref{A1} and \ref{A2}.

In the sequel, for a given factor representation $Z$ of (\ref{spde1}) satisfying Assumption (A2) and the assumptions in Lemma \ref{Gmatrix}, we set $Y = \mathcal{A}Z$. Let $\widehat{[Y]}_T:=(\hat{m}_{\ell k})_{1\le \ell,k\le \hat{d}}$ and $[Y]_T:=(m_{\ell k})_{1\le \ell,k\le d}$ be the matrices given, respectively, by

\begin{equation}\label{hatquadmatrix}
\hat{m}_{\ell k}:= \sum_{i=1}^{n-1} \big(\hat{Y}_{t_{i+1},\ell}(\hat{d}) -\hat{Y}_{t_{i},\ell}(\hat{d})\big)
\big( \hat{Y}_{t_{i+1},k}(\hat{d}) -\hat{Y}_{t_{i},k}(\hat{d})\big),
\end{equation}
for $1\le \ell,k\le \hat{d}$ and

$$m_{s v}:= [Y^s,Y^{v}]_T;~1\le s,v\le d.$$
We stress that $Y\in \mathcal{X}^d$ has a quadratic variation matrix in the sense of Definition \ref{qdef}.

\begin{proposition}\label{conv-qua-var}
If Assumptions~\textbf{(D1, D2, D3, D4)} and \textbf{(Q1, Q2, Q3, Q4)} hold true, then

$$\|\widehat{[Y]}_T - [Y]_T\|^2_{(2)}\stackrel{p}\to 0$$
as $n,N\rightarrow \infty$.
\end{proposition}
\begin{proof}
At first, by taking $n,N$ large enough, assumptions \textbf{(D1, D2, D3, D4, Q1)} allow us to use Lemma \ref{pestimation} and we assume that $\hat{d}=d$ because $\hat{d}$ is an integer-valued consistent estimator. In the sequel, if $P$ is a real-valued process then we write $\Delta_{t_i} P:= P_{t_{i+1}} - P_{t_{i}}; 1\le i \le n-1$. By using the definition of $\hat{Y}(d)$, one can actually write

\begin{small}
\begin{eqnarray*}
\hat{Y}_{t_i}(d)&=&H^{\top}_{d} Z_{t_i}+ L^{-1}_{nN}\rho(n)\sum_{\ell=1}^{\bar{n}}\hat{Y}_{t_{\ell}}(d)\Big(\gamma_N(t_{\ell},t_i) +\theta_N(t_{\ell},t_i) +\xi_N(t_{\ell},t_i) + \eta_N(t_{\ell},t_i) \Big)\\
& &\\
&=:& H^{\top}_{d}Z_{t_i}+ \hat{R}_{t_i}(n,N),
\end{eqnarray*}
\end{small}
\noindent where $H^\top_d := L^{-1}_{nN} \hat{Y}^{\top}(d)\mathbb{Z}\rho(n)\Lambda^{\top}\Lambda\delta(N)$ and $L_{nN}$ is
be the diagonal matrix of the eigenvalues of $\rho(n)\delta(N)\mathbb{X}\mathbb{X}^\top$ arranged in decreasing order (see Lemma \ref{LnN})).

To shorten notation, we set $\hat{W}_{t_i}:=H^{\top}_{d}Z_{t_i}$ and $\varphi_N(t_\ell,t_i):=\gamma_N(t_{\ell},t_i) + \theta_N(t_\ell,t_i) + \xi_N(t_\ell,t_i) + \eta_N(t_\ell,t_i)$ for $1\le i,\ell \le \bar{n}$. In the sequel, for $r,\ell=1,\ldots, d$ we denote $O_{\mathbb{P};(r,\ell)}(\xi_n)$ any random variable which is $O(\xi_n)$ in probability, $C$ is a constant which may differ from line to line and let us denote the $d\times d$-matrix given by $\tilde{W}:=(\hat{w}_{sq})$ where

$$\hat{w}_{sq}:=\sum_{i=1}^{\bar{n}-1}\Delta\hat{W}_{t_i}(s)\Delta \hat{W}_{t_i}(q)$$
for $s,q=1\ldots,d$. We claim that

\begin{equation}\label{qcon1}
\sum_{m=1}^{d}\sum_{i=1}^{\bar{n}-1}\Big(\Delta\hat{R}^{m}_{t_i}(n,N)\Big)^2 \stackrel{p}{\to} 0
\end{equation}
and
\begin{equation}\label{qcon2}
vec~(\tilde{W})\stackrel{p}{\to} vec~([Y]_T)
\end{equation}
as $n,N\rightarrow \infty$, where $vec$ is the usual vectorization operator. Let $L_{nN} = ~diag~(\gamma_1, \ldots, \gamma_{\bar{n}})$. By the very definition,
$$(H^\top_d)_{ij} = \sqrt{\rho(n)}\delta(N)\sum_{k=1}^d\big(\sum_{\ell=1}^{\bar{n}}\gamma_i^{-1}y^i_{t_\ell}Z^k_{t_\ell} \big)\big(\sum_{m=1}^N \lambda_k(x_m)\lambda_j(x_m) \big); 1\le i,j\le d.$$
By Lemma \ref{LnN}, we know that $L_{n,N}\stackrel{p}{\to} diag~(c_1, \ldots, c_d)$ as $n,N\rightarrow \infty$, where $(c_1,\ldots, c_d)$ are the eigenvalues of $\Sigma_\lambda\Sigma_Z$. Then Lemma \ref{Gmatrix} yields

\begin{eqnarray*}
\hat{w}_{sq} &=& \sum_{j=1}^d\sum_{r=1}^d (H^\top_d)_{qj} (H^\top_d)_{sr}\sum_{i=1}^{\bar{n}-1} \Delta Z^j_{t_i} \Delta Z^r_{t_i}\\
& &\\
&\stackrel{p}\to& \sum_{j=1}^d\sum_{r=1}^d \sum_{k=1}^d\sum_{m=1}^d\langle \lambda_k, \lambda_j \rangle_{L^2([a,b];\mathbb{R})} c_q^{-1}G_{qk}
\langle \lambda_m, \lambda_r \rangle_{L^2([a,b];\mathbb{R})} c_s^{-1}G_{sm}[Z^j, Z^r]_T\\
& &\\
&=&[Y^s, Y^q]_T; 1\le s,q\le d,
\end{eqnarray*}
as $n,N\rightarrow \infty$.  This shows that (\ref{qcon2}) holds. By noting that

\begin{eqnarray}
\nonumber\Delta \hat{Y}_{t_i,\ell}(d)\Delta \hat{Y}_{t_i,k}(d)&=&\Delta\hat{W}_{t_i}(k)\Delta\hat{W}_{t_i}(\ell)+ \nonumber\Delta\hat{W}_{t_i}(k)\Delta\hat{R}_{t_i}^{\ell}(n,N)\\
\label{identM}& &\\
\nonumber&+&\Delta\hat{R}^k_{t_i}(n,N)\Delta\hat{W}_{t_i}(\ell) + \nonumber\Delta\hat{R}^k_{t_i}(n,N)\Delta\hat{R}^{\ell}_{t_i}(n,N);~1\le k,\ell \le d,
\end{eqnarray}
we only need to check~(\ref{qcon1}) in order to conclude the proof. Let $\hat{S}_{t_i}(n,N): = L_{n,N}\hat{R}_{t_i}(n,N)\in \mathbb{M}_{d\times 1}$. From Lemma \ref{LnN}, we know that  $\|L^{-1}_{n,N}\|_{(2)}=O_{\mathbb{P}}(1)$, so we only need to check that

\begin{equation}\label{qcon3}
\sum_{m=1}^{d}\sum_{i=1}^{\bar{n}-1}\Big(\Delta\hat{S}^{m}_{t_i}(n,N)\Big)^2 \stackrel{p}{\to} 0\quad \text{as}~n,N\rightarrow \infty.
\end{equation}
At first, for each $k\in \{1,\ldots, d \}$ we shall write

\begin{eqnarray}
\nonumber\sum_{i=1}^{\bar{n}-1}\Big(\Delta\hat{S}^{k}_{t_i}(n,N)\Big)^2&=&
\rho(n)\sum_{i=1}^{\bar{n}-1}\sum_{\ell=1}^{\bar{n}}|y^k_{t_\ell}|^2(\Delta_i\varphi_N(t_\ell,t_i))^2\\
\nonumber & &\\
\nonumber&+& 2\rho(n) \sum_{i=1}^{\bar{n}-1}\sum_{1\le\ell < s\le \bar{n}}y^k_{t_\ell}\Delta_i\varphi_N(t_\ell, t_i)y^k_{t_s}\Delta_i\varphi_N(t_s,t_i)\\
\nonumber& &\\
\label{eq1}&=:& T_1(n,N) + T_2(n,N)
\end{eqnarray}
where $\Delta_i\varphi_N(t_\ell,t_i):= \varphi_N(t_\ell, t_i) - \varphi_N(t_\ell,t_{i-1}); 1\le i \le \bar{n}-1, 1\le \ell \le \bar{n}$. We divide the argument into two steps.

\

\noindent \textit{Analysis of}~$T_1(n,N)$.  It is sufficient to prove that

\begin{eqnarray*}
& &\rho(n)\sum_{i=1}^{\bar{n}-1}\sum_{\ell=1}^{\bar{n}}|y^k_{t_\ell}|^2\Big[(\Delta_i\gamma_N(t_\ell,t_i))^2 + (\Delta_i\theta_N(t_\ell,t_i))^2 \\
& &\\
&+& (\Delta \xi_N(t_\ell,t_i))^2 + (\Delta_i\eta_N(t_\ell,t_i))^2\Big]=O_\mathbb{P}(\rho(n))
\end{eqnarray*}
for each $k\in \{1, \ldots, d\}$. In fact, a simple application of Cauchy-Schwartz inequality and the fact that $\sum_{\ell=1}^{\bar{n}}|y^k_{t_\ell}|^2=1$ yield the following estimates

\begin{small}
\begin{equation}\label{j1}
\rho(n)\sum_{i=1}^{\bar{n}-1}\sum_{\ell=1}^{\bar{n}}|y^k_{t_\ell}|^2 (\Delta_i\gamma_N(t_\ell,t_i))^2
\le \rho(n)\mathbb{E}\sum_{m=1}^{\bar{N}}\sup_{s}|\varepsilon_{t_s}(x_m)|^2\delta(N)\sum_{i=1}^{\bar{n}-1}\sum_{k=1}^{\bar{N}}|
\Delta\varepsilon_{t_i}(x_k)|^2\delta(N),
\end{equation}

\begin{eqnarray}
\nonumber\rho(n)\sum_{i=1}^{\bar{n}-1}\sum_{\ell=1}^{\bar{n}}|y^k_{t_\ell}|^2 (\Delta_i\theta_N(t_\ell,t_i))^2&\le &2
\nonumber\Bigg(\rho(n)\sum_{i=1}^{\bar{n}-1}\sum_{\ell=1}^{\bar{n}} |y^k_{t_\ell}|^2 (\Delta_i\gamma_N(t_\ell,t_i))^2\Bigg)^{1/2}\\
\nonumber& &\\
\nonumber&\times & \Bigg(\rho(n)\sum_{i=1}^{\bar{n}-1}\sum_{\ell=1}^{\bar{n}} \nonumber|\delta(N)\varepsilon^{\top}_{t_\ell}\Delta\varepsilon_{t_i}y^k_{t_\ell}|^2\Bigg)^{1/2}\\
\nonumber& &\\
\label{j2}&+& \rho(n)\sum_{i=1}^{\bar{n}-1}\sum_{\ell=1}^{\bar{n}} |y^k_{t_\ell}|^2 \Big( (\Delta_i\gamma_N(t_\ell,t_i))^2 + (\delta(N)\varepsilon^{\top}_{t_\ell}\Delta\varepsilon_{t_i})^2\Big),
\end{eqnarray}

\begin{equation}\label{j3}
\rho(n)\sum_{i=1}^{\bar{n}-1}\sum_{\ell=1}^{\bar{n}}|y^k_{t_\ell}|^2  (\Delta_i\xi_N(t_\ell,t_i))^2\le C\rho(n) \sup_{t}\|\varepsilon_t\|^2_{\mathbb{R}^{\bar{N}}}\delta(N)\sum_{r=1}^d\sum_{i=1}^{\bar{n}-1}|\Delta Z^r_{t_i}|^2\|\lambda_r\|^2_{\mathbb{R}^{\bar{N}}}\delta(N)
\end{equation}

\begin{equation}\label{j4}
\rho(n)\sum_{i=1}^{\bar{n}-1}\sum_{\ell=1}^{\bar{n}}|y^k_{t_\ell}|^2  (\Delta_i\eta_N(t_\ell,t_i))^2 \le C\rho(n)\sum_{i=1}^{\bar{n}-1}\|\Delta \varepsilon_{t_i}\|^2_{\mathbb{R}^{\bar{N}}}\delta(N)\sum_{r=1}^d\sup_t|Z^r_t|^2\|\lambda_r\|^2_{\mathbb{R}^{\bar{N}}}\delta(N)
\end{equation}
\end{small}
The estimates~(\ref{j1}),(\ref{j2}), (\ref{j3}) and~(\ref{j4}) allow us to conclude that $T_1(n,N)=O_{\mathbb{P}}(\rho(n))$.

\

\noindent \textit{Analysis of}~$T_2(n,N)$. The estimates for the crossing terms are more involved. Let us split $T_2(n,N)$ according to the terms $\Delta_i\gamma_N(t_{\ell},t_i),$ $\Delta_i\theta_N(t_\ell,t_i)$, $\Delta_i\xi_N(t_\ell,t_i)$ and $\Delta_i\eta_N(t_\ell,t_i)$ as follows. To shorten  notation, in the sequel we denote $J(k,n)=\rho(n)\sum_{1\le\ell < s\le \bar{n}}|y^k_{t_\ell}y^k_{t_s}|$. Cauchy-Schwartz inequalities and routine algebraic manipulations yield the following estimates

\[
\rho(n) \sum_{i=1}^{\bar{n}-1}\sum_{1\le\ell < s\le \bar{n}}|y^k_{t_\ell} \Delta_i\gamma_N(t_\ell, t_i)y^k_{t_s} \Delta_i\xi_N(t_s,t_i)|\le C J(k,n)\Big(\sup_{t}\|\varepsilon_t\|^2_{\mathbb{R}^{\bar{N}}}\delta(N)\Big)^{1/2}\]

\[\Bigg[(\mathbb{E}\sup_{t}\|\varepsilon_t\|^2_{\mathbb{R}^{\bar{N}}}\delta(N))\sum_{i=1}^{\bar{n}-1}\sum_{r=1}^p|
\Delta Z^r_{t_i}|^2\|\lambda_r\|^2_{\mathbb{R}^{\bar{N}}}\delta(N)\Bigg]^{1/2}
\Bigg[\sum_{i=1}^{\bar{n}-1}\mathbb{E}\|\Delta\varepsilon_{t_i}\|^2_{\mathbb{R}^{\bar{N}}}\delta(N)\Bigg]^{1/2},\]

\begin{eqnarray*}
\rho(n) \sum_{i=1}^{\bar{n}-1}\sum_{1\le\ell < s\le \bar{n}}|y^k_{t_\ell} \Delta_i\xi_N(t_\ell,t_i)y^k_{t_s} \Delta_i\eta_N(t_s,t_i)|&\le& C J(k,n)\sum_{r,q=1}^pO_{\mathbb{P};(r,q)}(1)\\
& &\\
&\times&\Bigg[\sum_{i=1}^{\bar{n}-1}\|\Delta\varepsilon_{t_i}
\|^2_{\mathbb{R}^{\bar{N}}}\delta(N)\Bigg]^{1/2},
\end{eqnarray*}

\begin{eqnarray*}
\rho(n) \sum_{i=1}^{\bar{n}-1}\sum_{1\le\ell < s\le \bar{n}}|y^k_{t_\ell} \Delta_i\theta_N(t_\ell,t_i)y^k_{t_s} \Delta_i \xi_N(t_s,t_i)|&\le& C J(k,n)\sum_{r=1}^pO_{\mathbb{P};r}(1)\sup_t\|\varepsilon\|_{\mathbb{R}^{\bar{N}}}\\
& &\\
&\times& \Bigg(\sum_{i=1}^{\bar{n}-1}|\Delta Z^r_{t_i}|^2\sum_{i=1}^{\bar{n}-1} \|\Delta\varepsilon_{t_i}
\|^2_{\mathbb{R}^{\bar{N}}}\delta(N)\Bigg)^{1/2}.
\end{eqnarray*}
We also shall write
\begin{eqnarray*}
\rho(n) \sum_{i=1}^{\bar{n}-1}\sum_{1\le\ell < s\le \bar{n}}y^k_{t_\ell} \Delta_i\theta_N(t_\ell,t_i)y^k_{t_s} \Delta_i \xi_N(t_s,t_i)&=& \sum_{r,j=1}^pO_{\mathbb{P};(r,j)}(1)\rho(n)\sum_{1\le\ell < s\le \bar{n}}y^k_{t_\ell}y^k_{t_s}\\
& &\\
&\times&\langle \varepsilon_{t_\ell},\lambda_r\rangle_{\mathbb{R}^{\bar{N}}}\delta(N)
\langle \varepsilon_{t_s},\lambda_j\rangle_{\mathbb{R}^{\bar{N}}}\delta(N)
\end{eqnarray*}
and

$$
\rho(n) \sum_{i=1}^{\bar{n}-1}\sum_{1\le\ell < s\le \bar{n}}|y^k_{t_\ell} \Delta_i\gamma_N(t_\ell,t_i)y^k_{t_s} \Delta_i\gamma_N(t_s,t_i)|\le J(k,n)O_{\mathbb{P}}(1)\sum_{i=1}^{\bar{n}-1}\mathbb{E}\|\Delta \varepsilon_{t_i}\|^2_{\mathbb{R}^{\bar{N}}}\delta(N),
$$

$$
\rho(n) \sum_{i=1}^{\bar{n}-1}\sum_{1\le\ell < s\le \bar{n}}|y^k_{t_\ell} \Delta_i\gamma_N(t_\ell,t_i)y^k_{t_s} \Delta_i\eta_N(t_s,t_i)|\le C J(k,n)\Bigg(\sum_{i=1}^{\bar{n}-1}\|\Delta \varepsilon_{t_i}\|^2_{\mathbb{R}^{\bar{N}}}\delta(N)\Bigg)^{1/2},
$$

$$
\rho(n) \sum_{i=1}^{\bar{n}-1}\sum_{1\le\ell < s\le \bar{n}}|y^k_{t_\ell} \Delta_i\eta_N(t_\ell, t_i)y^k_{t_s} \Delta_i\eta_N(t_s,t_i)|\le C J(k,n)\sum_{q,r=1}^pO_{\mathbb{P};(q,r)}\Bigg(\sum_{i=1}^{\bar{n}-1}\|\Delta \varepsilon_{t_i}\|^2_{\mathbb{R}^{\bar{N}}}\delta(N)\Bigg).
$$
The remainder terms in $T_2(n,N)$ are analogous. Summing up the above estimates, we conclude that $T_2(n,N)\rightarrow 0$ in probability as $n,N\rightarrow \infty$. From identities~(\ref{identM}), (\ref{qcon3}), (\ref{eq1}) and (\ref{qcon1}), we conclude the proof.
\end{proof}

The next step is the analysis of the convergence of the loading factor estimators defined as follows. Let

$$\hat{\Lambda}^\top: = \rho(n)\hat{Y}^\top(\hat{d})\mathbb{X}\in \mathbb{M}_{\hat{d}\times N}$$
and

$$\hat{\varphi}_i(x):=\sqrt{\rho(n)}\sum_{k=1}^{\bar{n}} y^i_{t_k} X_{t_k}(x), \quad \xi_k(x): = ((\mathcal{A}^{-1})^\top \lambda(x))_k$$
for $a\le x\le b, 1\le i\le \hat{d}, 1\le k\le d.$ Since $\mathcal{A}\in \mathbb{M}_{d\times d}$ is non-singular a.s, then $\{\xi_1(\omega, \cdot), \ldots, \xi_d(\omega,\cdot)\}$ is a basis for $V$ for almost all $\omega\in \Omega$. More importantly,

 $$r_t = \phi_t + \sum_{k=1}^d Y^k_t \xi_k; 0\le t\le T.$$
where $Y = \mathcal{A}Z\in \mathcal{X}^d$.

\begin{proposition}\label{conv-norm}
If Assumptions \textbf{(D1, D2, D3, D4, Q1, Q5)} hold true, then
$$\sum_{j=1}^{\hat{d}}\|\hat{\varphi}_j - \xi_j\|^ 2_E\stackrel{p}\to 0 $$
as $n,N\rightarrow \infty$.
\end{proposition}
\begin{proof}
Let us fix $i\in \{1, \ldots, d\}$. Since $\hat{d}$ is an integer-valued consistent estimator for $d$, then we shall assume that $\hat{d}=d$. Under \textbf{(D1, D2, D3, D4, Q1)}, $\{\xi_i; 1\le i \le d\}$ is a well-defined random basis for $V$. By the very definition,

\begin{eqnarray*}
\hat{\varphi}_i(x) &=& \sqrt{\rho(n)}\sum_{k=1}^{\bar{n}} y^i_{t_k}X_{t_k}(x) = \sum_{m=1}^d\Big(\sum_{k=1}^{\bar{n}}\sqrt{\rho(n)} y^i_{t_k}Z^m_{t_k}\Big)\lambda_m(x)\\
& &\\
&+& \sum_{k=1}^{\bar{n}}\sqrt{\rho(n)} y^i_{t_k}\varepsilon_{t_k}(x)=:R_{1,i}(x) + R_{2,i}(x), x\in [a,b].
\end{eqnarray*}
Let us recall that for any $f\in E$, we can compute the Sobolev norm as follows $\|f\|^2_E = \sup_{\Pi}\sum_{s_j\in \Pi}\frac{|\Delta f(s_j)|^2}{\Delta s_j}< \infty$ where the sup is taken over all partitions $\Pi$ of $[a,b]$. See e.g Prop 1.45 in~\cite{friz} for further details. If $\Pi = \{s_j\}_{j=1}^M$ is a partition of $[a,b]$, then

\begin{eqnarray*}
\sum_{s_j\in \Pi}\frac{|\Delta R_{2,i}(s_j)|^2}{\Delta s_j}&=& \sum_{s_j\in \Pi}\sum_{k=1}^{\bar{n}}\rho(n)|y^i_{t_k}|^2\frac{|\Delta \varepsilon_{t_k}(s_j)|^2}{\Delta s_j}\\
& &\\
&+& 2\rho(n)\sum_{1\le k < m\le \bar{n}} y^i_{t_k} y^i_{t_m}\sum_{s_j\in \Pi}\frac{\Delta \varepsilon_{t_k}(s_j)}{\Delta s_j}\frac{\Delta \varepsilon_{t_m}(s_j)}{\Delta s_j}\\
& &\\
&=:&I_{1,i} +I_{2,i}
\end{eqnarray*}
Since $\rho(n)\hat{Y}^\top(d) \hat{Y}(d) = I_d~a.s$, then \textbf{(Q5)} yields

$$|I_{1,i}|\le \sum_{\ell=1}^{\bar{n}}\rho(n)|y^i_{t_\ell}|^2\|\varepsilon_{t_\ell}\|^2_{E}\le \sup_{0\le t\le T}\|\varepsilon_t\|_E^2\rho(n)\stackrel{p}\to 0$$
as $n\rightarrow\infty$. Cauchy-Schwartz inequality and \textbf{(Q5)} yield

$$|I_{2,i}|\le 2\rho(n)\sum_{1\le \ell < s\le \bar{n}}|y^i_{t_\ell}y^i_{t_s}|\|\varepsilon_{t_\ell}\|_E\|\varepsilon_{t_s}\|_E\stackrel{p}\to 0$$
as $n\rightarrow \infty$. From Lemma \ref{Gmatrix}, we know that $(G^\top)^{-1} = \mathcal{A}$ so that $(\mathcal{A}^{-1})^\top = G$. Since $\{\lambda_1, \ldots, \lambda_d\}\subset E$, then we obviously have $\|R_{1,i}(\cdot)-\xi_i(\cdot)\|_E\stackrel{p}\to 0$ as $n\rightarrow \infty$. This concludes the proof.
\end{proof}

In the sequel, $\hat{p}$ is any consistent estimator for $dim~\mathcal{Q}$ based on $X$. See Appendix for details. Let $\hat{\mathcal{L}}\in \mathbb{M}_{\hat{d}\times \hat{d}}$ be the matrix whose rows are given by $\hat{\mathcal{L}}_i:=\hat{v}_i; 1\le i\le \hat{d}$, where $\{\hat{v}_1, \ldots, \hat{v}_{\hat{d}}\}$ is an orthonormal eigenvector set of the matrix $\widehat{[Y]}_T$ (see (\ref{hatquadmatrix})) associated to the ordered eigenvalues $\hat{\theta}_1\ge \hat{\theta}_2\ge \ldots \ge \hat{\theta}_{\hat{d}}$. Let us define

\begin{equation}\label{PCAest}
\hat{Z}^j_{t_i}:=~j\text{-th component of}~\hat{\mathcal{L}}\hat{Y}_{t_i}; 0\le i \le \bar{n}, 1\le j\le \hat{d}
\end{equation}
and
$$[\hat{Z}^j]_T:=\sum_{i=1}^{\bar{n}}\Big(\hat{Z}^j_{t_i} - \hat{Z}^j_{t_{i-1}}\Big)^2$$
over a sample $0=t_0 < t_1 < \ldots < t_{\bar{n}}=T$. By the very definition, $[\hat{Z}^j]_T = \hat{\theta}_j; 1\le j\le \hat{d}$.

Now we are able to present the main result of this article. Before this, we need an elementary lemma from linear algebra.

\begin{lemma}\label{Gram-Schmidt}
Let $v_1,\ldots,v_d$ be a set of $d$ linearly independent vectors in a real Hilbert space $H$ with inner product $\langle \cdot, \cdot \rangle_H$ and  $\mathcal{V} = span\{v_1,\ldots,v_d\}$. Let $T:\mathcal{V}\to \mathcal{V}$ be an orthogonal matrix. If $\tau_1,\ldots,\tau_d$ is the Gram-Schmidt orthonormalization of $v_1,\ldots,v_d$ and $w_1,\ldots,w_d$ is the Gram-Schmidt orthonormalization of $Tv_1,\ldots,Tv_d$, then we have
$$w_i = T\tau_i; \quad i=1,\ldots,d,\quad.$$
\end{lemma}
\begin{proof}
The proof follows by just observing that for each $v\in \mathcal{V}$, $\|Tv\|_H = \|v\|_H,$ and for each $u\in V$, we have $T\left(Proj_v u\right) = Proj_{Tv}(Tu)$, where $Proj_v u = v\<u,v\>_H/\|v\|_H.$
\end{proof}

\begin{theorem}\label{mainTHPAPER}
Let $r$ be the stochastic PDE (\ref{spde1}) satisfying Assumptions (A1-A2). Assume the existence of a factor representation satisfying Assumption (A2) and \textbf{(D1, D2, D3, D4, Q1, Q2, Q3, Q4, Q5)}. For a given $h\in dom~(A)$, let $\mathcal{V}^h_t = \phi_t+ V; 0\le t\le T$ be the minimal foliation generated by $V$ such that $r_0 = h\in \mathcal{V}^h_0$ and we set

$$\widehat{\mathcal{N}}:=~span~\Big\{ (\hat{\mathcal{L}}\hat{\varphi})_{\hat{p}+1}, \ldots, (\hat{\mathcal{L}}\hat{\varphi})_{\hat{d}}   \Big\}\quad \widehat{\mathcal{Q}}:=~span~\Big\{ (\hat{\mathcal{L}}\hat{\varphi})_{1}, \ldots, (\hat{\mathcal{L}}\hat{\varphi})_{\hat{p}}   \Big\}.$$
Then, $V = \mathcal{Q}\oplus \mathcal{N}~a.s$ and

$$\max\{d(\widehat{\mathcal{N}}, \mathcal{N}),d(\widehat{\mathcal{Q}}, \mathcal{Q}) \}\stackrel{p}\to 0 ~\text{as}~n,N\rightarrow \infty.$$
Moreover,
\begin{equation}\label{finalorder}
[\hat{Z}^1]_T\ge\ldots\ge [\hat{Z}^{\hat{p}}]_T~a.s,~[\hat{Z}^{i}]_T\simeq 0,\quad~\hat{p}+1\le i\le \hat{d}~\text{as}~n,N\rightarrow \infty,
\end{equation}
and
$$
\sum_{j=1}^{\hat{d}}\hat{\theta}^2_j\stackrel{p}\to \|Q_T\|^2_{(2)}
$$
as $n,N\rightarrow \infty$.
\end{theorem}
\begin{proof}
From Assumptions (A1-A2), we shall fix a pair $(Z,\lambda)$ which realizes

$$r_t(x) = \phi_t(x) + \sum_{j=1}^dZ^j_t\lambda_j(x)=\phi_t(x) + \sum_{j=1}^d Y^j_t\xi_j(x);~0\le t\le T$$
where $V=\text{span} \{\lambda_1, \ldots,\lambda_d\} = \text{span} \{\xi_1, \ldots, \xi_d\}$, $Z$ is a continuous semimartingale satisfying Assumption (A2) and \textbf{(D1, D2, D3, D4)}, \textbf{(Q1, Q2, Q3, Q4, Q5)}. Here, we set $Y=\mathcal{A}Z$ and $\xi(x)=(\mathcal{A}^{-1})^\top\lambda(x)$, where $\mathcal{A}$ is given by (\ref{Aoperator}). From Remark \ref{invprop}, $Y$ satisfies Assumption $(A2)$ as well.

To shorten notation, we abbreviate Gram-Schmidt orthonormalization by GSO. Let

$$\widetilde{\mathcal{N}} = span\{(\widehat{\mathcal{L}}\xi)_{\widehat{p}+1},\ldots,(\widehat{\mathcal{L}}\xi)_{\widehat{d}}\},$$
and
$$\widetilde{\mathcal{Q}} = \text{span}\{(\widehat{\mathcal{L}}\xi)_{1},\ldots,(\widehat{\mathcal{L}}\xi)_{\widehat{p}}\}.$$

Following the same lines as in the proof of Theorem \ref{WDTh} and noting that (see Remark~\ref{invprop})
$$\Phi(\widetilde{\mathcal{N}}) = Ker([\widehat{Y}]_T)\quad\hbox{and}\quad \Phi(\mathcal{N}) = Ker([{Y}]_T),$$
we obtain
\begin{equation}\label{dist1}
d(\widetilde{\mathcal{N}},\mathcal{N})\stackrel{p}{\longrightarrow} 0,\quad\hbox{and}\quad d(\widetilde{\mathcal{Q}},\mathcal{Q})\stackrel{p}{\longrightarrow} 0,
\end{equation}
as $n, N\rightarrow \infty$.
By using the triangle inequality, we obtain
$$d(\widehat{\mathcal{N}},\mathcal{N}) \leq d(\widehat{\mathcal{N}},\widetilde{\mathcal{N}}) + d(\widetilde{\mathcal{N}},\mathcal{N}),$$
and from equation \eqref{dist1}, it is enough to prove that $d(\widehat{\mathcal{N}},\widetilde{\mathcal{N}})\stackrel{p}{\longrightarrow} 0.$ as $n, N\rightarrow \infty$.

Let $\{\widehat{\varphi}_1,\ldots,\widehat{\varphi}_{\widehat{d}}\}$ and $\{\xi_1,\ldots,\xi_d\}$ be as in Proposition \ref{conv-norm}. Let $n$ and $N$ be large enough so that $\widehat{p}=p$ and $\widehat{d}=d$. Let $\{\tau_1,\ldots,\tau_d\}$ and $\{\widehat{\tau}_1,\ldots,\widehat{\tau}_d\}$ be the GSO of $\{\xi_1,\ldots,\xi_d\}$ and $\{\widehat{\varphi}_1,\ldots,\widehat{\varphi}_d\}$, respectively. Lemma \ref{Gram-Schmidt} allows us to state that $\{\widehat{\mathcal{L}}\tau_1,\ldots, \widehat{\mathcal{L}}\tau_d\}$ is the GSO of $\{\widehat{\mathcal{L}}\xi_1,\ldots, \widehat{\mathcal{L}}\xi_d\}$ and $\{\widehat{\mathcal{L}}\widehat{\tau}_1,\ldots,\widehat{\mathcal{L}}\widehat{\tau}_d\}$ is the GSO of $\{\widehat{\mathcal{L}}\widehat{\varphi}_1,\ldots, \widehat{\mathcal{L}}\widehat{\varphi}_d\}$.

From the orthonormalization procedure, for each $k\leq d$, we have
$$span\{\widehat{\mathcal{L}}\widehat{\varphi}_1,\ldots, \widehat{\mathcal{L}}\widehat{\varphi}_k\} = span\{\widehat{\mathcal{L}}\widehat{\tau}_1,\ldots,\widehat{\mathcal{L}}\widehat{\tau}_k\}$$
and
$$span\{\widehat{\mathcal{L}}\xi_1,\ldots, \widehat{\mathcal{L}}\xi_k\} = span\{\widehat{\mathcal{L}}\tau_1,\ldots, \widehat{\mathcal{L}}\tau_k\}.$$
Thus,
$$\widehat{\mathcal{N}} = span\{\widehat{\mathcal{L}}\widehat{\tau}_1,\ldots,\widehat{\mathcal{L}}\widehat{\tau}_p\}\quad\hbox{and}\quad \mathcal{N} = span\{\widehat{\mathcal{L}}{\tau}_1,\ldots,\widehat{\mathcal{L}}{\tau}_p\}.$$
Therefore, since $\Phi$ is an isometry, we have
$$d(\widehat{\mathcal{N}},\widetilde{\mathcal{N}}) = D(\widehat{\mathcal{N}},\widetilde{\mathcal{N}}) = 1 - \frac{1}{d}\sum_{i,j} (\<(\widehat{\mathcal{L}}\hat{\tau})_i,(\widehat{\mathcal{L}}\tau)_j\>)^2~a.s.$$

Let us work with the quantity inside the brackets, and let us introduce some notation: Denote the matrix of $\hat{\mathcal{L}}$ by $\{\hat{a}_{ij}\}$, i.e., for any vector $v\in\mathbb{R}^d$,

 $$(\widehat{\mathcal{L}}v)_i := \sum_j \hat{a}_{ij}v_j.$$

Note that since the transformation $\widehat{\mathcal{L}}$ is orthogonal, we have
 $$\sum_k \hat{a}_{ik} \hat{a}_{jk} = \delta_{ij}~a.s.$$

Observe that from Proposition \ref{conv-norm}, we have that $\<\hat{\tau}_i,\tau_j\> \stackrel{p}{\to} \delta_{ij}$ as $n, N\rightarrow \infty$.
Since $\widehat{\mathcal{L}}$ is orthogonal and the set of orthogonal matrices
is compact, the set $\{\hat{a}_{i,j}\}$ is uniformly bounded in $n$ and $N$, so that
$$\left| \sum_{k\neq p} \hat{a}_{ik}\hat{a}_{jp} \<\hat{\tau}_k,\tau_p\>\right| \stackrel{p}{\to} 0,$$
and
$$\left| \sum_{k} \hat{a}_{ik}\hat{a}_{jk} (\<\hat{\tau}_k,\tau_k\>-1)\right| \stackrel{p}{\to} 0$$
as $n, N\rightarrow \infty$.

Therefore,
 \begin{eqnarray*}
  \left| \sum_{k,p} \hat{a}_{ik}\hat{a}_{jp} \<\hat{\tau}_k,\tau_p\> - \delta_{ij}\right| &=& \left| \sum_{k,p} \hat{a}_{ik}\hat{a}_{jp} \<\hat{\tau}_k,\tau_p\> - \sum_{k} \hat{a}_{ik}\hat{a}_{jk}\right|\\
  &\leq& \left| \sum_{k} \hat{a}_{ik}\hat{a}_{jk} (\<\hat{\tau}_k,\tau_k\>-1)\right| + \left| \sum_{k\neq p} \hat{a}_{ik}\hat{a}_{jp} \<\hat{\tau}_k,\tau_p\>\right|\\
  &\stackrel{p}{\to}& 0,
 \end{eqnarray*}
as $n,N\rightarrow \infty$, which thus implies that

$$\frac{1}{d}\sum_{i,j} (\<(\widehat{\mathcal{L}}\widehat{\tau})_i,(\widehat{\mathcal{L}}\tau)_j\>)^2=\frac{1}{d}\sum_{i,j}\left(\sum_{k,p} \hat{a}_{ik}\hat{a}_{jp} \<\hat{\tau}_k,\tau_p\>\right)^2 \stackrel{p}{\to} 1,$$
and then,
$$d(\widehat{\mathcal{N}},\widetilde{\mathcal{N}}) \stackrel{p}{\longrightarrow} 0$$
as $n, N \rightarrow \infty$.
The proof for the statement $\lim_{n,N\rightarrow \infty}d(\widehat{\mathcal{Q}},{\mathcal{Q}}) =0$ in probability follows from the same reasoning as in the proof of Theorem \ref{WDTh}. Let us now check the ordering (\ref{finalorder}). Let $\hat{\theta}_1\ge 1\ldots\ge \hat{\theta}_{\hat{d}}~a.s$ be the eigenvalues of the self-adjoint non-negative matrix $\widehat{[Y]}_T$ arranged in decreasing order a.s. By the very definition,

$$\hat{\theta}_i = [\hat{Z}^i]_T;1\le i\le \hat{d}~a.s$$
Moreover, the eigenvalues are continuous functions of the entries of the matrix and $\hat{p} = p$ and $\hat{d}=d$ for $n,N$ large enough. Then,

$$ \max_{\substack{\hat{p}+1\le i\le \hat{d}}}\hat{\theta}_i\stackrel{p}{\to}0$$
as $n.N\rightarrow \infty$. This shows that (\ref{finalorder}) holds. Lastly, by the very definition, the matrix of the random operator $Q_T:V\rightarrow V$ computed along the basis $\{\xi_1, \ldots, \xi_d\}$ is given by $\{[Y^i,Y^j]_T; 1\le i,j\le d\}\in \mathbb{M}_{d\times d}$. Therefore, $\|Q_T\|^2_{(2)} = \|[Y]_T\|^2_{(2)}~a.s$. Proposition \ref{conv-qua-var} yields

$$\|\widehat{[Y]}_T\|^2_{(2)} = \sum_{j=1}^{\hat{d}}\hat{\theta}^2_j = \sum_{j=1}^{\hat{d}}[\hat{Z}^j]^2_T\stackrel{p}{\to} \|Q_T\|^2_{(2)}$$
as $n,N\rightarrow \infty$. This concludes the proof.
\end{proof}

\section{Simulation Studies and Applications}\label{numerics}
In this section, we present some numerical results to illustrate the methodology developed in this article.

\subsection{Semimartingale PCA}
In this section, we illustrate the estimation of the factor spaces $(\mathcal{W},\mathcal{D})$ based on a finite-dimensional semimartingale system sampled in high-frequency. In particular, the goal is to illustrate Proposition \ref{ranking}. In the simulation below, we assume that one observes a 4-dimensional semimartingale as follows: We consider a Markov diffusion

$$dM_t = \mu(M_t)dt + \sigma(M_t)dB_t$$
driven by a 3-dimensional Brownian motion $B = (B^1, B^2, B^3)$ and the vector fields $\mu:\mathbb{R}^4\rightarrow \mathbb{R}^4$ and $\sigma:\mathbb{R}^4\rightarrow \mathbb{M}_{4\times 3}$ are given by $\mu(x_1, \ldots, x_4) = (x_2, -2x_1+x_3, x_4,-x_1)$ and

$$
\sigma(x_1, \ldots,x_4)=\left(
  \begin{array}{ccc}
    1 & 0 & x_2 \\
    0 & 1 & 0 \\
    0 & 0 & x_2 \\
    0 & 0 & x_2 \\
  \end{array}
\right)
$$
One can easily check $\mathcal{W} =\{ B^1, B^2, \int M^2dB^3\}$ and since $M$ is a truly 4-dimensional semimartingale, then $\mathcal{M} = \mathcal{W}\oplus \mathcal{D}$ where $dim~\mathcal{D}=1$. The observation times are taken to be equidistant: $t^n_k=\frac{2\pi}{n-1}k; k=0, \ldots, n-1$ where the total number of observations is $n=2000$. The estimated factors in Figure \ref{fig1} are ranked in terms quadratic variation (see Theorem \ref{WDTh}) and we clearly observe that $\hat{J}^4$ identifies a null quadratic variation factor which generates $\mathcal{D}$. The quadratic variation explained by the principal components are given by Table 1, where $\hat{\eta}_i=\frac{\sum_{j=1}^i\hat{\theta}_j}{\sum_{r=1}^4\hat{\theta}_r}, 1\le i \le 4$ and $\hat{\theta}_i$ is the $i$-th estimated eigenvalue related to $i$-th estimated principal component $\hat{J}^i$.

\begin{table}[htb]
\caption{ Quadratic variation explained by the principal components}
\begin{center}
\begin{tabular}{|c|c|c|c|}
  \hline
  $\hat{\eta}_1$ &$\hat{\eta}_2$  & $\hat{\eta}_3$ & $\hat{\eta_4}$ \\
  \hline
  $0.7523$ & $0.9021$ & $0.9996$ & $1.0000$\\
  \hline
\end{tabular}
\end{center}
\end{table}
\noindent

\begin{figure}[htb]
\begin{center}
\includegraphics[scale=0.55]{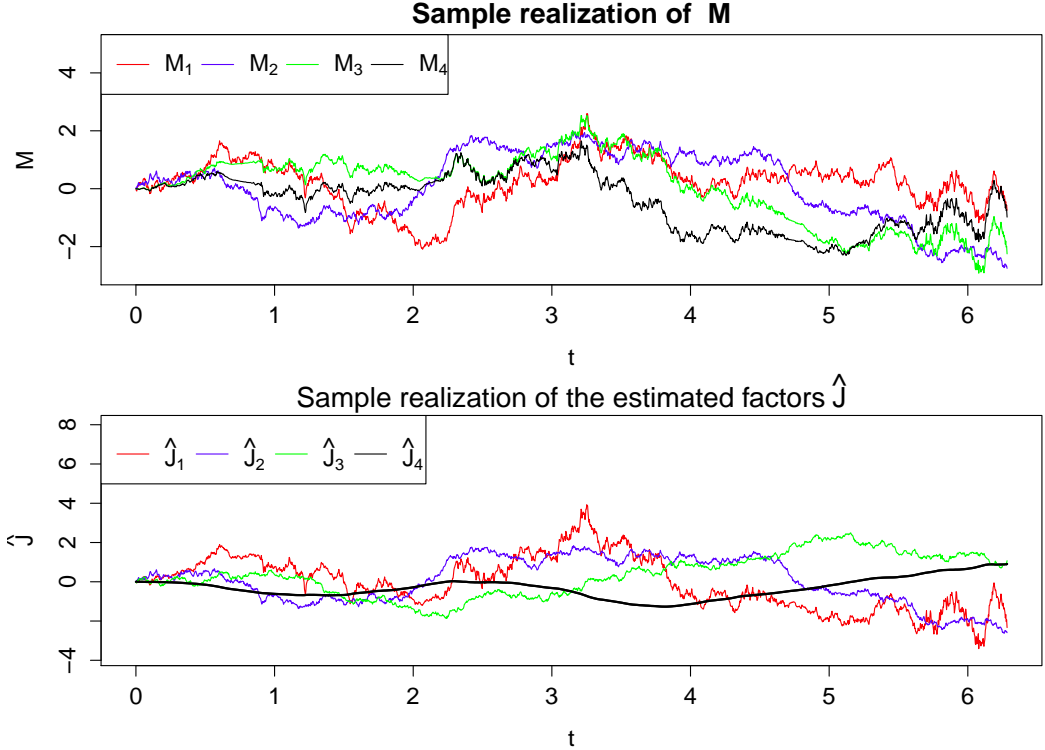}
\end{center}
\caption{Estimation of a basis for~$\mathcal{W}$ and $\mathcal{D}$}
\label{fig1}
\end{figure}

\subsection{Variance versus quadratic variation}
In this section, the goal is to illustrate that any naive attempt to implement standard factor models towards dimension reduction in term of quadratic variation is hopeless. For this purpose, we consider two very simple space-time two-dimensional semimartingales driven by a single Brownian motion $B$.

$$X_t = B_t\lambda_1(x)  + \big(sin(15 t) - B_t\big)\lambda_2(x)$$
$$U_t =  B_t\lambda_1(x)  + \big(sin(3 t) - B_t\big)\lambda_2(x)$$
where $B$ is a one-dimensional Brownian motion, $\lambda_1(x) = x\text{cos}(x)$ and $\lambda_2(x) = \text{cos}(x) - x \text{sin}(x); 0 \le x\le 5, 0\le t\le 2\pi$. In the sequel, the drift components are denoted by $\Gamma^1_t = \text{sin}(15t)$, $\Gamma^2_t = \text{sin}~(3t)$ and we set $H^1=(B,\Gamma^1-B)$ and $H^2 = (B,\Gamma^2-B)$. Let
$\mathcal{M}(H^1)$ and $\mathcal{M}(H^2)$ be the dynamic spaces generated by $H^1$ and $H^2$, respectively. We clearly have

$$\mathcal{M}(H^1) = \text{span}\{B\}\oplus \text{span}~\{\Gamma^1\}, \quad\mathcal{M}(H^2) = \text{span}~\{B\}\oplus \text{span}~\{\Gamma^2\}$$
Here, in the time variable, the observation times are taken to be equidistant: $t^k_n=\frac{2\pi}{n-1}k; k=0, \ldots, n-1$ where the total number of observations is $n=2000$. In the spatial variable, the observation times are taken to be equidistant: $x^n_k=\frac{5}{n-1}k; k=0, \ldots, n-1$ where the total number of observations is $n=31$. The estimated pair of factors provided by the high-frequency factor model based on variance will be denoted by $(\widehat{Y}^1,\widehat{Y}^2)$. Here, $\widehat{Y}^1$ is the estimator of the leading factor component in terms variance. The estimated pair of factors provided by the high-frequency factor model based on quadratic variation will be denoted by $(\widehat{Z}^1,\widehat{Z}^2)$. Here, $\widehat{Z}^1$ is the leading factor component in terms of quadratic variation (see (\ref{PCAest})).

In Figure \ref{fig33}, we clearly see that the factor analysis based on second moments is not able to identify $(B,\Gamma^1)$. The leading component estimated factor $\widehat{Y}^1$ resembles a bounded variation process with large variance and the second estimated factor $\widehat{Y}^2$ is essentially the first one distorted by the Brownian paths in such way that the true pair $(B,\Gamma^1)$ is by no means identified. In strong contrast, Figure \ref{fig33} clearly reports that the estimated pair $(\widehat{Z}^1, \widehat{Z}^2)$ identifies the pair $(B,\Gamma^1)$. We stress that in this two-dimensional setting, the true factors can be estimated up to multiplicative constants so that the results presented in Figures \ref{fig22} and \ref{fig33} shows a very consistent estimation of $\mathcal{M}(H^1)$ by using our methodology. More importantly, the correct splitting and ranking in term of quadratic variation is fairly estimated. This numerical example illustrates the use of factor analysis based on variance to infer volatility (quadratic variation) does not have any sound basis even in a very simple space-time semimartingale model given by $X$ above.

Figure \ref{fig44} presents the results for the model $U$. In this numerical experiment, the goal is to illustrate that null quadratic variation factors with large variance may be the leading component by using standard factor models in terms of variance. In Figure \ref{fig44}, we report that $\widehat{Y}^2$ estimates well the Brownian component $B$ responsible for the quadratic variation subspace of $U$, $\widehat{Y}^1$ estimates well the bounded variation component $\Gamma^2$ responsible for the null quadratic variation subspace of $U$. However, the correct leading component of the space-time semimartingale $U$ is the Brownian motion and not $\Gamma^2$. This simple example shows that prioritising components with large variance by using standard factor models may be completely superfluous in terms of quadratic variation. This shows that dimension reduction for semimartingale systems can not be accurately performed by using classical dimension reduction based on variance.

\begin{figure}[htb]
\begin{center}
\includegraphics[scale=0.55]{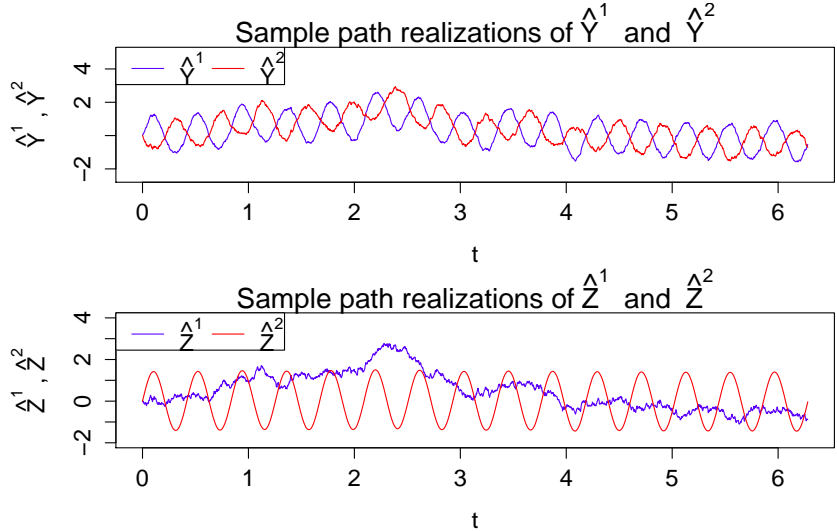}
\end{center}
\caption{Estimated factors of $X$.}
\label{fig22}
\end{figure}

\begin{figure}[htb]
\begin{center}
\includegraphics[scale=0.55]{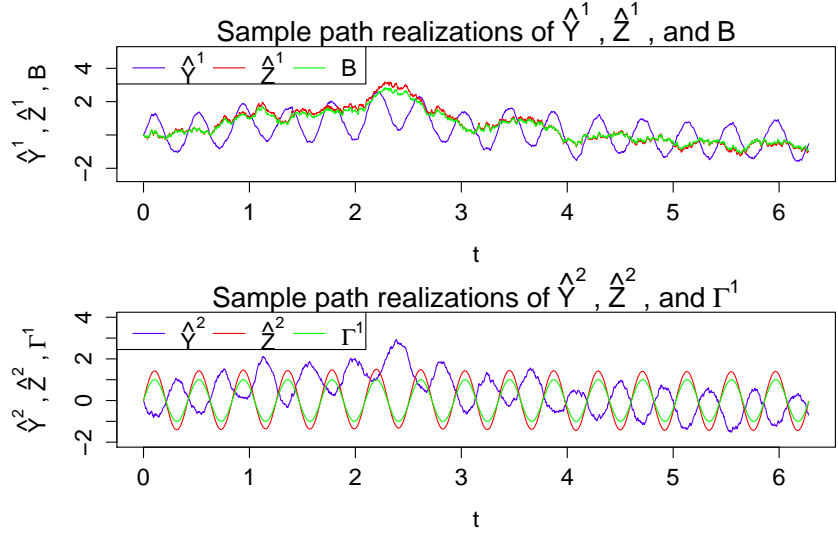}
\end{center}
\caption{Comparation of the estimated factors of Figure \ref{22}}
\label{fig33}
\end{figure}

\begin{figure}[htb]
\begin{center}
\includegraphics[scale=0.55]{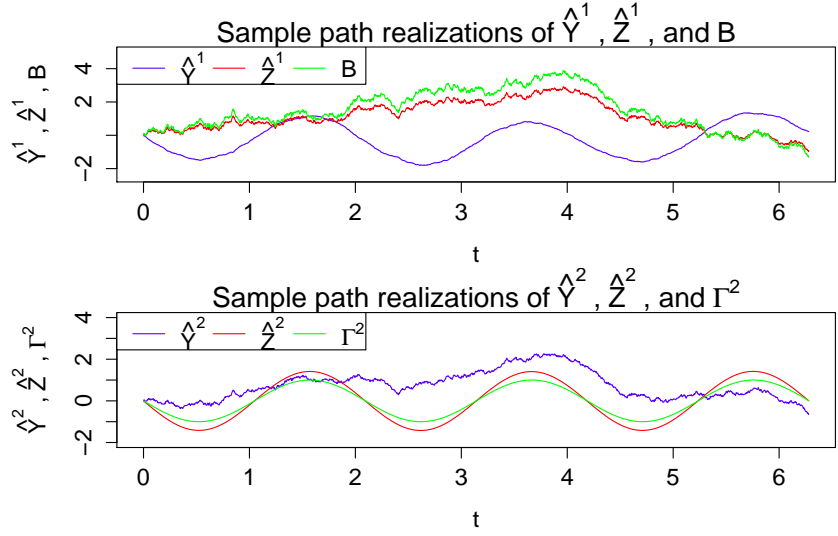}
\end{center}
\caption{Specification of the standard factor analysis for a null quadratic variation component as the leading semimartingale component in terms of variance}
\label{fig44}
\end{figure}
\subsection{Estimating finite-dimensional realizations from a SPDE}

Here, we illustrate our methodology with some applications to space-time semimartingale models. The first example is based on a Markov diffusion

$$dM_t = \mu(M_t)dt + \sigma(M_t)dB_t$$
driven by a 3-dimensional Brownian motion $B = (B^1, B^2, B^3)$ and the vector fields $\mu:\mathbb{R}^4\rightarrow \mathbb{R}^4$ and $\sigma:\mathbb{R}^4\rightarrow \mathbb{M}_{4\times 3}$ are given by $\mu(x_1, \ldots, x_4) = (x_2,-2x_1+x_3,x_4,-x_1)$ and

$$
\sigma(x_1, \ldots,x_4)=\left(
  \begin{array}{ccc}
    1 & 0 & 0 \\
    0 & x_2 & 0 \\
    0 & 0 & x_1 \\
    0 & 0 & 0 \\
  \end{array}
\right)
$$
where,

\begin{equation}\label{stsmlast}
r_t = \sum_{i=1}^4 M_t^i \lambda_i,
\end{equation}
and $\lambda_1 = x\cos (x), \lambda_2 = \cos(x)-x\sin(x), \lambda_3(x) = -2\sin(x)-x\cos(x)$, and $\lambda_4(x) = x\sin(x)-3\cos(x)$. In this case, $\mathcal{W} = \text{span}~\{M^1,M^2,M^3\}$, $\mathcal{D} = \text{span}~\{M^4\}$, $\mathcal{Q} = \text{span}~\{\lambda_1, \lambda_2,\lambda_3\}$ and $\mathcal{N} = \text{span}~\{\lambda_4\}$. Figure \ref{fig5} shows the estimated factors of equation (\ref{stsmlast}) by using the variance-based factor model $\hat{Y}$ and the PCA semimartingale $\hat{Z}$ developed in this paper. Clearly, the variance-based factor model is not able to identify the subspace $\mathcal{D}$ (and hence $\mathcal{N}$ as well), while the PCA semimartingale does. In addition, Table 2 presents $\hat{\lambda}_k := \sum_{j=1}^k \hat{m}_{jj}/\sum_{j=1}^{4} \hat{m}_{jj}$ where $\hat{m}_{jj}$ is the $(j,j)$-th element of the matrix $\widehat{[Y]}_T$ (see (\ref{hatquadmatrix})). Table 3 presents the variation explained by PCA semimartingale  by means of $\hat{\eta}_i=\frac{\sum_{j=1}^i\hat{\theta}_j}{\sum_{r=1}^4\hat{\theta}_r}, 1\le i \le 4$ where $\hat{\theta}_i$ is the $i$-th estimated eigenvalue related to $i$-th estimated principal component $\hat{Z}^i$. One clearly see the use of PCA semimartigale is more efficient than the variance-based factor model in identifying quadratic variation dimension.

\begin{figure}[htb]
\begin{center}
\includegraphics[scale=0.7]{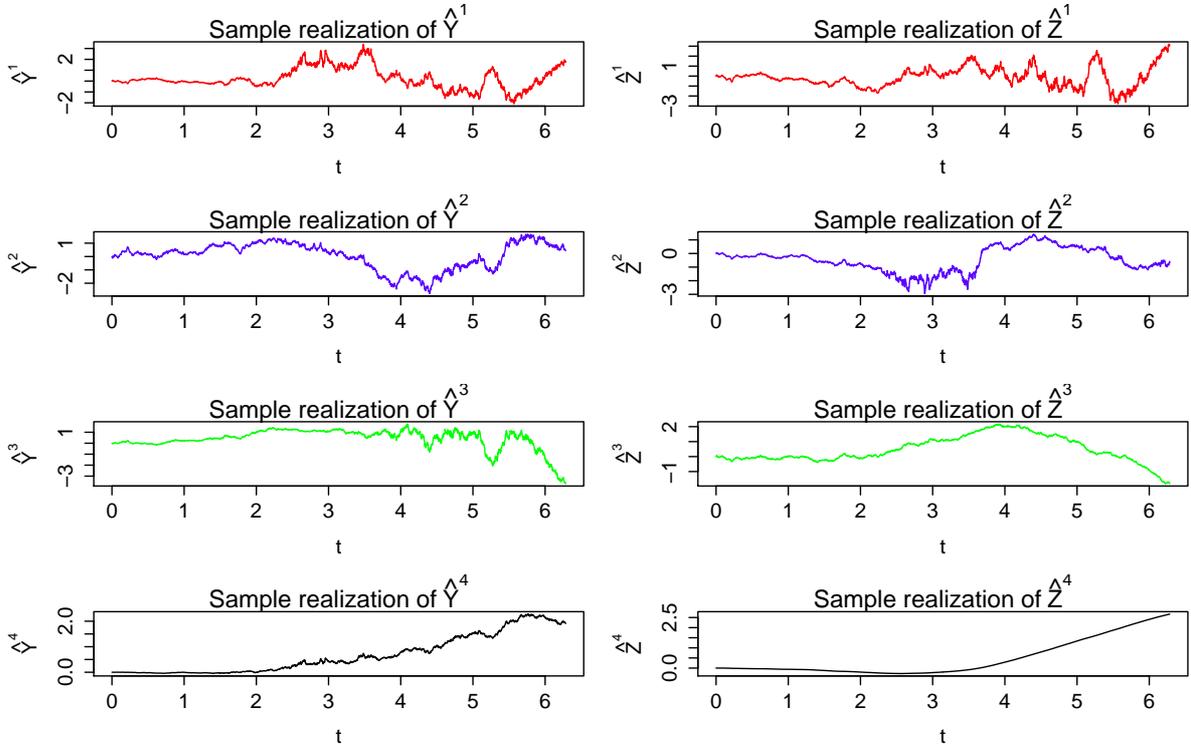}
\end{center}
\caption{Estimated factors for the space-time semimartingale~(\ref{stsmlast})}
\label{fig5}
\end{figure}

\begin{table}[htb]
\caption{Quadratic variation explained by the principal components: Variance-based factor model}
\begin{center}
\begin{tabular}{|c|c|c|c|}
  \hline
  $\hat{\lambda}_1$ &$\hat{\lambda}_2$  & $\hat{\lambda}_3$ & $\hat{\lambda}_4$\\
  \hline
  $0.4795 $ & $0.6912$ & $0.9886$ & $1.0000$\\
  \hline
\end{tabular}
\end{center}
\end{table}

\begin{table}[htb]
\caption{Quadratic variation explained by the principal components: PCA semimartingale}
\begin{center}
\begin{tabular}{|c|c|c|c|}
  \hline
  $\hat{\eta}_1$ &$\hat{\eta}_2$  & $\hat{\eta}_3$ & $\hat{\eta}_4$\\
  \hline
  $0.7805$ & $0.9794$ & $0.9998$ & $1.0000$\\
  \hline
\end{tabular}
\end{center}
\end{table}

\

\noindent \textbf{Dynamic distance between manifolds}
Let us now investigate the robustness of our methodology in the estimation of the minimal invariant manifold, say $V$, for a stochastic PDE. For this purpose, we consider the following objects: Let $\hat{V} = \widehat{Q}\oplus \widehat{N}$ be the estimator for $V$ based on the entire sample $\{(t_i,x_j); 0\le i\le \bar{n}, 0\le j\le \bar{N}\}$. Let $\hat{V}_{-k}$ be the same estimator but computed over the reduced sample $\{(t_i,x_j); 1=0, \ldots, \bar{n}-k, j=0, \ldots, \bar{N}\}$ where $1\le k \le \tilde{K}$ and $\tilde{K}$ is a fixed integer smaller than $\bar{n}$. To compute the distance $D$ between manifolds, we use the following approximation for the Sobolev inner product $\langle \cdot, \cdot\rangle_E$

$$\langle f, g\rangle_{a} := \sum_{i=1}^{\bar{N}}\frac{\big(  f(x_j) - f(x_{j-1}) \big)\big( g(x_j)- g(x_{j-1})\big)}{x_j - x_{j-1}}; f,g\in E.$$
See e.g Prop 1.45 in \cite{friz} for more details. Based on this approximation for $\langle \cdot, \cdot\rangle_E$, we perform the Gram Schmidt algorithm to orthonomalize $\hat{V}$ and $\hat{V}_{-k}$. We then use (\ref{metricsub}) computed in terms of $\langle \cdot, \cdot\rangle_a$. We repeat the above procedure for $k = 1,\ldots, \tilde{K}$ where $\tilde{K}$ is a prescribed integer smaller than $\bar{n}$. The idea is to compute

\begin{equation}\label{dynamicD}
d(\hat{V},\hat{V}_{-k}); k=5, 10, 15, 20, \ldots, 250.
\end{equation}
Under existence of a finite-dimensional invariant manifold, $k\mapsto d(\hat{V},\hat{V}_{-k})$ must be null as $n,N\rightarrow \infty$. In order to illustrate the invariance aspect of Theorem \ref{mainTHPAPER}, we consider the following stochastic PDE

\begin{equation}\label{ss2}
dr_t = \big(A(r_t) + \alpha_{HJM}(r_t)\big)dt + \sum_{i=1}^3\lambda_{i} dB_t^i.
\end{equation}
where the volatilities curves are $\lambda_1 = x\cos (x), \lambda_2 = \cos(x)-x\sin(x), \lambda_3(x) = -2\sin(x)-x\cos(x)$ and $\lambda_4(x) = x\sin(x)-3\cos(x)$, $r_0 = 0$, $A = \frac{d}{dx}$ is the infinitesimal generator of the right-shift semigroup $(S_t)_{t\ge 0}$ defined by the action $S_t\varphi(x):= \varphi(t+x)$. We set $\alpha=\alpha_{HJM}$ as the classical Heath-Jarrow-Morton drift (see Heath et al~\cite{heath}). One can easily check that this HJM model admits a finite-dimensional realization of the form

\begin{eqnarray*}
dZ_t^1 &=& - Z_t^2dt + dB_t^1\nonumber\\
dZ_t^2 &=& (-2Z_t^1 + Z_t^3)dt + dB_t^2 \label{sde1} \\
dZ_t^3 &=& (Z_t^4 - Z_t^1)dt + dB_t^3\nonumber \\
dZ_t^4 &=& - Z_t^1dt \nonumber
\end{eqnarray*}
and a parametrization

\begin{eqnarray*}
\phi_t(x) &=& -\frac{1}{2}\left(x\sin(x)+\cos(x)\right)^2 + \frac{1}{2}\left((x+t)\sin(x+t)+\cos(x+t)\right)^2\\
&-&\frac{1}{2}\lambda_1(x)^2 + \frac{1}{2}\lambda_1(x+t)^2-\frac{1}{2}\lambda_2(x)^2 + \frac{1}{2}\lambda_2(x+t)^2,
\end{eqnarray*}
In this case,

$$r_t = \phi_t + \sum_{i=1}^4 Z_t^i \lambda_i$$
is the strong solution of (\ref{ss2}). We compute (\ref{dynamicD}) for the model (\ref{ss2}) which shows that it fluctuates between zero and $2.5 \times 10^{-8}$ so that we prefer do not report this numerical experiment in this section. More interesting than this is to illustrate (\ref{dynamicD}) in the presence of noise. For this purpose, we consider an observational process $X_t(x) = r_t(x)+ \varepsilon_t(x)$ where $\varepsilon_t(x) = \frac{\sqrt{2}}{3}u_t \text{sin}~(\pi x)$ where $u_t$ is a standard Gaussian variable for every $t\ge 0$ such that $u_t$ is independent form $u_s$ whenever $s\neq t$. Figure \ref{fig8} illustrates that the presence of noise may lead to an erroneous analysis for the existence of a finite-dimensional invariant manifold for the stochastic PDE (\ref{ss2}). As the backward lag increases the distance between manifolds increases as well with short periods of stability.

\begin{figure}[htb]
\begin{center}
\includegraphics[scale=0.5]{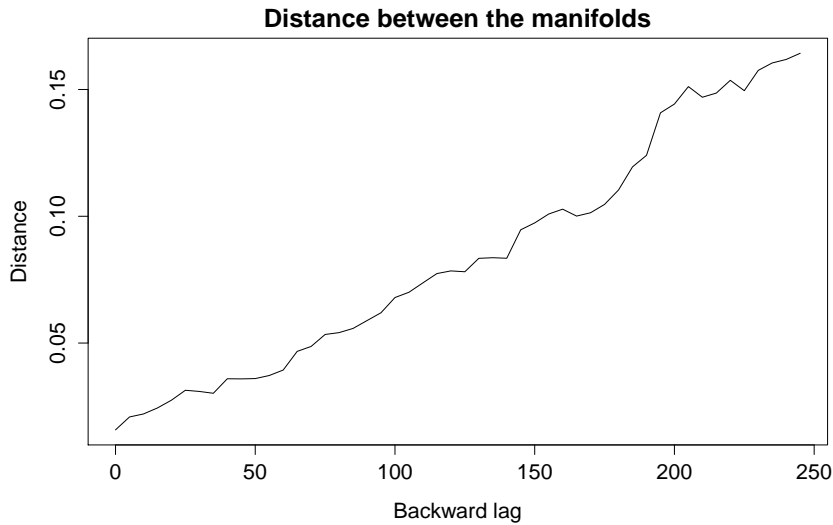}
\end{center}
\caption{Dynamic distance (\ref{dynamicD}): Finite-dimensional realization with noise}
\label{fig8}
\end{figure}

\subsection{Application to real data sets}\label{libor}
In this section, we illustrate the theoretical results of this article with an application to a real data set. We consider the UK nominal spot curve obtained by the Bank of England with maturities .5 to 25 years (50 maturities) and daily data ranging from 27 May 2005 to 9 October 2007 summing 601 observations. We postulate an affine structure in the data, for instance finite-dimensional realizations for a SPDE data generating process. The first task is to estimate the underlying dimension of the affine manifold. The penalty function used in (\ref{pc}) to estimate the number of factors is given in page 201 of \cite{bai}. Any of these penalty functions produce identical results for the estimation of the underlying dimension. The statistics $\hat{d}$ (given by (\ref{dhat}) estimates seven factors for this data. In order to estimate the dimension of the quadratic variation space $\mathcal{Q}$, we make use of the Fourier-type estimator introduced by \cite{mancino}. Under the assumptions of Proposition \ref{fest} in Appendix, we take $\epsilon = n^{-\frac{1}{3}}$ in Corollary \ref{pepsilon} in the estimation procedure. The estimation indicates $dim~\widehat{\mathcal{Q}}=6$, so that $dim~\widehat{\mathcal{D}}=1$. Figures \ref{fig6} and \ref{fig7} report the time series of the estimated factors $(\widehat{Y},\widehat{Z})$, where $\widehat{Y}$ denotes the variance-based factor estimator and $\widehat{Z}$ is given by (\ref{PCAest}).

Figure \ref{fig7} and the estimation $dim~\widehat{\mathcal{D}}=1$ strongly indicate the presence of a non-trivial drift dynamics in the data. In particular, the estimated factor with smallest variance $\widehat{Y}^7$ is not able to identify the drift while $\widehat{Z}^7$ seems to estimate a bounded variation curve subject to small errors due to observational errors or microstructure effects.

In order to compare our methodology with the standard factor model, we also perform a principal component analysis in two different versions. In Table 4, $\hat{\eta}_i=\sum_{j=1}^i\hat{\theta}_j/\sum_{r=1}^{7}\hat{\theta}_r, 1\le i \le 7$ and $\hat{\theta}_i$ is the $i$-th estimated eigenvalue related to the PCA semimartingale estimated principal component $\hat{Z}^i$ given by (\ref{PCAest}). In table 5,
$\hat{\lambda}_k := \sum_{j=1}^k \hat{m}_{jj}/\sum_{j=1}^{7} \hat{m}_{jj}$ where $\hat{m}_{jj}$ is the $(j,j)$-th element of the matrix $\widehat{[Y]}_T$ (see (\ref{hatquadmatrix})). The first PCA semimartingale component already explains 50 per cent of the total variation while only the third classical factor  approximates half of the quadratic variation contained in the data.

Figure \ref{fig9} reports the dynamic distance (\ref{dynamicD}) of the estimated manifold $\hat{V}$ over the entire period of our sample against $\hat{V}_{-k}$ where $k=5,10, 15, \ldots 200.$ As the backward lag increases, the distance increases as well but we observe some periods of stability over time.

\begin{figure}[htb]
\begin{center}
\includegraphics[scale=0.5]{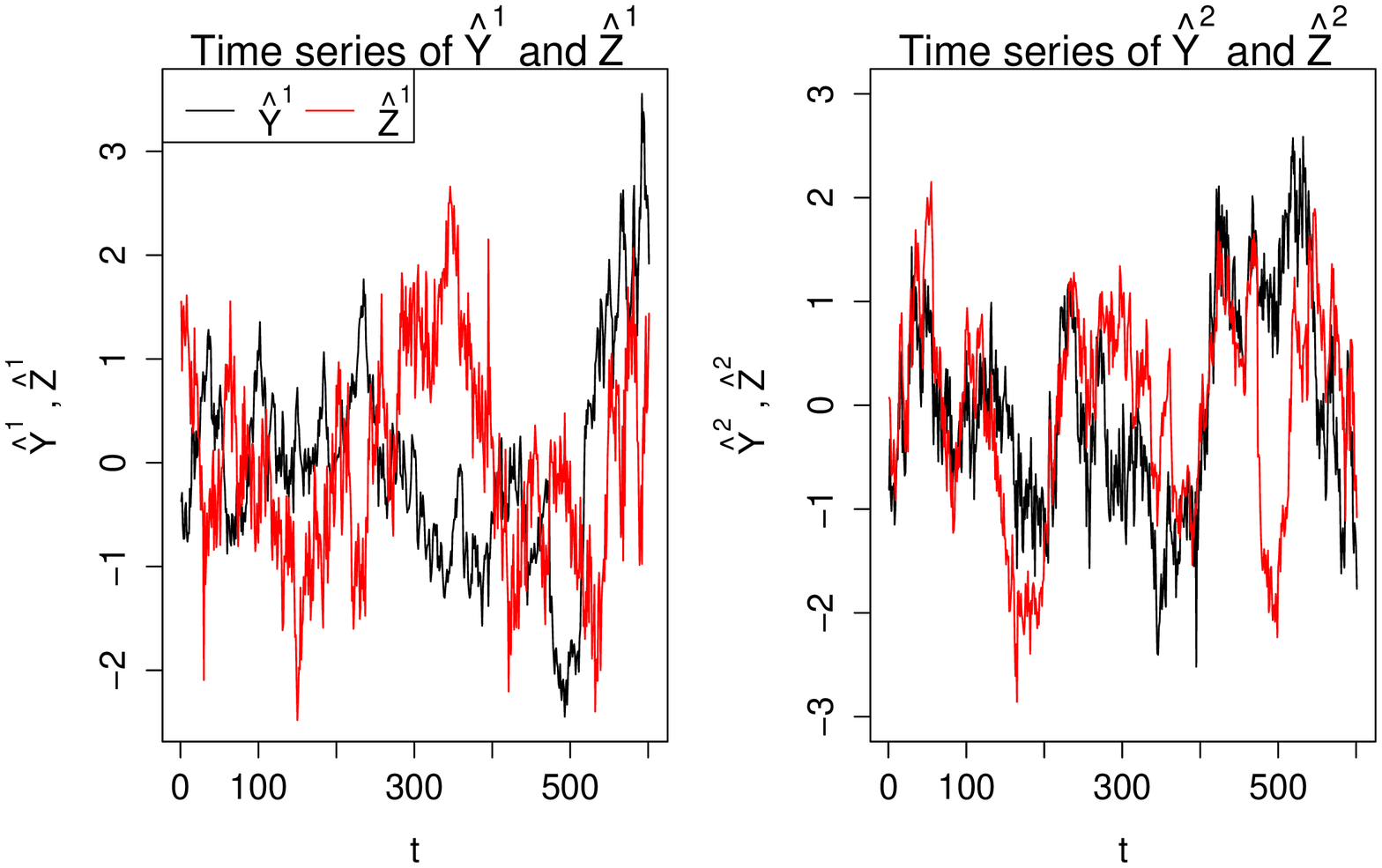}
\includegraphics[scale=0.5]{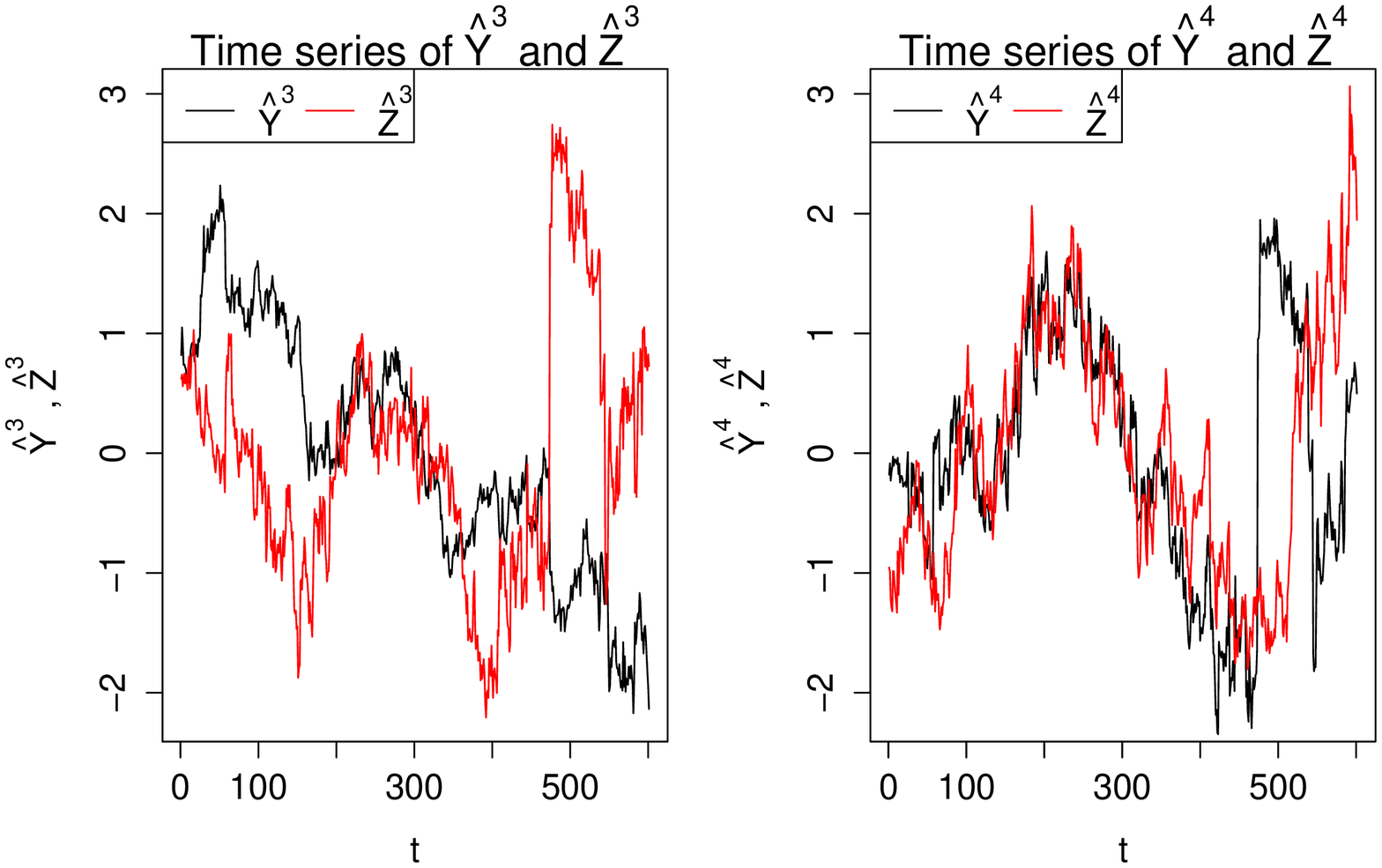}
\end{center}
\caption{Time series of the estimated factors}
\label{fig6}
\end{figure}

\begin{figure}[htb]
\begin{center}
\includegraphics[scale=0.5]{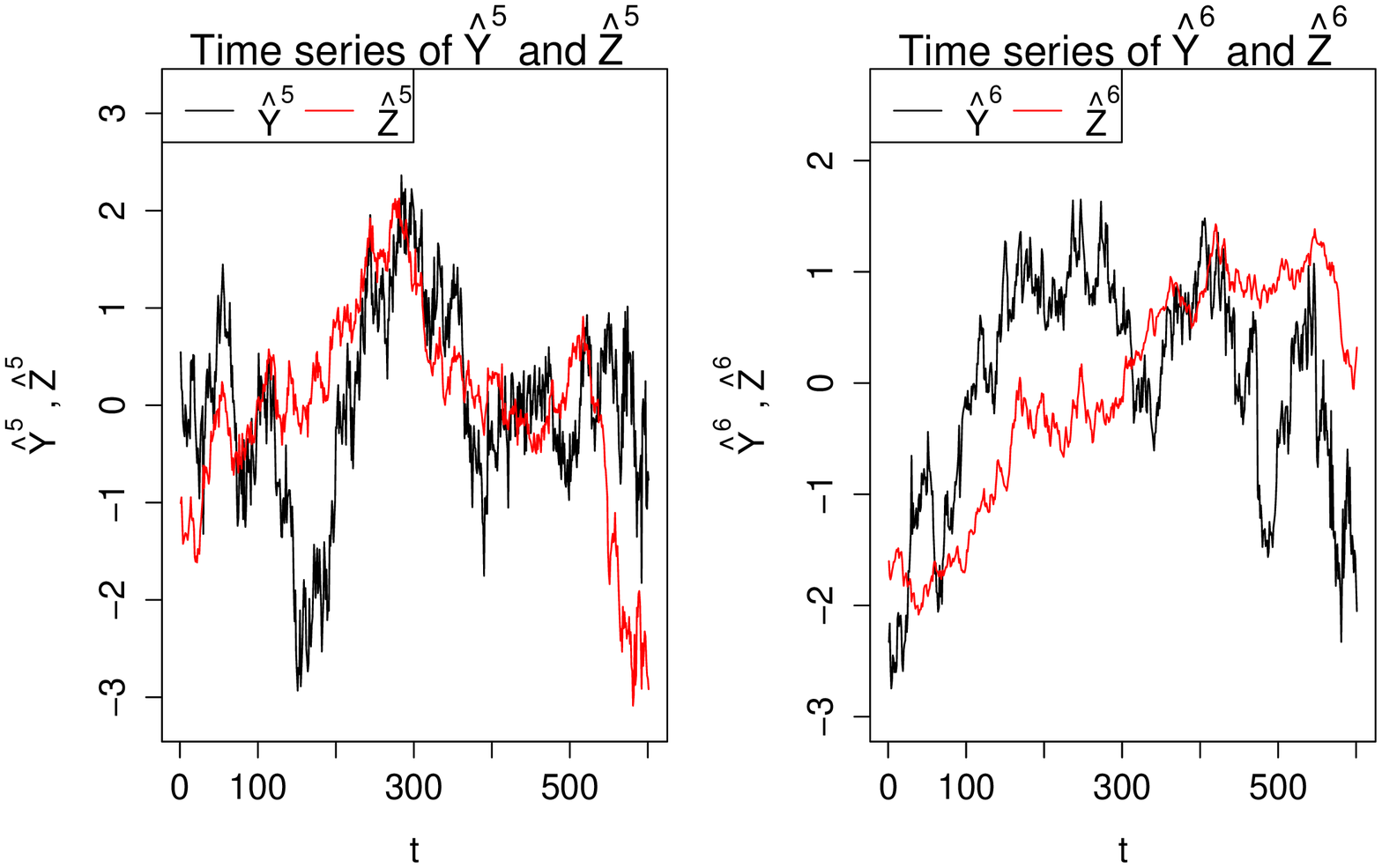}
\includegraphics[scale=0.5]{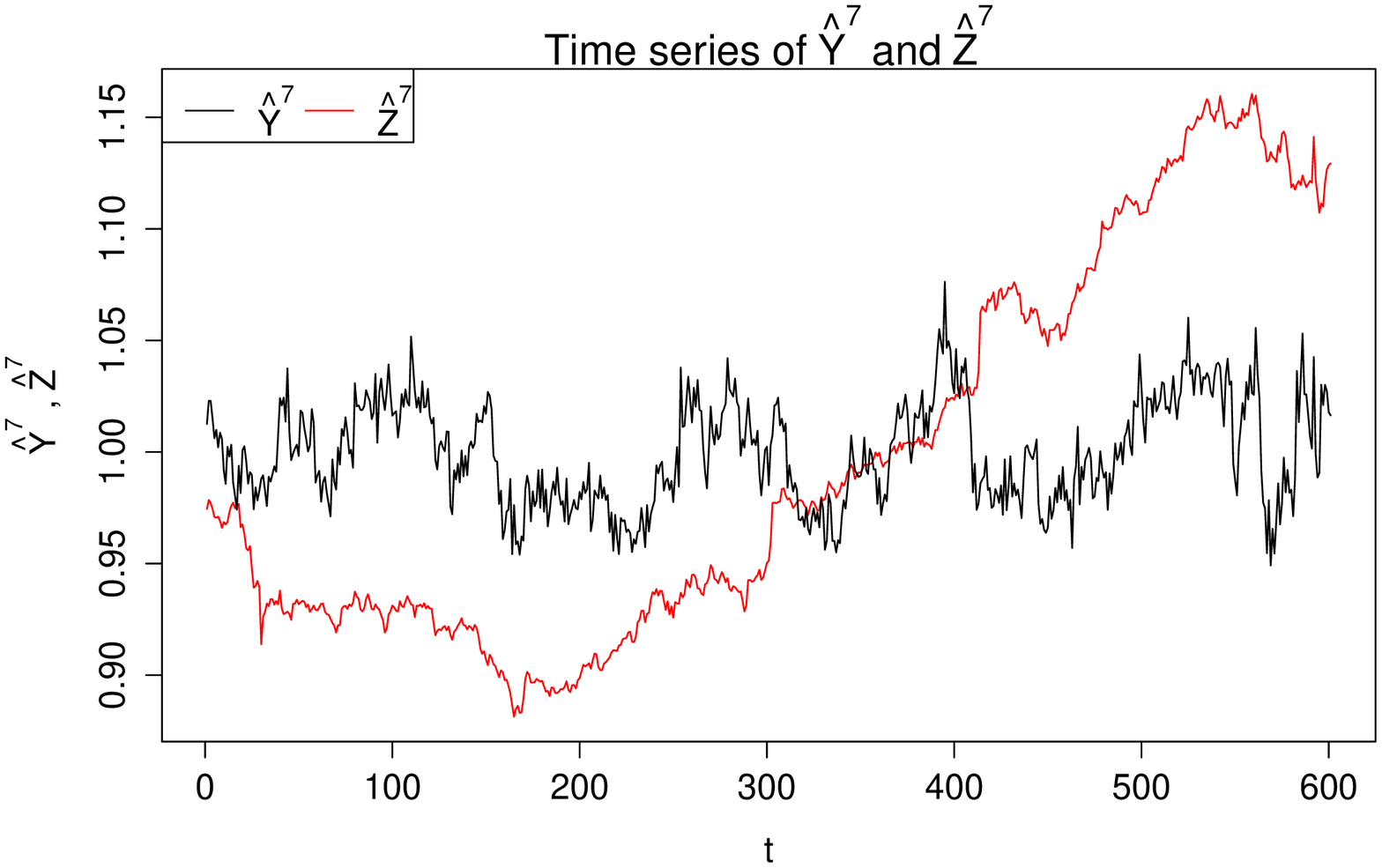}
\end{center}
\caption{Time series of the estimated factors}
\label{fig7}
\end{figure}

\begin{table}[htb]
\caption{Quadratic variation explained by the principal components: PCA semimartingale}
\begin{center}
\begin{tabular}{|c|c|c|c|c|c|c|}
  \hline
  $\hat{\eta}_1$ &$\hat{\eta}_2$  & $\hat{\eta}_3$ & $\hat{\eta_4}$& $\hat{\eta_5}$& $\hat{\eta_6}$& $\hat{\eta_7}$ \\
  \hline
  $0.5010 $ & $0.7152$ & $0.8548$ & $0.9471$& $0.9925$&$0.9999$&$1.0000$\\
  \hline
\end{tabular}
\end{center}
\end{table}

\begin{table}[htb]
\caption{Quadratic variation explained by the principal components: Variance-based factor model}
\begin{center}
\begin{tabular}{|c|c|c|c|c|c|c|}
  \hline
  $\hat{\lambda}_1$ &$\hat{\lambda}_2$  & $\hat{\lambda}_3$ & $\hat{\lambda_4}$& $\hat{\lambda_5}$& $\hat{\lambda_6}$& $\hat{\lambda_7}$ \\
  \hline
  $0.1161$ & $0.4553$ & $0.4911$ & $0.6198$& $0.8927$&$0.9996$&$1.0000$\\
  \hline
\end{tabular}
\end{center}
\end{table}

\begin{figure}[htb]
\begin{center}
\includegraphics[scale=0.5]{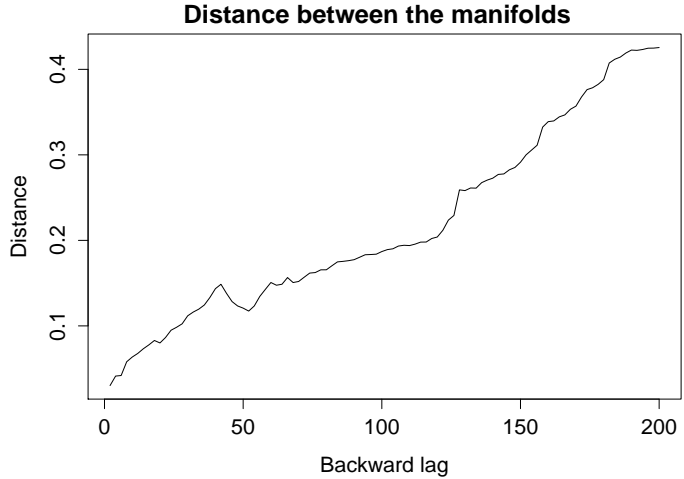}
\end{center}
\caption{Dynamic distance (\ref{dynamicD}) of the UK nominal daily spot curve}
\label{fig9}
\end{figure}

\section{Appendix: Estimating dim~$\mathcal{Q}$}\label{Appendix}
In this section, we give a concrete alternative for estimating $p=\text{dim}~\mathcal{Q}$ which is an important component in Theorem \ref{mainTHPAPER}. From (\ref{ss1}), we notice that if the stochastic PDE admits a finite-dimensional realization, then the matrix of the finite-rank linear operator $Q_T$ is given by $[Z]_T$ whenever

$$r_t = \phi_t + \sum_{j=1}^d Z^j_t \lambda_j; 0\le t\le T,$$
for each basis $\{\lambda_1, \ldots, \lambda_d\}$ of a finite-dimensional subspace which generates a finite-dimensional realization. Rigorously speaking, we cannot estimate directly $\text{dim}~\mathcal{Q}$ through high-frequency sampling of factors because they are not observed. In this case, one has to work with a high-frequency sampling of observed curves subject to noise. This section provides a feasible estimation procedure for this.

We choose to work with the Fourier-type estimator proposed by Malliavin and Mancino \cite{mancino} but we stress that other quadratic variation estimators can be certainly used as well. The strategy is to find the minimum requirements on the residual process in such way that one can estimate the random operator $Q_{T}$ via an observed curve process

$$X_t(x) = r_t(x)+\varepsilon_t(x);0\le t\le T, a\le x\le b.$$
If $\varepsilon$ is negligible in the quadratic variation sense, then the method of the estimation of the kernel $Q_{T}(u,v)$ is fully based on any reasonable non-parametric estimator of the integrated volatility. We assume that $X$ is a well-defined semimartingale random field.

In the sequel, without any loss of generality we assume that $[0,T]=[0,2\pi]$. Let $\Pi= \{t^n_{i}; i=0,\ldots,\bar{n}\}$ be the instant times of observations and $\rho(n)=\max_{0\le h\le \bar{n}-1}|t^n_{h+1}-t^n_h|$. In this section, we assume that $\rho(n)\rightarrow 0$ as $n\rightarrow \infty$ so we are able to sample the curves $x\mapsto X_t(x)$ in high-frequency in time. In the sequel, we make use of the following notation

$$\varphi_n(t):=\sup\{ t^n_k; t^n_k\le t\}$$

For given positive integers $M\ge 1$ and $n\ge 1$, we define

 \begin{eqnarray*}
\hat{Q}_{T}(u,v)&:=& \frac{1}{2M+1}\sum_{|s|\le M}\int_0^Te^{\text{i}s\varphi_n(t)}dX_t(u)\int_0^Te^{-\text{i}s\varphi_n(t)}dX_t(v)\\
& & \\
&=& \sum_{\ell,j=0}^{n-1}d_M(t^n_\ell-t^n_j)\Delta X_{t^n_{\ell+1}}(u)\Delta X_{t^n_{j+1}}(v),~u,v\in [a,b],
\end{eqnarray*}
where $\Delta X_{t^n_{\ell+1}}(u):=X_{t^n_{\ell+1}}(u)-X_{t^n_{\ell}}(u)$ and

\[ d_M(t) :=\frac{1}{2M+1}\sum_{|s|\le M}e^{\text{i}ts}=\left\{
\begin{array}{crcr}
1;~t=sT,~s\in\mathbb{Z} \\
\frac{1}{2M+1}\frac{sin[(M+1/2)t]}{sin(t/2)};~t\neq sT
\end{array}
\right.
\]

\noindent is the normalized Dirichlet kernel. Here $M$ encodes the Bohr convolution product and $n$ the discretization level of the Fourier transform of $\sigma_t(u,v)$. See Malliavin and Mancino~\cite{mancino} for more details.

To keep notation simple, from now on we set $t_i:= t^n_i; 0\le i \le \bar{n}$. The kernel $\hat{Q}_{T}(u,v)$ induces a random linear operator $\hat{Q}_{T}$ on the complexification $E_\mathbb{C}$ as follows

\begin{eqnarray*}
(\hat{Q}_{T}f)(u)&=& \langle  \hat{Q}_{T}(u,\cdot), f\rangle_{E_\mathbb{C}}\\
& &\\
&=& \frac{1}{2M+1}\sum_{|s|\le M}\sum_{\ell=0}^{n-1}  \sum_{k=0}^{n-1}\exp(\text{i}s(t_k-t_\ell))\Delta X_{t_{k+1}}(u)\langle \Delta X_{t_{l+1}}, f \rangle_{E_\mathbb{C}},
\end{eqnarray*}
for $f\in E_{\mathbb{C}}$. The reason to consider $E$ on the field $\mathbb{C}$ is due to a nice representation as follows. If $A$ and $B$ are two linear operators on Hilbert spaces then $AB$ and $BA$ share the same nonzero eigenvalues. Furthermore, if $\gamma$ is an eigenvector of $BA$, $A\gamma$ is an eigenvector of $AB$ with the same eigenvalue. So the strategy is to write $\hat{Q}_{T} = AB$ in such way that $BA:\mathbb{C}^p\rightarrow \mathbb{C}^p$ for some $p\ge 1$ and therefore one can easily relate the eigenvalues of $BA$ to $AB$. In fact, by the very definition of $\hat{Q}_{T}$ we have

$$
\hat{Q}_{T} = AB
$$
where $B:E_{\mathbb{C}}\rightarrow \mathbb{C}^{2M+1}$ is defined by

$$B(\cdot):= \Bigg(\sum_{\ell=0}^{n-1}\exp(-\text{i}(-M)t_\ell)\langle \Delta X_{t_{\ell+1}}, \cdot~\rangle_{E_\mathbb{C}},\ldots,  \sum_{\ell=0}^{n-1}\exp(-\text{i}(M)t_\ell)\langle \Delta X_{t_{\ell+1}},\cdot~\rangle_{E_{\mathbb{C}}}  \Bigg)$$

\noindent and $A:\mathbb{C}^{2M+1}\rightarrow E_{\mathbb{C}}$ is defined by

\begin{equation}\label{Aoper}
(Ax)(\cdot):= \frac{1}{2M+1}\sum_{|s|\le M}x_s \sum_{k=0}^{n-1}\exp(\text{i}st_k)\Delta X_{t_{k+1}}(\cdot);x\in\mathbb{C}^{2M+1}.
\end{equation}

By the very definition, $\bar{Q}_{T}:=BA:\mathbb{C}^{2M+1}\rightarrow \mathbb{C}^{2M+1}$ is given componentwise by

$$\bar{Q}_{T} y = \frac{1}{2M+1}\sum_{|s|\le M}\sum_{k,\ell=1}^{n-1}y_s\exp\big(\text{i}(st_k-mt_\ell)\big)\langle \Delta X_{t_{\ell+1}},\Delta X_{t_{k+1}} \rangle_{E_{\mathbb{C}}},
$$

\noindent for $y\in\mathbb{C}^{2M+1},~m=-M,\ldots, M.$ We then arrive at the following elementary result.

\begin{lemma}
The random linear operators $\hat{Q}_{T}$ and $\bar{Q}_{T}$ share the same nonzero eigenvalues in $\mathbb{C}$. Let $\hat{p}$ be the number of nonzero eigenvalues $\{\hat{\theta}_i=1,\ldots,\hat{p}\}$ of $\bar{Q}_T$ and let $\gamma_j=(\gamma_j(-M),\ldots,\gamma_j(M)),~j=1,\ldots,\hat{p}$ be the corresponding eigenvectors in $\mathbb{C}^{2M+1}$. Then

$$\frac{1}{2M+1}\sum_{|s|\le M}\gamma_j(s)\Bigg(\sum_{k=0}^{n-1}\exp(\text{i} st_k)\Delta X_{t_{k+1}} \Bigg),\quad j=1,\ldots,\hat{p}$$

\noindent are the $\hat{p}$ eigenfunctions of $\hat{Q}_{T}$.
\end{lemma}
\begin{proof}
Let $\hat{\theta}_j\in\mathbb{C}$ be a nonzero eigenvalue of $\bar{Q}_{T}$ and let $\gamma_j\in\mathbb{C}^{2M+1}$ be the corresponding eigenvector. Then

\begin{equation}\label{sr}
\bar{Q}_{T}\gamma_j = \hat{\theta}_j\gamma_j~\text{and}~\hat{Q}_{T}(A\gamma_j)=\hat{\theta}_jA\gamma_j~a.s,
\end{equation}

\noindent where $A$ is the operator given by~(\ref{Aoper}). By writing~(\ref{sr}) component by component we have

$$\frac{1}{2M+1}\sum_{|s|\le M}\sum_{k,\ell=0}^{n-1}\gamma_j(s)\exp\big(\text{i}(st_k-rt_\ell)\big)\langle \Delta X_{t_{\ell+1}}, \Delta X_{t_{k+1}}  \rangle_{E_{\mathbb{C}}} = \hat{\theta}_j\gamma_j(r)~a.s$$

\noindent for $r=-M,\ldots, M$. On the other hand,

$$A\gamma_j = \frac{1}{2M+1}\sum_{|s|\le M}\gamma_j(s)\Bigg(\sum_{k=0}^{n-1}\exp(\text{i}st_k)\Delta X_{t_{k+1}} \Bigg),\quad j=1,\ldots,\hat{p}~a.s.$$

\noindent This concludes the proof of the Lemma.
\end{proof}

\begin{remark}
Let

\begin{equation}\label{belements}
\frac{1}{2M+1}\sum_{|s|\le M}\gamma_j(s)\Bigg(\sum_{k=0}^{n-1}\exp(\text{i}st_k)\Delta X_{t_{k+1}} \Bigg),\quad j=1,\ldots,\hat{p}~a.s
\end{equation}

\noindent be the eigenvectors of $\hat{Q}_{T}$ related to its nonzero eigenvalues $\{\hat{\theta}_j; j=1,\ldots, \hat{p} \}$. Since $\hat{Q}_{T}$ is a self-adjoint finite-rank operator then the following spectral decomposition holds a.s

$$\hat{Q}_{T} f = \sum_{i=1}^{\hat{p}}\hat{\theta}_i\langle f,\hat{\varphi}_i\rangle_{E_{\mathbb{C}}} \hat{\varphi}_{i};~ f\in E_{\mathbb{C}},$$
where $\hat{p}\le 2M+1$ a.s for every $n,M$ and $\{\hat{\varphi}_i; i=1,\ldots, \hat{p}\}$ is an orthonormal set by applying a Gram-Schmidt algorithm to the functions given by~(\ref{belements}).
\end{remark}

Let us now introduce the basic assumptions on the residual process $\varepsilon$. Since the estimation is based on a high-frequency sampling we need to impose some structure on the continuous-time dynamics.

\

\noindent \textbf{(B1)} The residual process $\varepsilon$ is an It\^o semimartingale field where the drift component $h$ satisfies

$$\sup_{0\le t\le T}\|h_t\|_E \in L^p$$
for every $p>1$.

\

\noindent \textbf{(B2)} The quadratic variation of $\varepsilon(u)$ at time $T$ satisfies

$$\int_{\mathbb{R}}[\varepsilon(u),\varepsilon(u)]_{T}\mu(du)=0~a.s.$$

\

\noindent\textbf{(B3)} The following growth assumption holds

$$0 < \liminf_{n,M\rightarrow\infty} M\rho(n) \le \limsup_{n,M\rightarrow\infty} M\rho(n) < \infty.$$

\

\noindent \textbf{(H1)} The vector fields $F,\sigma^i:E\rightarrow E$ are globally Lipschitz for each $i=1,\ldots, m$.

\

\noindent \textbf{(H2)} Linear growth condition on the vector fields $F, \sigma^1; 1\le i \le m$ in (\ref{spde1}): There exists a constant $C>0$ such that

$$\|F(x)\|^2_E + \sum_{i=1}^m\|\sigma^i(x)\|^2_E \le C^2(1+ \|x\|^2_E)~\text{for every}~x\in E.$$

\
\begin{remark}\label{discussionFRM}
As far as the consistency problem of the HJM model (see Section 4.2), Assumption (\textbf{B2}) means that the initial fitting method used to interpolate points which generates $X$ cannot introduce an extrinsic volatility. The interpolation must
be chosen in such way that the resulting observed volatility on the whole curve
must be fully dictated by the market and not to the particular choice of fitting. See also Assumption (\textbf{Q4}). The semimartingale
decomposition yields the following structure on the residual process

$$\varepsilon_t= e + \int_0^th_sds$$
for some $\mathcal{F}_0$-measurable random variable $e = X_0 - r_0$ and an integrable adapted process $h$ satisfying (\textbf{B1}). Assumption (\textbf{B3}) is a technical assumption in order to get optimal bounds but it is also linked with different flavors
between the exact Fourier estimator and the usual quadratic variation estimator for $Q_T$ . See Malliavin and Mancino \cite{mancino} and Clement and Gloter \cite{clement} for further details.
\end{remark}

\noindent Under \textbf{(H1)} and \textbf{(H2)}, it is well-known for every initial condition $\xi\in E$, there exists a unique mild solution $r_t^\xi$ of the stochastic PDE. Moreover, the following integrability property holds

$$\mathbb{E}\sup_{0\le t\le T}\|r^\xi_t\|^q_E< \infty$$
for every $q> 1$ and $\xi\in E$.

The following result is a functional version of (and almost straightforward consequence) of Proposition 1 and Lemma 3 in \cite{clement}. In the sequel, $\|\cdot\|_{(2)}$ is the Hilbert-Schmidt norm operator over $E_\mathbb{C}$ and to keep notation simple, we write $\|\cdot\| = \|\cdot\|_{E_\mathbb{C}}$.

\begin{proposition}\label{fest}
Assume that~\textbf{(A1, A2, B1, B2, B3, H1, H2)} hold and in addition

\begin{equation}\label{a1}
\mathbb{E}^{1/2}\sup_{0\le t\le T}|\partial_v\sigma^j(r_t)(\cdot)|^4 \in L^1(\mu)
\end{equation}
for each $j=1,\ldots, m$. Then

$$
\mathbb{E}\|Q_{T}-\hat{Q}_{T}\|^2_{(2)}=O(\rho(n)).
$$
\end{proposition}
\begin{proof}
In the sequel, we denote by $C$ a positive constant which may differ from line to line. We also decompose

$$X_t(u) = \tilde{r}_t(u) + \tilde{\varepsilon}_t(u),$$
into
$$\tilde{r}_t(u):=\sum_{j=1}^m\int_0^t\sigma^j(r_s)(u) dB^j_s;~0\le t\le T,~u\in [a,b],$$
$$\tilde{\varepsilon}_t(u) = \int_0^t\xi_s(u)ds$$
where
$$\xi_t(u): = Ar_t(u) + F(r_t)(u) + h_t(u); 0\le t\le T, u\in [a,b].$$
For a given $(u,v)\in[a,b]\times [a,b]$, integration by parts and~\textbf{(B2)} yield

\begin{eqnarray*}
Q_{T}(u,v)-\hat{Q}_{T}(u,v) &=&\int_0^{T}\Big(\int_0^td_{M,n}(\ell,t)dX_\ell(u)\Big)dX_t(v)\\
& &\\
&+& \int_0^{T}\Big(\int_0^td_{M,n}(\ell,t)dX_\ell(v)\Big)dX_t(u)\\
& &\\
&=:&R_{n,M}(u,v)
\end{eqnarray*}
where

$$d_{M,n}(\ell, t):=d_M(\varphi_n(\ell) - \varphi_n(t)),$$

\begin{eqnarray*}
R_{n,M}(u,v) = J_{n,M,1}(u,v) + J_{n,M,2}(u,v) + \sum_{i=1}^3I_{n,M,i}(u,v) + \sum_{i=1}^3\hat{I}_{n,M,i}(u,v),
\end{eqnarray*}
and
\begin{small}
$$J_{n,M,1}(u,v):= \int_0^{T}\Big(\int_0^td_{M,n}(\ell,t)d\tilde{r}_\ell(u)\Big)d\tilde{r}_t(v), $$

$$J_{n,M,2}(u,v):= \int_0^{T}\Big(\int_0^td_{M,n}(\ell,t)d\tilde{r}_\ell(v)\Big)d\tilde{r}_t(u),$$

$$I_{n,M,1}(u,v) := \int_0^{T}\Big(\int_0^td_{M,n}(\ell,t)d\tilde{\varepsilon}_\ell(u)\Big)d\tilde{r}_t(v),$$

$$I_{n,M,2}(u,v) := \int_0^{T}\Big(\int_0^td_{M,n}(\ell,t)d\tilde{r}_\ell(u)\Big)d\tilde{\varepsilon}_t(v),$$

$$I_{n,M,3}(u,v) := \int_0^{T}\Big(\int_0^td_{M,n}(\ell,t)d\tilde{\varepsilon}_\ell(u)\Big)d\tilde{\varepsilon}_t(v),$$
\end{small}
where $\hat{I}_{n,M,i}$ are the symmetric quantities w.r.t $I_{n,M,i}$. By the very definition

\begin{eqnarray*}
\nonumber \|Q_{T} - \hat{Q}_{T}\|^2_{(2)}&=&\|(Q_{T}-\hat{Q}_{T})(0,\cdot)\|^2 + \int_{[a,b]}\Big\|\partial_u (Q_{T}-\hat{Q}_{T})(u,\cdot)\Big\|^2\mu(du)\\
& &\\
&=& |Q_{T}(0,0)-\hat{Q}_{T}(0,0)|^2 + \int_{[a,b]}|\partial_vQ_{T}(0,v)-\partial_v\hat{Q}_{T}(0,v) |^2\mu(dv)\\
& &\\
&+& \int_{[a,b]}|\partial_uQ_{T}(u,0)-\partial_u \hat{Q}_{T}(u,0)|^2\mu(du)\\
& &\\
&+& \int_{[a,b]^2}|\partial^2_{vu}Q_{T}(u,v)-\partial^2_{vu}\hat{Q}_{T}(u,v)|^2\mu(du)\mu(dv)\\
& &\\
&=:&T_1(n,M) + T_2(n,M) + T_3(n,M) + T_4(n,M).
\end{eqnarray*}

\noindent \textit{Step 1:}~ The term $T_1$. By the invariance hypothesis~\textbf{(A1)}, we know that~[see~\cite{tappe1}; Corollary 2.13] $V\subset dom~(A)$ so we may consider $(dom~(A), A)$ as a bounded operator restricted to $V$. Moreover, we shall represent

$$
r_t = \textbf{p}_{{V}^{\perp}}r_t+ \textbf{p}_{V}r_t = \textbf{p}_{V^\perp}\mathcal{V}^h_t + \textbf{p}_{V}r_t,
$$
where $\textbf{p}$ is the usual projection and $h=r_0$. From Theorem 2.11 in~\cite{tappe1}, we also know that $t\mapsto A(\pi_{V^\perp}\mathcal{V}^h_t)$ is continuous and therefore there exists a constant $C$ such that
$\|Ar_t\| \le C + C \|r_t\|$ for every $t\in [0,T]$. Based on these facts, we may use the linear growth conditions \textbf{(H1-H2)} and~(\textbf{B1}) to arrive at the following estimate

\begin{equation}\label{driftest}
\sup_{0\le t\le T}\|\xi_t\|\le C +  C\sup_{0\le t\le T}\|r_t\| + C\sup_{0\le t\le T}\|h_t\|.
\end{equation}
In this case, one can easily check that assumptions in Proposition 1 of~\cite{clement} hold trivially and all their estimates as well. In this case, we have
$$T_1(n,M)\le C \int_0^T\int_0^t d^2_{M,n}(\ell,t)d\ell dt.$$
Lemma 3 in~\cite{clement} yields $T_1(n,M)=O(\rho(n))$.

\

\noindent \textit{Step 2:}~The term $T_2 + T_3$. Let us now treat $T_{2}(n,M) + T_3(n,M)$. For a given $(i,j)\in \{1,\ldots, m \}^2$, Burkholder-Davis-Gundy inequality and~\textbf{(H1-H2)} yield

\begin{small}
\begin{eqnarray*}
\mathbb{E}\int_{[a,b]}\Big|\int_0^{T}\int_0^td_{M,n}(\ell,t)\sigma^j(r_\ell)(0)dB^j_\ell
\partial_v\sigma^i(r_t)(v)dB^i_t\Big|^2\mu(dv)&\le& C\int_0^{T}\int_0^t d^2_{M,n}(\ell,t)d\ell dt
\end{eqnarray*}
\end{small}
This yields $\int_{[a,b]}\mathbb{E}|\partial_v J_{n,M,1}(0,v)|^2\mu(dv) = O(\rho(n))$. The same argument also holds for $\partial_vJ_{n,M,2}(0,v)$ and we conclude that

$$\int_{[a,b]}\mathbb{E}|\partial_vJ_{n,M,1}(0,v) +\partial_vJ_{n,M,2}(0,v) |^2\mu(dv) = O(\rho(n))$$
The drift part is estimated as follows. Cauchy-Schwartz and Burkholder-Davis-Gundy inequalities, the estimate~(\ref{driftest}),~\textbf{(H1-H2), A1, B1} and Lemma 3 in~\cite{clement} yield

\begin{small}
\begin{eqnarray*}
\int_{[a,b]}\mathbb{E}|\partial_vI_{n,M,1}(0,v)|^2\mu(dv)&\le&\mathbb{E}\sum_{i=1}^d\int_0^T\Bigg(\int_0^td^2_{M,n}(\ell,t)d\ell
\times \int_0^t\|\xi_\ell\|^2d\ell\\
& &\\
&\times& \|\sigma^i(r_t)\|^2\Bigg)dt\\
& &\\
&\le &C \int_0^T\int_0^t d^2_{M,n}(\ell,t)d\ell dt= O(\rho(n)).
\end{eqnarray*}
\end{small}
The term $\partial_vI_{n,M,2}(0,v)$ is more evolved but we can repeat the same steps as in the proof of Theorem 1 in~\cite{clement} in page 1114 to represent

$$|I_{n,M,2}(u,v)|^2 = \int_{[0,T]^2}Y_{n,M}(u,t,t) \xi_t(v)Y_{n,M}(u,t^{'},t^{'})\xi_{t^{'}}(v)dtdt^{'},$$

where

$$Y_{n,M}(u,t,s):=\int_0^sd_{M,n}(\ell,t)d\tilde{r}_\ell(u).$$

We fix $\eta > 0$ and we split $|I_{n,M,2}(0,v)|^2 =A_{n,M,1}(u,v,\eta) + A_{n,M,2}(u,v,\eta)$ so that

$$\partial_vA_{n,M,1}(0,v,\eta):=\int_0^{T}\int_{t-\eta}^{t}Y_{n,M}(0,t,t) \partial_v\xi_t(v)Y_{n,M}(0,t^{'},t^{'})\partial_v\xi_{t^{'}}(v)dt^{'}dt$$

$$\partial_vA_{n,M,2}(0,v,\eta):= \int_0^{T}\int_{0}^{t-\eta}Y_{n,M}(0,t,t) \partial_v\xi_t(v)Y_{n,M}(0,t^{'},t^{'})\partial_v\xi_{t^{'}}(v)dt^{'}dt$$

By applying the same arguments in the proof of Theorem 1 in~\cite{clement} with small $\eta>0$ together with Cauchy-Schwartz inequalities on $E$ and assumptions \textbf{(H1-H2), B1, B3}, we also have

$$\int_{[a,b]}\mathbb{E}|I_{n,M,2}(0,v)|^2\mu(dv) = O(\rho(n)).$$

Moreover, \textbf{(B1,B3)} and Lemma 3 in~\cite{clement} yield

$$\int_{[a,b]}\mathbb{E}|I_{n,M,3}(0,v)|^2\mu(dv)\le C\int_0^{T}\int_0^t d^2_{M,n}(\ell,t)d\ell dt \le C\rho(n).$$ By the symmetry of the other terms, these estimates allow us to conclude that $T_2(n,M) + T_3(n,M)= O(\rho(n))$.

\

\noindent \textit{Step 3:} The term $T_4$. For given $(i,j)\in \{1,\ldots, m \}^2$, Burkholder-Davis-Gundy and Cauchy-Schwartz inequalities and \textbf{(H1-H2)} with~(\ref{a1}) yield

\begin{small}
\begin{eqnarray*}
\mathbb{E}\int_{[a,b]^2}|\partial^2_{vu}J_{n,M,1}(u,v)|^2\mu(dv)\mu(du)&=&
\int_{[a,b]}\mathbb{E}\int_0^{T}\Big(\int_0^td_{M,n}(\ell,t)
\partial_u\sigma^j(r_\ell)(u)dB^j_\ell\Big)^2\|\sigma^i(r_t)\|^2dt\mu(du)\\
& &\\
&\le &C \int_0^T\int_0^t d^2_{M,n}(\ell,t)d\ell dt\int_{[a,b]} \mathbb{E}^{1/2}\sup_{0\le t\le T}|\partial_v\sigma^j(r_t)(\cdot)|^4\mu(du).
\end{eqnarray*}
\end{small}

Therefore, by the symmetry of the martingale terms and Lemma 3 in~\cite{clement} we have

$$\int_{\mathbb{R}^2_+}\mathbb{E}|\partial^2_{vu}J_{n,M,1}(u,v) +\partial^2_{vu}J_{n,M,2}(u,v) |^2\mu(dv)\mu(du) = O(\rho(n)).$$

Similarly, Lemma 3 in~\cite{clement} and~(\ref{driftest}) yield

\begin{eqnarray*}
\int_{\mathbb{R}_+}\mathbb{E}|\partial^2_{vu}I_{n,M,1}(u,v)|^2\mu(du)\mu(dv)&\le&
\int_0^T\int_0^t d^2_{M,n}(\ell,t)d\ell dt\\
& &\\
&\times&\mathbb{E}\sup_{0\le t\le T}\|\xi_t\|^2\\
& &\\
&\le& C\int_0^T\int_0^t d^2_{M,n}(\ell,t)d\ell dt= O(\rho(n)).
\end{eqnarray*}
Summing up all the inequalities for $T_1,~T_2,~T_3$ and $T_4$, we conclude the proof.
\end{proof}

In the sequel, for each $\epsilon >0$, we define

$$\hat{p}^\epsilon:=\text{number of non-zero eigenvalues of}~\bar{Q}_T~\text{greater or equal to}~\epsilon~a.s$$

\begin{corollary}\label{pepsilon}
Assume that Assumptions in Proposition \ref{fest} hold and let $\mathcal{Q}=Range~Q_T$ with dimension $p$. Let $\epsilon\to 0$ in such a way that $\epsilon^2 \rho(n)^{-1} \to\infty$ as $n\to\infty$. Then, $\mathbb{P}(\hat{p}^\epsilon \neq p)=O(\epsilon^{-2}\rho(n))$.

\end{corollary}
\begin{proof}
From Proposition~\ref{fest}, we know that $\mathbb{E}\|Q_{T}-\hat{Q}_{T}\|^2_{(2)} = O(\rho_n)$. Since we are considering the ordered eigenvalues, $\hat{\theta}_1\ge\hat{\theta}_2\ge \cdots\ge 0$, we have that $\{\hat{p}^{\epsilon}> p\} = \{\hat{\theta}_{p+1}>\epsilon\}$.

A simple calculation on the Hilbert-Schmidt norm together with $\theta_{p+1} = 0~a.s$ yield

$$\hat{\theta}_{p+1} = |\hat{\theta}_{p+1}-\theta_{p+1}| \le \|\hat{Q}_{T}-Q_{T}\|_{(2)}.$$
Therefore,
$$\mathbb{P}(\hat{p}^\epsilon>p)\le \epsilon^{-2}\mathbb{E}\|\hat{Q}_{T}-Q_{T}\|^2_{(2)} = O(\epsilon^{-2}\rho(n)).$$
By noticing that $\{\hat{\theta} < p\} = \{\hat{\theta}_{p-1}< \epsilon\}$ and $\theta_{p-1}> 0~a.s$, we do the same argument to get

$$\mathbb{P}\{ \hat{p} < p  \} = (\theta_{p+1}-\epsilon)^2\mathbb{E}\|\hat{Q}_T - Q_T\|^2_{(2)} = O(\rho(n))$$
Since $\mathbb{P}(\hat{p}^\epsilon \neq p) = \mathbb{P}(\hat{p}^\epsilon>p) +\mathbb{P}(\hat{p}^\epsilon <p)$, the
result follows.
\end{proof}



\begin{thebibliography}{15}
\bibitem{almeida1} Almeida, C. I. R. (2005). Affine Processes, Arbitrage-Free Term Structures of Legendre Polynomials and Option Pricing. \textit{Int. J. Theor. Appl. Finan.}, \textbf{8}, 2, 161-184.

\bibitem{almeida2} Almeida, C. I. R. and Vicente, J. (2008). The Role of No-Arbitrage on Forecasting: Lessons from a Parametric Term Structure Model. \textit{Journal of Banking and Finance}, \textbf{32}, 12, 2695-2705.

\bibitem{alexander} Alexander, C. \textit{Market risk analysis}. Volume II. John Willey e Sons. 2008.

\bibitem{sahalia} Ait-Sahalia,Y. and Xiu, D. Increased correlation among asset classes: Are volatility or jumps to Blame, or both ? Chicago Booth Paper No. 14-11. Preprint.

\bibitem{sahalia2} Ait-Sahalia, Y. and Xiu, D. Principal Component Analysis of High Frequency Data. Chicago Booth Paper No. 15-39. Preprint.


\bibitem{alek} Alekseevsky, D., Kriegl, A., Losik, M. and Michor, P.W. (1998). Choosing roots of polynomials smoothly. \textit{Israel J. Math}, \textbf{105}, 2013-233.

\bibitem{angeline} Angeline, F.and Herzel, S. (2002). Consistent Initial Curves for Interest Rate Models. \textit{J. Derivatives} \textbf{9}, 4, 8-17.

\bibitem{bai} Bai, J. e Ng, S. (2002). Determining the Number of Factors in Approximate Factor Models.
\textit{Econometrica}, 70, 191-221.

\bibitem{bai1} Bai, J. (2003). Inferential Theory for Factor Models of Large Dimensions. \textit{Econometrica}, \textbf{71}, 1, 135-171.


\bibitem{bandorff} Bandorff-Nielsen, O.E. and Sheppard, N. (2004). Econometric analysis of realized covariation:High-frequency based covariance regression, and correlation if financial economics. \textit{Econometrica}, \textbf{72}, 3, 885-925.

\bibitem{bathia} Bathia, N. Yao, Q. and Ziegelmann, F. (2010). Identifying Finite Dimensionality of Curve Time Series. \textit{Ann. Stat},~\textbf{38}, 3352-3386.

\bibitem{teichmann} Baudoin, F. and Teichmann, J. (2005). Hypoellipticity in infinite dimensions and an application in interest rate theory. \textit{Ann. Appl. Probab,} \textbf{15}, 3, 1765-1777.

\bibitem{bjork} Bjork, T. and Christensen, B. J. (1999). Interest rate dynamics and consistent forward rate curves. \textit{Math. Finance}, \textbf{9}, 323-348.

\bibitem{bjork1} Bjork, T. C. Land\'en. (2002). On the construction of finite dimensional realizations for nonlinear forward rate models.~\textit{Finance Stochast,} \textbf{6}, 303-331.

\bibitem{bjork2} Bjork, T. and Svensson, L. (2001). On the existence of finite-dimensional realizations for nonlinear forward rate models.~\textit{Math. Finance,} \textbf{11}, 205-243.

\bibitem{burashi} Burashi, A., Porchia, P. and Trojani, F. (2010). Correlation Risk and Optimal Portfolio Choice. \textit{The J. of Finance}, \textbf{65}, 1, 393-420.

\bibitem{buraschi1} Buraschi,A., Cieslak, A. and Trojani, F. Correlation risk and the term structure
of interest rates, Working Paper, University of St. Gallen


\bibitem{bibinger} Bibinger, M., Hautsch, N., Malec, P. and Rei, M. (2014). Estimating the quadratic covariation matrix from noisy observations: Local method
    of moments and efficiency. \textit{Annals of Stat}, \textbf{42}, 4, 1312-1346.

\bibitem{mikland} Bibinger, M. and Mikland, P. A. Inference for Multi-Dimensional High-Frequency
Data: Equivalence of Methods, Central Limit Theorems, and an Application to Conditional
Independence Testing. arXiv:1301.2074v2.


\bibitem{clement} Clement, E. and Gloter, A. (2011). Limit theorems in the Fourier transform method for the estimation of multivariate volatility. \textit{Stoch. Proc. Appl,} \textbf{121}, 1097-1124.

\bibitem{podolskij} Christensen, K., Podolskij, M. and Vetter, M. (2013): On covariation estimation for multivariate continuous Ito semimartingales with noise in non-synchronous observation schemes. \textit{J. Multivariate Anal}, \textbf{120}, 59-84



\bibitem{duffie} Duffie, D and Khan, R. (1996). A yield-factor model of interest rate. \textit{Math. Finance}, \textbf{6}, 4,379-406.


\bibitem{fan} Fan, J., Li, Y. and Yu, K. (2012). Vast volatility matrix estimation using high-frequency data for
portfolio selection. \textit{J. Amer. Statist. Assoc},\textbf{107}, 412–428.

\bibitem{filipo} Filipovic, D. (2001). Consistency Problems for Heath--Jarrow--Morton Interest Rate Models. \textit{Lecture Notes in Maths} 1760. Springer, Berlin.

\bibitem{filipo1} Filipovic, D. and Teichmann, J. (2003). Existence of invariant manifolds for stochastic equations in infinite dimension. \textit{J. Funct. Anal}. 197, 398-432.

\bibitem{filipo5} Filipovic, D. Teichmann, J. (2004). On the Geometry of the Term structure of Interest Rates. \textit{Proc. R. Soc. Lond. Ser. A Math. Phys. Eng. Sci.}, \textbf{460}, 129-167.

\bibitem{filipo6} Filipovic, D. and Sharef. (2004). Conditions for consistent expoenential-polynomial forward rate processes with multiple nontrivial factors. \textit{Int. J. Theor. Appl. Finan,} \textbf{7}, 6, 685-700.

\bibitem{filipo2} Filipovic, D. (1999). A note on the Nelson Siegel family. \textit{Math. Finance}, \textbf{9}, 349-359.

\bibitem{filipo3} Filipovic, D. (1999). Exponential polynomial families and the term struture of interest rates. \textit{Bernoulii}, \textbf{9}, 349-359.


\bibitem{fissler} Fissler, T. and Podolskij, M. Testing the maximal rank of the volatility process for continuous diffusions observed with noise.  arXiv:1410.6698.


\bibitem{forni} Forni, M., Hallin, M., Lippi, M; and Reichlin, L. (2000). The Generalized Dynamic Factor
Model: Identification and Estimation. \textit{Rev. Econ. Stat}, \textbf{82}, 540-554.





\bibitem{heath} Heath, D., Jarrow, R. and Morton, A. (1992). Bond pricing and the term structure of interest rates: A new metodology for contingent claims valuation. \textit{Econometrica}, \textbf{60}, 77-105.


\bibitem{friz} Friz, P. Victoir, N. \textit{Multidimensional stochastic processes as rough paths. Theory and Applications}. Cambridge University Press. 2011.

\bibitem{harms} Harms, P., Stefanovits, D., Teichmann, J., and Wuthrich, M.
Consistent Recalibration of Yield Curve Models. arXiv:1502.02926.

\bibitem{henseler} Henseler, M., Peters, C. and Seydel, R.C. A Tractable Multi-Factor Dynamic
Term-Structure Model for Risk Management. In: Available at SSRN 2225738
(2013).

\bibitem{heston} Heston, S.L (1993). A closed-form solution for options with stochastic volatility with
  applications to bond and currency options. \textit{Rev. Financ. Stud}, \textbf{6}, 2, 327-343.

\bibitem{jacod} Jacod, J. and Podolskij, M. (2013). A test for the rank of the volatility process: the random perturbation approach. \textit{Ann. Statist}, \textbf{41}, 5, 2391-2427.

\bibitem{laurini} Laurini, M. and Ohashi, A. (2015). A Noisy Principal Component Analysis for Forward Rate Curves. \textit{Eur. J. Oper. Res,} \textbf{246}, 1, 140-153.



\bibitem{mancino} Malliavin, P. and Mancino, M.E. (2009). A Fourier transform method for nonparametric estimation of multivariate volatility. \textit{Ann. Statist.} \textbf{37}, 4, 1983-2010.

\bibitem{mancino1} Malliavin, P; Mancino, M. E.; Recchioni, M. C. (2007). A non-parametric calibration of the HJM geometry: an application of It\^o calculus to financial statistics. \textit{Japan. J. Math,} \textbf{2}, 1, 55-77.








\bibitem{ohashi}Ohashi, A. (2009). Fractional term structure models: no-arbitrage and consistency. \textit{Ann. Appl. Probab,} \textbf{19}, 4, 1553-1580.



\bibitem{pelger} Pelger, M. Large-dimensional factor modeling based on high-frequency observations. \textit{Preprint}.

\bibitem{piazzesi} Piazzesi, M. (2010). \textit{Affine Term-Structure Models}. Handbook of Financial Econometrics. Chapter 12, 691 - 766.

\bibitem{richter} Richter, A. and Teichmann, J. Discrete Time Term Structure Theory and Consistent Recalibration Models.  Preprint arXiv:/1409.1830.


\bibitem{rutter} Rutter, J. W. \textit{Geometry of Curves}, CRC Press, 2009.


\bibitem{slinko} Slinko, I. (2010). On finite-dimensional realizations of two-country interest rate models. \textit{Math. Finance}, \textbf{20}, 1, 117-143.

\bibitem{stock} Stock, J. H. and Watson, M. W. (2002). Forecasting Using Principal Components from a Large Number of Predictors. \textit{J. Amer. Statist. Assoc}, \textbf{97}, 460, 1167-1179.

\bibitem{tappe1} Tappe, S. (2010). An alternative approach on the existence of affine realizations for HJM term-structure models. \textit{Proc. R. Soc. Lond. Ser. A Math. Phys. Eng. Sci,} \textbf{466}, 2122, 3033-3060.

\bibitem{teichmann} Teichmann, J. and Wüthrich, M. Consistent Long-Term Yield Curve Prediction, arXiv/1203.2017, preprint.




\bibitem{wang} Wang, Y. and Zou, J. (2010). Vast volatility matrix estimation for high-frequency financial data. \textit{Ann. Statist.}
\textbf{38}, 2, 943-978.

\bibitem{zeng} Zeng, X. and Li, Y. (2011). On the estimation of integrated covariance matrices of high dimensional diffusion processes. \textit{Ann. Statist.} \textbf{39}, 6, 3121-3151.

\bibitem{tao} Tao, M., Wang, Y.and Zhou, H.H. (2013). Optimal sparse volatility matrix estimation for high-dimensional Itô processes with measurement errors. \textit{Ann. Statist}, \textbf{41}, 4, 1816-1864.

\end{thebibliography}
\end{document}